\documentclass[10pt, reqno]{amsart}
\usepackage{graphicx, amssymb, amsmath, amsthm}
\numberwithin{equation}{section}

\newcommand{\R}{\mathbb{R}}
\newcommand{\C}{\mathbb{C}}
\newcommand{\wh}[1]{\widehat{#1}}
\newcommand{\wt}[1]{\widetilde{#1}}
\newcommand{\wwt}[1]{\tilde{#1}}

\renewcommand{\Re}{\text{Re}}
\renewcommand{\Im}{\text{Im}}
\newcommand{\norm}[1]{\big\| #1\big\|}

\newcommand{\eps}{\varepsilon}
\newcommand{\p}{\mathcal{P}}
\newcommand{\pb}[2]{\{#1,#2\}_{\mathcal{P}}}
\newcommand{\mb}[2]{\{#1,#2\}_m}
\newcommand{\psih}{\check{\psi}}
\newcommand{\psith}{\check{\psi_0}}
\def\smallint{\begingroup\textstyle \int\endgroup}

\newcommand{\norml}[2]{\big\|#1\big\|_{L_x^#2}}
\newcommand{\xonorm}[3]{\big\|#1\big\|_{L_t^{#2}L_x^{#3}}}
\newcommand{\xnorm}[4]{\big\|#1\big\|_{L_t^{#2}L_x^{#3}(#4\times\R^d)}}
\newcommand{\xonorms}[2]{\big\|#1\big\|_{L_{t,x}^{#2}}}
\newcommand{\bxonorms}[2]{\bigg\|#1\bigg\|_{L_{t,x}^{#2}}}

\newcommand{\xnorms}[3]{\big\|#1\big\|_{L_{t,x}^{#2}(#3\times\R^d)}}

\newcommand{\jap}[1]{\langle#1\rangle}
\newcommand{\N}{\mathcal{N}}

\newcommand{\M}{M}

\newcommand{\nsc}{\vert\nabla\vert^{s_c}}
\newcommand{\ntw}[1]{\vert\nabla\vert^{#1}}

\newtheorem{theorem}{Theorem}[section]

\newtheorem{lemma}[theorem]{Lemma}

\newtheorem{corollary}[theorem]{Corollary}
\newtheorem{conjecture}[theorem]{Conjecture}
\newtheorem{proposition}[theorem]{Proposition}

\theoremstyle{definition}
\newtheorem{definition}[theorem]{Definition}
\newtheorem{remark}[theorem]{Remark}

\theoremstyle{remark}

\begin{document}
\title[Inter-critical NLS: critical $\dot{H}^s$-bounds imply scattering]{Inter-critical NLS: \\ critical $\dot{H}^{s}$-bounds imply scattering}
\author{Jason Murphy}
\begin{abstract} We consider a class of power-type nonlinear Schr\"odinger equations for which the power of the nonlinearity lies between the mass- and energy-critical exponents. Following the concentration-compactness approach, we prove that if a solution $u$ is bounded in the critical Sobolev space throughout its lifespan, that is, $u\in L_t^\infty \dot{H}_x^{s_c}$, then $u$ is global and scatters. 
\end{abstract}
\address{Department of Mathematics, University of California at Los Angeles,
               Los Angeles, CA 90095-1555, USA}
        \email{murphy@math.ucla.edu}
\maketitle
\tableofcontents
\section{Introduction}
\noindent We consider the initial-value problem for defocusing nonlinear Schr\"odinger equations of the form 

\begin{equation}\label{nls} \left\{ \begin{array}{ll} (i\partial_t+\Delta) u=\vert u\vert^p u  \\ u(0,x)=u_0(x), \end{array}\right.
\end{equation} 

\noindent where $p> 0$ is chosen to lie between the mass- and energy-critical exponents, that is, $\tfrac{4}{d}<p<\tfrac{4}{d-2}$. Here $u:\R_t\times\R_x^d\to\C$ is a complex-valued function of time and space.

The class of solutions to \eqref{nls} is left invariant by the scaling \begin{equation}\nonumber u(t,x)\mapsto\lambda^{\frac{2}{p}}u(\lambda^2t,\lambda x)\end{equation} for $\lambda>0$. This scaling defines a notion of \emph{criticality}. In particular, one can check that the only homogeneous $L_x^2$-based Sobolev space that is left invariant by this scaling is $\dot{H}_x^{s_c}(\R^d)$, where the \emph{critical regularity} $s_c$ is given by $s_c:=\tfrac{d}{2}-\tfrac{2}{p}.$ If we take $u_0\in\dot{H}_x^s(\R^d)$ in \eqref{nls}, then for $s=s_c$, we call the problem \emph{critical}. For $s>s_c$, we call the problem \emph{subcritical}, while for $s<s_c$ the problem is \emph{supercritical}.

We consider the critical problem for \eqref{nls} in the \emph{inter-critical} regime, that is, $0<s_c<1$. For $(d,s_c)$ satisfying an appropriate set of constraints, we prove that any maximal-lifespan solution that is uniformly bounded (throughout its lifespan) in $\dot{H}_x^{s_c}(\R^d)$ must be global and scatter. 

We begin by making the notion of a solution more precise.
\begin{definition}[Solution]\label{solution} A function $u:I\times\R^d\to\mathbb{C}$ on a non-empty time interval $I\ni 0$ is a \emph{solution} to \eqref{nls} if it belongs to $C_t\dot{H}_x^{s_c}(K\times\R^d)\cap L_{t,x}^{p(d+2)/2}(K\times\R^d)$ for every compact $K\subset I$ and obeys the Duhamel formula 
\begin{equation}\nonumber u(t)=e^{it\Delta}u_0-i\int_0^t e^{i(t-s)\Delta}(\vert u\vert^p u)(s)\, ds
\end{equation}
for all $t\in I$. We call $I$ the \emph{lifespan} of $u$; we say $u$ is a \emph{maximal-lifespan solution} if it cannot be extended to any strictly larger interval. If $I=\R$, we say $u$ is \emph{global}.
\end{definition}

We define the \emph{scattering size} of a solution $u$ to \eqref{nls} on a time interval $I$ by 
\begin{equation}\label{scattering size} S_I(u):=\int_I\int_{\R^d}\vert u(t,x)\vert^{\frac{p(d+2)}{2}}\, dx\, dt.
\end{equation} 
Standard arguments show that if a solution $u$ to \eqref{nls} is global, with $S_{\R}(u)<\infty$, then it \emph{scatters}; that is, there exist unique $u_\pm\in\dot{H}_x^{s_c}(\R^d)$ such that $$\lim_{t\to\pm\infty}\norm{u(t)-e^{it\Delta}u_{\pm}}_{\dot{H}_x^{s_c}(\R^d)}=0.$$

The goal of this paper is to address some cases of the following
\begin{conjecture}\label{the conjecture}
Let $d\geq1$ and $p>0$ such that $s_c:=\tfrac{d}{2}-\tfrac{2}{p}\geq 0$. Let $u:I\times\R^d\to\C$ be a maximal-lifespan solution to \eqref{nls} such that $u\in L_t^\infty \dot{H}_x^{s_c}(I\times\R^d)$. Then $u$ is global and scatters, with 
\begin{equation}S_\R(u)\leq C(\norm{u}_{L_t^\infty \dot{H}_x^{s_c}(\R\times\R^d)})\nonumber\end{equation} 
for some function $C:[0,\infty)\to [0,\infty).$ 
\end{conjecture}
For two special cases, it is unnecessary to take $u\in L_t^\infty \dot{H}_x^{s_c}$ as an additional hypothesis, as this bound follows from conservation laws. In particular, in the mass-critical case, $s_c=0$ (i.e. $p=\tfrac{4}{d}$), the fact that $u\in L_t^\infty L_x^2$ follows from the conservation of \emph{mass}, defined by $$M[u(t)]:=\int_{\R^d}\vert u(t,x)\vert^2\,dx,$$
while in the energy-critical case, $s_c=1$ (i.e. $p=\tfrac{4}{d-2}$, $d\geq 3$), the fact that $u\in L_t^\infty \dot{H}_x^1$ follows from the conservation of \emph{energy}, defined by $$E[u(t)]:=\int_{\R^d}\tfrac12\vert\nabla u(t,x)\vert^2+\tfrac{1}{p+2}\vert u(t,x)\vert^{p+2}\,dx.$$

Due to the presence of conserved quantities at critical regularity, the mass- and energy-critical equations have been the most widely studied; in fact, Conjecture \ref{the conjecture} has been settled in these cases. For $s_c\notin\{0,1\}$, the hypothesis that $u$ stays bounded in $\dot{H}_x^{s_c}$ in Conjecture \ref{the conjecture} plays the role of the `missing conservation law' at critical regularity; the aim of this paper is to show that with this extra assumption, the techniques developed to treat the mass- and energy-critical NLS can be applied in the inter-critical regime, $0<s_c<1.$ 

The defocusing energy-critical case was handled first by Bourgain \cite{bourgain}, Grillakis \cite{grillakis}, and Tao \cite{Tao:gwp radial} for radial data, and subsequently by Colliander--Keel--Staffilani--Takaoka--Tao \cite{CKSTT:gwp}, Ryckman--Vi\c{s}an \cite{RV}, and Vi\c{s}an \cite{Monica:thesis art, Monica:thesis} for arbitrary data. (See also \cite{kenig merle, KV5} for results in the focusing case.) The primary obstacle to establishing these results was the lack of any a priori estimates with critical scaling (besides the conservation of energy); that is to say, none of the known monotonicity formulae for NLS (i.e. Morawetz inequalities) scale like the energy (in contrast to the energy-critical nonlinear wave equation, for example). It was ultimately the `induction on energy' technique of Bourgain that showed how one can move beyond this difficulty: by finding a bubble of concentration inside a solution, one introduces a characteristic length scale into a scale-invariant problem. Having `broken' the scaling symmetry in this way, the available Morawetz inequalities come back into play, despite their non-critical scaling. All subsequent work for NLS at critical regularity has built upon this fundamental idea.

In the energy-critical case, the critical regularity associated to the available Morawetz estimates is lower than the critical regularity of the problem. Bourgain was able to make use of the Lin--Strauss Morawetz inequality (appearing first in \cite{LS}), which scales like $\dot{H}_x^{\frac12}$ and is well-suited to the radial case; to remove the radial assumption, Colliander--Keel--Staffilani--Takaoka--Tao introduced the interaction Morawetz inequality (see \cite{CKSTT}), which has the scaling of $\dot{H}_x^{\frac14}$ (but still requires control over at least half of a derivative, cf. \eqref{int mor} below). These considerations lead us to believe that the techniques developed to handle the energy-critical problem should be applicable to resolve Conjecture \ref{the conjecture} in the case $s_c\geq\tfrac12$. 

In particular, we will be making use of the concentration-compactness (or `minimal counterexample') approach to induction on energy. Minimal counterexamples were introduced over the course of several papers in the context of the mass-critical problem (see, for example, \cite{begout-vargas, bourg.2d, carles-keraani, keraani, Keraani:L2, merle-vega}); however, the first application of minimal counterexamples to establish a global well-posedness result was carried out by Kenig--Merle \cite{kenig merle}, who developed the technique in the focusing, energy-critical setting.

Minimal counterexamples were also used to establish Conjecture \ref{the conjecture} in the mass-critical setting, first for spherically-symmetric data in dimensions $d\geq 2$ (see \cite{KTV, KVZ, TVZ:sloth}), and subsequently for arbitrary initial data in all dimensions by Dodson \cite{dodson:3, dodson:2, dodson:1}. (For results in the focusing case, see \cite{dodson:f, KTV, KVZ}.) Notice that in the mass-critical case, the critical regularity of the problem is lower than that of the available Morawetz estimates; thus one needs to prove additional regularity (instead of decay) to access these estimates. We pause here to point out \cite{revisit2, revisit1}, as well, which revisit the defocusing energy-critical problem from the perspective of minimal counterexamples. 

The first case of Conjecture \ref{the conjecture} at non-conserved critical regularity to be addressed was the case $d=3$ and $s_c=\tfrac12$, in which case the nonlinearity is cubic. Kenig--Merle \cite{KM} were able to handle this case by using their concentration-compactness technique (as in \cite{kenig merle}), together with the Lin--Strauss Morawetz inequality (which is scaling-critical in this case). As we will see, this case also falls into the range of cases that we consider, although we will opt to use the interaction Morawetz inequality instead.

Some cases of Conjecture \ref{the conjecture} in the energy-supercritical regime (i.e. $s_c>1$) have also been handled by Killip--Vi\c{s}an \cite{KV:supercritical}, also through the use of minimal counterexamples. In particular, they deal with the case of a cubic nonlinearity for $d\geq 5$, along with some other cases for which $s_c>1$ and $d\geq 5$. Their restriction to high dimensions stems ultimately from their use of the so-called `double Duhamel trick'; for more details, see \cite{KV:supercritical} and the references cited therein.  

Before we discuss our contribution, we note that the analogous conjecture has also been studied for the nonlinear wave equation. For progress in the energy-supercritical case, one can refer to works of Kenig--Merle \cite{KM NLW}, Killip--Vi\c{s}an \cite{KV NLW, KV NLW2}, and Bulut \cite{bulut:1, bulut:2, bulut:3}. For some results in the energy-subcritical case with radial data, see \cite{shen1, shen2}. 

Finally, we are in a position to describe the cases of Conjecture \ref{the conjecture} that we will address in this paper. As mentioned above, we will work in the inter-critical regime, $0<s_c<1$ (that is, $\tfrac{4}{d}<p<\tfrac{4}{d-2}$). Our primary restriction is technical; namely, we only consider cases for which $p\geq 1$. This restriction serves to simplify the analysis, which still becomes a bit complicated. For example, when we need to estimate a fractional number of derivatives of the function $G(z)=\vert z\vert^p$, things are quite a bit simpler when $G$ is locally Lipschitz, rather than merely H\"older continuous.

Next, in order to make use of the interaction Morawetz inequality, we restrict to the cases $d\geq 3$ and $s_c\geq\tfrac12$ (cf. \eqref{int mor} below). For this restriction to be compatible with $p\geq 1$, we must then restrict to $d\leq 5$. The use of the interaction Morawetz inequality ultimately leads to a further (more severe) restriction: when $d=3$, we must exclude the cases $\tfrac{3}{4}<s_c<1$. Let us briefly describe the reason for this additional restriction.

The standard interaction Morawetz inequality may be written as follows: for $u$ solving \eqref{nls}, we have
				\begin{align}
				-\int_I\iint_{\R^d\times\R^d}& 
				\vert u(t,x)\vert^2\Delta(\tfrac{1}{\vert \cdot\vert})(x-y)\vert u(t,y)\vert^2\,dx\,dy\,dt\nonumber
				\\ &\lesssim\xnorm{u}{\infty}{2}{I}^2\xnorm{\vert\nabla\vert^{1/2}u}{\infty}{2}{I}^2.
				\label{int mor}\end{align}
As we are in the case $s_c\geq \tfrac{1}{2}$, we see that to guarantee that the right-hand side is finite, we must truncate the solution $u$ to high frequencies (that is, we must work with $u_{\geq N}$ for some $N>0$). In Section \ref{flim section}, we do exactly this. However, $u_{\geq N}$ is no longer a solution to \eqref{nls}; thus, the truncation results in error terms for \eqref{int mor} that must be handled to arrive at a useful estimate. 
When $d=3$ and $s_c>\tfrac{3}{4}$, we find that there is one error term that we cannot handle unless we also impose a spatial truncation on the weight we use to derive \eqref{int mor}; see, for example, \cite{CKSTT:gwp, revisit2}, which address the case $d=3$, $s_c=1$. This spatial truncation, however, results in even \emph{more} error terms; it turns out that some of these additional error terms then require control over the $\dot{H}_x^1$-regularity of the solution. Thus, in the energy-critical case, one can push the argument through, while in our case we cannot succeed without significant additional input. For further comments, see Remark~\ref{exclusion}.

Our final restriction is to exclude the case $(d,s_c)=(5,\tfrac12)$, for which $p=1$. In this case, the proof of Lemma \ref{decouple lemma}, which is essential to the reduction to almost periodic solutions in Section \ref{reduction section}, breaks down. For further discussion, see Remark~\ref{goodbye p=1}.

To summarize, our main results will apply to \eqref{nls} with $(d,s_c)$ satisfying the following: 
\begin{equation}\label{constraints}
\left\{\begin{array}{llllll}
			\tfrac12\leq s_c\leq\tfrac34&\text{if }d=3,
\vspace{1mm}
			\\ \tfrac12\leq s_c<1&\text{if }d=4,
\vspace{1mm}
				\\ \tfrac12<s_c<1&\text{if }d=5.
\end{array}\right.
\end{equation}
It is worth mentioning, however, that many of our results will apply to a less restrictive set of $(d,s_c)$. In particular, many of our results will apply to
\begin{equation}\label{constraints2}
d\in\{3,4,5\}\quad\text{and}\quad \tfrac12\leq s_c<1,\quad\text{excluding}\quad(d,s_c)=(5,\tfrac12).
\end{equation}
As we proceed, we will keep track of which restrictions are necessary for which results.

Our main result is the following:

\begin{theorem}\label{swamprat} Suppose $(d,s_c)$ satisfies \eqref{constraints}. Let $u:I\times\R^d\to\C$ be a maximal-lifespan solution to \eqref{nls} such that $u\in L_t^\infty \dot{H}_x^{s_c}(I\times\R^d).$ Then $u$ is global and scatters, with 
			$$S_\R(u)\leq C(\norm{u}_{L_t^\infty\dot{H}_x^{s_c}(\R\times\R^d)})$$
for some function $C:[0,\infty)\to[0,\infty).$  
\end{theorem}

To establish Theorem \ref{swamprat}, we will model our approach after several sources, including \cite{KV, KV5, KV:supercritical, revisit2, revisit1}. In particular, Section \ref{local theory} follows the presentation in \cite{KV, KV:supercritical}; Section \ref{reduction section} draws heavily from \cite{KV5}; and the presentation of the remaining sections is inspired largely by \cite{revisit2, revisit1}. We note also that we rely on \cite{KV} for several standard results regarding almost periodic solutions in the outline below. 

The first step towards a global-in-time theory for \eqref{nls} is to develop a good local-in-time theory for this equation. In particular, building off arguments of Cazenave--Weissler \cite{cw}, one can prove the following
\begin{theorem}[Local well-posedness]\label{local theorem} Let $(d,s_c)$ satisfy \eqref{constraints2}. Then, given $u_0\in\dot{H}_x^{s_c}(\R^d)$, there exists a unique maximal-lifespan solution $u:I\times\R^d\to\C$ to \eqref{nls}. Moreover, this solution satisfies the following:

$\bullet$ (Local existence) $I$ is an open neighborhood of $0$.

$\bullet$ (Blowup criterion) If $\sup I$ is finite, then $u$ blows up forward in time, in the sense that $S_{[0,\sup I)}(u)=\infty$. If $\inf I$ is finite, then $u$ blows up backwards in time, in the sense that $S_{(\inf I,0]}(u)=\infty.$ 

$\bullet$ (Scattering) If $\sup I=+\infty$ and $u$ does not blow up forward in time, then $u$ scatters forward in time; that is, there exists a unique $u_+\in \dot{H}_x^{s_c}(\R^d)$ such that 
\begin{equation}\label{scattering} \lim_{t\to\infty}\norm{u(t)-e^{it\Delta}u_+}_{\dot{H}_x^{s_c}(\R^d)}=0.
\end{equation} Conversely, for any $u_+\in\dot{H}_x^{s_c}(\R^d)$, there is a unique solution $u$ to \eqref{nls} so that \eqref{scattering} holds. The analogous statements hold backward in time.

$\bullet$ (Small-data global existence) There exists $\eta_0=\eta_0(d,p)$ such that if $$\norm{u_0}_{\dot{H}_x^{s_c}(\R^d)}^2<\eta_0,$$ then $u$ is global and scatters, with $S_\R(u)\lesssim\norm{u_0}_{\dot{H}_x^{s_c}(\R^d)}^{p(d+2)/2}.$ 
\end{theorem}
\begin{remark}
We note here that the notion of blowup described above corresponds exactly to the impossibility of extending the solution to a larger time interval in the class described in Definition \ref{solution}.
\end{remark}
In Section \ref{local theory}, we will establish this theorem as a corollary of a local well-posedness result of Cazenave--Weissler \cite{cw} (Theorem \ref{standard lwp}) and a stability result (Theorem \ref{stability}). This stability result plays an essential role in the argument we present, specifically in the proof of Theorem \ref{reduction}. 

\subsection{Outline of the proof of Theorem \ref{swamprat}}\label{outline} The proof is by contradiction. We first recall that Theorem \ref{swamprat} holds if we restrict to sufficiently small initial data (cf. Theorem \ref{local theorem}); thus, the failure of Theorem \ref{swamprat} would imply the existence of a `threshold' size, below which Theorem \ref{swamprat} holds, but above which we can find (almost) counterexamples. Using a limiting argument, we then find blowup solutions \emph{at} this threshold, so-called `minimal blowup solutions'. By carefully analyzing such solutions, we can show that they must have so many nice properties that in fact, they cannot exist at all.

The main property of these special counterexamples is that of almost periodicity modulo symmetries:

\begin{definition}[Almost periodic solutions]\label{almost periodic definition} Let $s_c>0$. A solution $u$ to \eqref{nls} with lifespan $I$ is said to be \emph{almost periodic (modulo symmetries)} if $u\in L_t^\infty\dot{H}_x^{s_c}(I\times\R^d)$ and there exist functions $N:I\to\R^+,$ $x:I\to\R^d$, and $C:\R^+\to \R^+$ such that for all $t\in I$ and all $\eta>0$,
			$$\int_{\vert x-x(t)\vert\geq \frac{C(\eta)}{N(t)}}\big\vert\nsc u(t,x)\big\vert^2\, dx
			+\int_{\vert\xi\vert\geq C(\eta)N(t)}\vert\xi\vert^{2s_c}\vert\wh{u}(t,\xi)\vert^2\, d\xi\leq\eta.$$
We call $N$ the \emph{frequency scale function}, $x$ the \emph{spatial center function}, and $C$ the \emph{compactness modulus function}. 
\end{definition} 
\begin{remark}\label{arzela ascoli} By the Arzel\`a--Ascoli theorem, a family of functions is precompact in $\dot{H}_x^{s_c}(\R^d)$ if and only if it is norm-bounded and there exists a compactness modulus function $C$ such that 
				$$\int_{\vert x\vert\geq C(\eta)}\big\vert\nsc f(x)\big\vert^2\,dx
				+\int_{\vert \xi\vert\geq C(\eta)}\vert\xi\vert^{2s_c}\vert\wh{f}(\xi)\vert^2\,d\xi\leq \eta$$ 
uniformly for all functions $f$ in the family. Thus, an equivalent formulation of Definition \ref{almost periodic definition} is the following: $u$ is almost periodic (modulo symmetries) if and only if
			$$\{u(t):t\in I\}\subset\{\lambda^{\frac{2}{p}}f(\lambda(x+x_0)):
			\lambda\in (0,\infty),\ x_0\in\R^d,\ \text{and}\ f\in K\}$$ 
for some compact $K\subset\dot{H}_x^{s_c}(\R^d).$ 

Furthermore, Sobolev embedding gives that every compact set in $\dot{H}_x^{s_c}(\R^d)$ is also compact in $L_x^{\frac{dp}{2}}(\R^d)$; thus, for any almost periodic solution $u:I\times\R^d\to\C$, we also have 
			$$\int_{\vert x-x(t)\vert\geq \frac{C(\eta)}{N(t)}}\vert u(t,x)\vert^{\frac{dp}{2}}\,dx\leq\eta$$ 
for all $t\in I$ and $\eta>0$.  
\end{remark}

\begin{remark}\label{another consequence} Another consequence of almost periodicity modulo symmetries is the existence of a function $c:\R^+\to\R^+$ so that for all $t\in I$ and all $\eta>0$, 
			$$\int_{\vert x-x(t)\vert\leq \frac{c(\eta)}{N(t)}}\big\vert\nsc u(t,x)\big\vert^2\, dx
			+\int_{\vert\xi\vert\leq c(\eta)N(t)}\vert\xi\vert^{2s_c}\vert\wh{u}(t,\xi)\vert^2\, d\xi\leq\eta.$$
\end{remark}

One can show (see \cite[Lemma 5.18]{KV}, for example) that the modulation parameters of almost periodic solutions obey the following local constancy property: 
\begin{lemma}[Local constancy]\label{local constancy} Let $u:I\times\R^d\to\C$ be a maximal-lifespan almost periodic solution to \eqref{nls}. Then there exists $\delta=\delta(u)>0$ such that if $t_0\in I$, then 
				$$[t_0-\delta N(t_0)^{-2},t_0+\delta N(t_0)^{-2}]\subset I,$$ 
with 
				$$N(t)\sim_u N(t_0)\indent\text{for}\quad\vert t-t_0\vert\leq\delta N(t_0)^{-2}.$$ 
\end{lemma}
Given a maximal-lifespan almost periodic solution $u:I\times\R^d\to \C$ to \eqref{nls}, Lemma~\ref{local constancy} allows us to subdivide $I$ into \emph{characteristic subintervals} $J_k$ on which $N(t)$ is constant and equal to some $N_k$, with $\vert J_k\vert \sim_u N_k^{-2}$. To do this, we need to modify the compactness modulus function by a time-independent multiplicative factor.

The local constancy property also has the following consequence (see \cite[Corollary~5.19]{KV}):
\begin{corollary}[$N(t)$ at blowup]\label{N at blowup} Let $u:I\times\R^d\to\C$ be a maximal-lifespan almost periodic solution to \eqref{nls}. If $T$ is any finite endpoint of $I$, then $N(t)\gtrsim_u\vert T-t\vert^{-1/2}$; in particular, $\lim_{t\to T}N(t)=\infty.$ If $I$ is infinite or semi-infinite, then for any $t_0\in I$, we have $N(t)\gtrsim\jap{t-t_0}^{-1/2}.$ 
\end{corollary}

Finally, we need the following result, which relates the frequency scale function of an almost periodic solution to its Strichartz norms. 
\begin{lemma}[Spacetime bounds]\label{spacetime bounds} Let $(d,s_c)$ satisfy \eqref{constraints2}, and suppose $u$ is an almost periodic solution to \eqref{nls} on a time interval $I$. Then 
$$\int_I N(t)^2\, dt\lesssim_u \xnorm{\nsc u}{2}{\frac{2d}{d-2}}{I}^2\lesssim_u 1+\int_I N(t)^2\, dt.$$ 
\end{lemma}
One can prove this result by adapting the proof of \cite[Lemma 5.21]{KV}; the key is to notice that $\int_I N(t)^2\,dt$ is approximately the number of characteristic subintervals inside $I$. The restriction on $(d,s_c)$ is not actually necessary, but it covers our cases of interest. 

With these preliminaries established, we can now describe the first major step in the proof of Theorem \ref{swamprat}. 

\begin{theorem}[Reduction to almost periodic solutions]\label{reduction} Suppose that Theorem~\ref{swamprat} fails for $(d,s_c)$ satisfying \eqref{constraints2}. Then there exists a maximal-lifespan solution $u:I\times\R^d\to\C$ to \eqref{nls} such that $u\in L_t^\infty\dot{H}_x^{s_c}(I\times\R^d)$, $u$ is almost periodic modulo symmetries, and $u$ blows up both forward and backward in time. Moreover, $u$ has minimal $L_t^\infty \dot{H}_x^{s_c}$-norm among all blowup solutions; i.e., $$\sup_{t\in I}\norm{u(t)}_{\dot{H}_x^{s_c}(\R^d)}\leq\sup_{t\in J}\norm{v(t)}_{\dot{H}_x^{s_c}(\R^d)}$$ for all maximal-lifespan solutions $v:J\times\R^d\to\C$ that blow up in at least one time direction. 
\end{theorem}
We sketch a proof of Theorem \ref{reduction} in Section \ref{reduction section}. By now, the reduction to almost periodic solutions is a fairly standard technique in the study of dispersive equations at critical regularity. Keraani \cite{Keraani:L2} first proved the existence of minimal blowup solutions (in the mass-critical setting), while Kenig--Merle \cite{kenig merle} were the first to use them as a tool to prove global well-posedness (in the energy-critical setting). For many more examples of these techniques, one can refer to \cite{KM, KM NLW, KTV, KV5, KV, KV:supercritical, KV NLW, KV NLW2, KVZ, TVZ, TVZ:sloth}, for example.

Still, while the underlying ideas are well-established, we will see that to carry out the reduction in the cases we consider will require some new ideas and careful analysis; indeed, the resolution of this problem is the chief novelty of this paper. One of the principal difficulties arises in the proof of Lemma \ref{decouple lemma}, in which we establish a decoupling of nonlinear profiles in order to show that a sequence of approximate solutions to \eqref{nls} converges in some sense to a true solution. 

For the mass- and energy-critical cases, one can use pointwise estimates and well-known arguments of \cite{keraani} to establish this decoupling; in our setting, the nonlocal nature of $\nsc$ prevents the direct use of any pointwise estimates. The authors of \cite{KV:supercritical}, who dealt with some cases in the energy-supercritical regime, overcame this difficulty by establishing analogous pointwise estimates for a square function of Strichartz that shares estimates with $\nsc$ (see \cite{strichartz square}). With such estimates in hand, the usual arguments can then be pushed through. Their approach does not work in our setting, however, as it strongly relies on the fact that $s_c>1$. In \cite{KM}, the authors treat a cubic nonlinearity in dimension $d=3$ (in which case $s_c=\tfrac12$); by exploiting the polynomial nature of the nonlinearity and employing a paraproduct estimate, they too overcome the nonlocal nature of $\nsc$ and put themselves in a position where the standard arguments are applicable. In our setting, the combination of fractional derivatives and non-polynomial nonlinearities presents a new technical challenge. Ultimately, the resolution of this problem comes from a careful reworking of the proof of the fractional chain rule, in which the Littlewood--Paley square function allows us to work at the level of individual frequencies. By making use of some tools from harmonic analysis (including maximal and vector maximal inequalities), we are eventually able to adapt the standard arguments to establish the decoupling in our setting. For further discussion, see Section \ref{reduction section}.

After establishing Theorem \ref{reduction}, we make some further refinements to the class of solutions we consider. First, we can use a rescaling argument to restrict our attention to almost periodic solutions that do not escape to arbitrarily low frequencies on at least half of their maximal lifespan, say $[0,T_{max})$. We will not include the details here; one can instead refer to Section 4 in any of \cite{KTV, KV5, TVZ:sloth}. Next, following \cite{dodson:3}, we will divide these solutions into two classes that depend on the control given by the interaction Morawetz inequality; these classes will correspond to the `rapid frequency-cascade' and `quasi-soliton' scenarios. Finally, as described above, we use Lemma~\ref{local constancy} to subdivide $[0,T_{max})$ into characteristic subintervals $J_k$ and set $N(t)$ to be constant and equal to $N_k$ on each $J_k$, with $\vert J_k\vert\sim_u N_k^{-2}$. In this way, we arrive at

\begin{theorem}[Two special scenarios for blowup]\label{special scenarios} Suppose that Theorem \ref{swamprat} fails for $(d,s_c)$ satisfying \eqref{constraints2}. Then there exists an almost periodic solution $u:[0,T_{max})\times\R^d\to\C$ that blows up forward in time, with $$N(t)\equiv N_k\geq 1$$ for each $t\in J_k$, where $[0,T_{max})=\cup_k J_k$. Furthermore,
$$\text{either}\quad\int_0^{T_{max}}N(t)^{3-4s_c}\, dt<\infty\quad\text{or}\quad\int_0^{T_{max}} N(t)^{3-4s_c}\, dt=\infty.$$
\end{theorem}

Thus, to establish Theorem \ref{swamprat}, it remains to preclude the existence of the almost periodic solutions described in Theorem \ref{special scenarios}. The main technical tool we will use to achieve this will be a long-time Strichartz inequality, Proposition \ref{lts lemma}. Such inequalities were originally developed by Dodson \cite{dodson:3} for almost periodic solutions in the mass-critical setting; for variants in the energy-critical setting, see \cite{revisit2, revisit1}. In this paper, we establish a long-time Strichartz estimate for the first time in the inter-critical regime. The proof of Proposition \ref{lts lemma} is by induction; the recurrence relation is derived with the aid of Strichartz estimates, together with a paraproduct estimate (Lemma \ref{paraproduct}) and a bilinear Strichartz inequality (Lemma \ref{bilinear strichartz}).

In Section \ref{frequency-cascade section}, we preclude the rapid frequency-cascade scenario, that is, almost periodic solutions as in Theorem \ref{special scenarios} for which $\smallint_0^{T_{max}}N(t)^{3-4s_c}\,dt<\infty.$ This proof requires two ingredients. The first ingredient is the long-time Strichartz inequality (Proposition \ref{lts lemma}), while the second is  the following

\begin{proposition}[No-waste Duhamel formula]\label{no waste} Let $u:[0,T_{max})\times\R^d\to\mathbb{C}$ be an almost periodic solution to \eqref{nls} with $N(t)\equiv N_k\geq 1$ on each characteristic subinterval $J_k$. Then for all $t\in[0,T_{max}),$ we have
$$u(t)=i\lim_{T\to T_{max}}\int_t^T e^{i(t-s)\Delta}(\vert u\vert^p u)(s)\, ds,$$ where the limits are taken in the weak $\dot{H}_x^{s_c}$ topology. 
\end{proposition}
To prove Proposition \ref{no waste}, one can adapt the proof of \cite[Proposition 5.23]{KV}; we omit the details. Using Proposition \ref{no waste} and Strichartz estimates, we can upgrade the information given by Proposition \ref{lts lemma} to see that a rapid frequency-cascade solution must have finite mass. In fact, we can show that the solution has zero mass, which contradicts that the solution blows up.

In Section \ref{Quasi}, we preclude the quasi-soliton scenario, that is, almost periodic solutions as in Theorem \ref{special scenarios} for which $\smallint_0^{T_{max}}N(t)^{3-4s_c}\,dt=\infty.$ The main ingredient is a frequency-localized interaction Morawetz inequality (Proposition \ref{flim}), which we prove in Section \ref{flim section}. To establish this estimate, we begin with the usual interaction Morawetz inequality, truncate to high frequencies, and use Proposition \ref{lts lemma} to control the resulting error terms. (As described above, one of these error terms eventually forces us to exclude the cases $(d,s_c)\in\{3\}\times(\tfrac34,1)$ from Theorem \ref{swamprat}; see Remark~\ref{exclusion}.) To rule out the quasi-soliton scenario and thereby complete the proof of Theorem \ref{swamprat}, we notice that the frequency-localized interaction Morawetz inequality provides uniform control over $\int_I N(t)^{3-4s_c}\,dt$ for all compact time intervals $I\subset[0,T_{max})$; thus we can derive a contradiction by taking $I$ to be sufficiently large inside of $[0,T_{max})$. 

\subsection*{Acknowledgements} I owe many thanks to my advisors, Rowan Killip and Monica Vi\c{s}an, for all of their guidance and support. I am very grateful to them, not only for bringing this problem to my attention, but also for engaging in many helpful discussions, and for a careful reading of the manuscript. This work was supported in part by NSF grant DMS-1001531 (P.I. Rowan Killip). 

\section{Notation and useful lemmas}\label{lemmata}
\subsection{Some notation} We write $X\lesssim Y$ or $Y\gtrsim X$ whenever $X\leq CY$ for some $C>0$. If $X\lesssim Y\lesssim X$, we write $X\sim Y$. If the implicit constant $C$ depends on the dimension $d$ or the power $p$, this dependence will be suppressed; dependence on additional parameters will be indicated with subscripts. For example, $X\lesssim_u Y$ indicates that $X\leq CY$ for some $C=C(u).$ 

For a spacetime slab $I\times\R^d$, we write $L_t^qL_x^r(I\times\R^d)$ to denote the Banach space of functions $u:I\times\R^d\to\C$ equipped with the norm
				$$\xnorm{u}{q}{r}{I}:=\left(\int_I \norm{u(t)}_{L_x^r(\R^d)}^q\,dt\right)^{1/q},$$ 
with the usual conventions when $q$ or $r$ is infinity. If $q=r$, we abbreviate $L_t^qL_x^q=L_{t,x}^q.$ At times we will also abbreviate $\norm{f}_{L_x^r(\R^d)}$ to $\norm{f}_{L_x^r}$ or $\norm{f}_r.$ 

We define the Fourier transform on $\R^d$ by $$\wh{f}(\xi):=(2\pi)^{-d/2}\int_{\R^d} e^{-ix\cdot\xi}f(x)\,dx.$$ 
For $s\in \R$, we can then define the fractional differentiation operator $\vert\nabla\vert^s$ via $$\wh{\vert\nabla\vert^s f}(\xi):=\vert\xi\vert^s\wh{f}(\xi),$$ which in turn defines the homogeneous Sobolev norm $$\norm{f}_{\dot{H}_x^s(\R^d)}:=\norm{\vert\nabla\vert^s f}_{L_x^2(\R^d)}.$$ 
\subsection{Basic harmonic analysis} Let $\varphi$ be a radial bump function supported in the ball $\{\xi\in\R^d:\vert \xi\vert\leq\tfrac{11}{10}\}$ and equal to 1 on the ball $\{\xi\in\R^d:\vert\xi\vert\leq 1\}.$ For $N\in 2^{\mathbb{Z}}$, we define the Littlewood--Paley projection operators via
\begin{align*}
&\wh{P_{\leq N} f}(\xi):=\wh{f_{\leq N}}(\xi):=\varphi(\xi/N)\wh{f}(\xi),
\\ &\wh{P_{> N}f}(\xi):=\wh{f_{> N}}(\xi):=(1-\varphi(\xi/N))\wh{f}(\xi),
\\ &\wh{P_{N}f}(\xi):=\wh{f_N}(\xi):=(\varphi(\xi/N)-\varphi(2\xi/N))\wh{f}(\xi).
\end{align*} We define $P_{<N}$ and $P_{\geq N}$ similarly. We also define $$P_{M<\cdot\leq N}:=P_{\leq N}-P_{\leq M}=\sum_{M<N'\leq N}P_{N'}$$ for $M<N.$ All such summations are understood to be over $N\in 2^{\mathbb{Z}}.$ Being Fourier multiplier operators, these Littlewood--Paley projection operators commute with $e^{it\Delta}$, as well as differential operators (for example, $i\partial_t+\Delta$). We will need the following standard estimates for these operators:
\begin{lemma}[Bernstein estimates] For $1\leq r\leq q\leq\infty$ and $s\geq 0$,
\begin{align*}
\norm{\vert\nabla\vert^{\pm s}P_{N}f}_{L_x^r(\R^d)}	&\sim N^{\pm s}\norm{P_Nf}_{L_x^r(\R^d)},
\\ \norm{\vert\nabla\vert^s P_{\leq N}f}_{L_x^r(\R^d)}&\lesssim N^s\norm{P_{\leq N}f}_{L_x^r(\R^d)},
\\ \norm{P_{\geq N}f}_{L_x^r(\R^d)}&\lesssim N^{-s}\norm{\vert\nabla\vert^s P_{\geq N}f}_{L_x^r(\R^d)},
\\ \norm{P_{\leq N}f}_{L_x^q(\R^d)}&\lesssim N^{\frac{d}{r}-\frac{d}{q}}\norm{P_{\leq N}f}_{L_x^r(\R^d)}.
\end{align*}
\end{lemma}
\begin{lemma}[Littlewood--Paley square function estimates]\label{square function estimates} For $1<r<\infty$, 
\begin{align*}\norm{\big(\sum \vert P_N f(x)\vert^2\big)^{1/2}}_{L_x^r(\R^d)}&\sim \norm{f}_{L_x^r(\R^d)},
\\ \norm{\big(\sum N^{2s}\vert f_N(x)\vert^2\big)^{1/2}}_{L_x^r(\R^d)}& \sim\norm{\ntw{s}f}_{L_x^r(\R^d)}\quad\text{for }s\in\R,
\\ \norm{\big(\sum N^{2s}\vert f_{> N}(x)\vert^2\big)^{1/2}}_{L_x^r(\R^d)}&\sim\norm{\ntw{s}f}_{L_x^r(\R^d)}\quad\text{for }s>0.
\end{align*}
\end{lemma}

We will also need the following general inequalities, which appear originally in \cite{CW}. For a textbook treatment, one can refer to \cite{taylor}. 
\begin{lemma}[Fractional product rule, \cite{CW}]\label{fractional product rule} Let $s\in(0,1]$ and $1<r,r_1,r_2,q_1,q_2<\infty$ such that $\tfrac{1}{r}=\tfrac{1}{r_i}+\tfrac{1}{q_i}$ for $i=1,2$. Then $$\norm{\ntw{s}(fg)}_r\lesssim\norm{f}_{r_1}\norm{\ntw{s}g}_{q_1}+\norm{\ntw{s}f}_{r_2}\norm{g}_{q_2}.$$
\end{lemma}

\begin{lemma}[Fractional chain rule, \cite{CW}]\label{fractional chain rule} Suppose $G\in C^1(\C),$ $s\in(0,1]$, and $1<r,r_1,r_2<\infty$ are such that $\tfrac{1}{r}=\tfrac{1}{r_1}+\tfrac{1}{r_2}.$ Then $$\norm{\vert\nabla\vert^s G(u)}_r\lesssim \norm{G'(u)}_{r_1}\norm{\vert\nabla\vert^s u}_{r_2}.$$
\end{lemma}
We will also make use of the following refinement of the fractional chain rule, which appears in \cite{KV NLW2}:

\begin{lemma}[Derivatives of differences, \cite{KV NLW2}]\label{derivatives of differences} Fix $p>1$ and $0<s<1$. Then for $1<r,r_1,r_2<\infty$ such that $\tfrac{1}{r}=\tfrac{1}{r_1}+\tfrac{p-1}{r_2},$ we have
$$\norm{\vert\nabla\vert^s[\vert u+v\vert^p-\vert u\vert^p]}_r\lesssim\norm{\vert\nabla\vert^s u}_{r_1}\norm{v}_{r_2}^{p-1}+\norm{\vert\nabla\vert^s v}_{r_1}\norm{u+v}_{r_2}^{p-1}.$$
\end{lemma}

Finally, we prove a paraproduct estimate, very much in the spirit of Lemma 2.3 in \cite{revisit1}.
\begin{lemma}[Paraproduct estimate]\label{paraproduct} Fix $d\in\{3,4,5\}.$

(a) For $p>0$ such that $s_c:=\tfrac{d}{2}-\tfrac{2}{p}\in[\tfrac12,1)$, we have $$\norml{\vert\nabla\vert^{-\frac12s_c}(fg)}{\frac{2d}{d+2}}\lesssim \norml{\vert\nabla\vert^{-\frac12s_c}f}{\frac{2d}{d-2}}\norml{\vert\nabla\vert^{\frac12s_c}g}{\frac{4dp}{p(d+8)-4}}.$$

(b) We also have  $$\norml{\ntw{-\frac{2}{5}}(fg)}{\frac{2d}{d+2}}\lesssim\norml{\ntw{-\frac25}f}{\frac{10d}{5d-11}}\norml{\ntw{\frac25}g}{\frac{2d}{5}}.$$ 
\end{lemma}
\begin{proof} For (a), we will prove the equivalent estimate
				$$\norml{\vert\nabla\vert^{-\frac12s_c}\left(\vert\nabla\vert^{\frac12s_c}f\, \vert\nabla\vert^{-\frac12s_c}g\right)}{\frac{2d}{d+2}}
				\lesssim\norml{f}{\frac{2d}{d-2}}\norml{g}{\frac{4dp}{p(d+8)-4}}$$ 
by decomposing the left-hand side into low-high and high-low frequency interactions. More precisely, we introduce the projections $\pi_{l,h}$ and $\pi_{h,l}$, defined for any pair of functions $\phi,\psi$ by 
		$$\pi_{l,h}(\phi,\psi):=\sum_{N\lesssim M}\phi_N\psi_M
		\quad\text{and}\quad
		\pi_{h,l}(\phi,\psi):=\sum_{N\gg M}\phi_N\psi_M.$$ 

We first consider the low-high interactions: by Sobolev embedding, we have 
\begin{equation}\label{paraproduct lemma 1}\norml{\vert\nabla\vert^{-\frac12s_c}\pi_{l,h}(\vert\nabla\vert^{\frac12s_c}f,\vert\nabla\vert^{-\frac12s_c}g)}{\frac{2d}{d+2}}\lesssim\norml{\pi_{l,h}(\vert\nabla\vert^{\frac12s_c}f,\vert\nabla\vert^{-\frac12s_c}g)}{\frac{4dp}{p(3d+4)-4}}.
\end{equation} 
We note here that when $d=3$, the assumption $s_c<1$ guarantees that
				$$\tfrac{4dp}{p(3d+4)-4}>1.$$
If we now consider the multiplier of the operator given by $$T(f,g):=\pi_{l,h}(\vert\nabla\vert^{\frac12s_c}f,\vert\nabla\vert^{-\frac12s_c}g),$$ that is, $$\sum_{N\lesssim M}\vert\xi_1\vert^{\frac12s_c}\wh{f_N}(\xi_1)\vert\xi_2\vert^{-\frac12s_c}\wh{g_M}(\xi_2),$$ then we see that this multiplier is a symbol of order zero with $\xi=(\xi_1,\xi_2)$. Thus, continuing from \eqref{paraproduct lemma 1}, we can cite a theorem of Coifman--Meyer (see \cite{coifman meyer 1, coifman meyer 2}, for example) to conclude
			$$\norml{\vert\nabla\vert^{-\frac12s_c}\pi_{l,h}(\vert\nabla\vert^{\frac12s_c}f,\vert\nabla\vert^{-\frac12s_c}g)}{\frac{2d}{d+2}}
			\lesssim\norml{f}{\frac{2d}{d-2}}\norml{g}{\frac{4dp}{p(d+8)-4}}.$$ 

We now consider the high-low interactions: if we consider the multiplier of the operator given by $$S(f,h):=\vert\nabla\vert^{-\frac12s_c}\pi_{h,l}(\vert\nabla\vert^{\frac12s_c}f,h),$$ that is,
$$\sum_{N\gg M}\vert\xi_1+\xi_2\vert^{-\frac12s_c}\vert\xi_1\vert^{\frac12s_c}\wh{f_N}(\xi_1)\wh{h_M}(\xi_2),$$ then we see that this multiplier is also a symbol of order zero. Thus, using the result cited above, along with Sobolev embedding, we can estimate
\begin{align*} \norml{\vert\nabla\vert^{-\frac12s_c}\pi_{h,l}(\vert\nabla\vert^{\frac12s_c}f,\vert\nabla\vert^{-\frac12s_c}g)}{\frac{2d}{d+2}}&\lesssim\norml{f}{\frac{2d}{d-2}}\norml{\vert\nabla\vert^{-\frac12s_c}g}{\frac{d}{2}}
\\ &\lesssim \norml{f}{\frac{2d}{d-2}}\norml{g}{\frac{4dp}{p(d+8)-4}}.
\end{align*} Combining the low-high and high-low interactions, we recover (a). Mutis mutandis, the exact same proof gives (b). \end{proof}
\subsection{Strichartz estimates} Let $e^{it\Delta}$ be the free Schr\"odinger propagator,
			$$[e^{it\Delta}f](x)=\tfrac{1}{(4\pi it)^{d/2}}\int_{\R^d} e^{i\vert x-y\vert^2/4t}f(y)\,dy$$
for $t\neq 0$. This explicit formula immediately implies the dispersive estimate $$\norm{e^{it\Delta}f}_{L_x^\infty(\R^d)}\lesssim\vert t\vert^{-\frac{d}{2}}\norm{f}_{L_x^1(\R^d)}$$
for $t\neq 0$. Interpolating with $\norm{e^{it\Delta} f}_{L_x^2(\R^d)}=\norm{f}_{L_x^2(\R^d)}$ (cf. Plancherel), one arrives at
 \begin{equation}\label{dispersive} \norm{e^{it\Delta}f}_{L_x^r(\R^d)}\lesssim \vert t\vert^{-(\frac{d}{2}-\frac{d}{r})}\norm{f}_{L_x^{r'}(\R^d)}
\end{equation}
for $t\neq 0$ and $2\leq r\leq\infty$, with $\tfrac1r+\tfrac{1}{r'}=1.$ This estimate can be used to prove the standard Strichartz estimates, which we state below. First, we need the following
\begin{definition}[Admissible pairs] For $d\geq 3$, we call a pair of exponents $(q,r)$ \emph{Schr\"odinger admissible} if $$\tfrac{2}{q}+\tfrac{d}{r}=\tfrac{d}{2}\indent\text{and}\indent2\leq q,r\leq\infty.$$ For a spacetime slab $I\times\R^d$, we define $$\norm{u}_{S^0(I)}:=\sup\left\{\xnorm{u}{q}{r}{I}:(q,r)\text{ admissible}\right\}.$$ We define $S^0(I)$ to be the closure of the test functions under this norm, and denote the dual of $S^0(I)$ by $N^0(I)$. We note $$\norm{u}_{N^0(I)}\lesssim \xnorm{u}{q'}{r'}{I}\quad\text{for any admissible pair }(q,r).$$
\end{definition}
We now state the Strichartz estimates in the form we will need them.
\begin{lemma}[Strichartz] Let $s\geq 0$, let $I$ be a compact time interval, and let $u:I\times\R^d\to\C$ be a solution to the forced Schr\"odinger equation $$(i\partial_t+\Delta)u=F.$$ Then $$\norm{\vert\nabla\vert^s u}_{S^0(I)}\lesssim\norm{\vert\nabla\vert^s u(t_0)}_{L_x^2(\R^d)}+\norm{\vert\nabla\vert^s F}_{N^0(I)}$$ for any $t_0\in I$.
\end{lemma}
\begin{proof} As mentioned, the key ingredient is \eqref{dispersive}. For the endpoint $(q,r)=(2,\tfrac{2d}{d-2})$ in $d\geq 3$, see \cite{keel tao}. For the non-endpoint cases, see \cite{ginibre velo smoothing, strichartz}, for example.
\end{proof}
The free propagator also obeys some local smoothing estimates (see \cite{cons, sj, lvega} for the original results). We will make use of the following, which appears as Proposition~4.14 in \cite{KV}: 
\begin{lemma}[Local smoothing]\label{local smoothing} For any $f\in L_x^2(\R^d)$ and any $\eps>0$, 
$$\int_\R\int_{\R^d}\big\vert[\vert\nabla\vert^{\frac{1}{2}}e^{it\Delta}f](x)\big\vert^2\langle x\rangle^{-1-\eps}\,dx\,dt\lesssim_\eps\norm{f}_{L_x^2(\R^d)}^2.$$ 
\end{lemma}

Next, we record the following bilinear Strichartz estimates. The version we need can be deduced from (the proof of) Corollary 4.19 in \cite{KV}.
\begin{lemma}[Bilinear Strichartz]\label{bilinear strichartz} Let $0<s_c<\tfrac{d-1}{2}$. For any spacetime slab $I\times\R^d$ and any frequencies $M>0$ and $N>0$, we have 
				$$\xnorms{u_{\leq M}v_{\geq N}}{2}{I}
				\lesssim M^{(\frac{d-1}{2}-s_c)}N^{-(\frac12+s_c)}
				\norm{\nsc u_{\leq M}}_{S_0^*(I)}\norm{\nsc v_{\geq N}}_{S_0^*(I)},$$
where 
			$$\norm{u}_{S_0^*(I)}
			:=\xnorm{u}{\infty}{2}{I}+\xnorms{(i\partial_t+\Delta)u}{\frac{2(d+2)}{d+4}}{I}.$$
\end{lemma}
\begin{remark} We will use Lemma \ref{bilinear strichartz} in the proof of Proposition \ref{lts lemma}. In that context, we will have $u=v$ an almost periodic solution to \eqref{nls} and $I=J_k$, a characteristic subinterval. In this case, interpolating between $u\in L_t^\infty\dot{H}_x^{s_c}$ and Lemma~\ref{spacetime bounds} gives 
					$$\norm{\nsc u}_{S^0(J_k)}\lesssim_u 1,$$ 
so that we can use the fractional chain rule and Sobolev embedding to estimate
\begin{align*}
\xnorms{\nsc(\vert u\vert^p u)}{\frac{2(d+2)}{d+4}}{J_k}&\lesssim\xnorms{u}{\frac{p(d+2)}{2}}{J_k}^p\xnorms{\nsc u}{\frac{2(d+2)}{d}}{J_k}
\\ &\lesssim \norm{\nsc u}_{S^0(J_k)}^{p+1}
\\ &\lesssim_u 1.
\end{align*}
Thus, in this setting, an application of Lemma \ref{bilinear strichartz} gives
$$\xnorms{u_{\leq M}u_{\geq N}}{2}{J_k}\lesssim_u M^{(\frac{d-1}{2}-s_c)}N^{-(\frac12+s_c)}.$$
\end{remark}
\subsection{Concentration-compactness} We record here the linear profile decomposition that we will use to prove the reduction in Theorem \ref{reduction}. We begin with the following 
\begin{definition}[Symmetry group]\label{symmetry group} For any position $x_0\in\R^d$ and scaling parameter $\lambda>0$, we define a unitary transformation $g_{x_0,\lambda}:\dot{H}_x^{s_c}(\R^d)\to\dot{H}_x^{s_c}(\R^d)$ by $$[g_{x_0,\lambda}f](x):=\lambda^{-\frac{2}{p}}f\left(\lambda^{-1}(x-x_0)\right)$$ (recall $s_c:=\tfrac{d}{2}-\tfrac{2}{p}$). We let $G$ denote the collection of such transformations. For a function $u:I\times\R^d\to\C$, we define $T_{g_{x_0,\lambda}} u:\lambda^2 I\times\R^d\to\C$ by the formula $$[T_{g_{x_0,\lambda}}u](t,x):=\lambda^{-\frac{2}{p}}u\left(\lambda^{-2}t,\lambda^{-1}(x-x_0)\right),$$ where $\lambda^2I:=\{\lambda^2 t:t\in I\}.$ Note that if $u$ is a solution to \eqref{nls}, then $T_g u$ is a solution to \eqref{nls} with initial data $gu_0$. 
\end{definition}
\begin{remark} It is easily seen that $G$ is a group under composition. The map $u\mapsto T_g u$ maps solutions to \eqref{nls} to solutions with the same scattering size, that is, $S(T_gu)=S(u)$. Furthermore, $u$ is a maximal-lifespan solution if and only if $T_gu$ is a maximal-lifespan solution.
\end{remark}
We now state the linear profile decomposition in the form that we need. For $s_c=0$, this result was originally proven in \cite{ begout-vargas, carles-keraani, merle-vega}, while for $s_c=1$ it was established in \cite{keraani}. In the generality we need, a proof can be found in \cite{shao}.
\begin{lemma}[Linear profile decomposition, \cite{shao}]\label{linear profile decomposition} Fix $0<s_c<1$ and let $\{u_n\}_{n\geq 1}$ be bounded sequence in $\dot{H}_x^{s_c}(\R^d).$ Then, after passing to a subsequence if necessary, there exist functions $\{\phi^j\}_{j\geq 1}\subset\dot{H}_x^{s_c}(\R^d),$ group elements $g_n^j\in G$, and times $t_n^j\in\R$ such that for all $J\geq 1$, we have the decomposition $$u_n=\sum_{j=1}^J g_n^j e^{it_n^j\Delta}\phi^j+w_n^J$$ with the following properties:

$\bullet$ For all $n$ and all $J\geq 1$, we have $w_n^J\in\dot{H}_x^{s_c}(\R^d)$, with 
\begin{equation}\label{asymptotic vanishing}\lim_{J\to\infty}\limsup_{n\to\infty}\ \xnorms{e^{it\Delta}w_n^J}{\frac{p(d+2)}{2}}{\R}=0.\end{equation}

$\bullet$ For any $j\neq k$, we have the following asymptotic orthogonality of parameters: 
\begin{equation}\label{asymptotic orthogonality}\frac{\lambda_n^j}{\lambda_n^k}+\frac{\lambda_n^k}{\lambda_n^j}+\frac{\vert x_n^j-x_n^k\vert^2}{\lambda_n^j\lambda_n^k}+\frac{\vert t_n^j(\lambda_n^j)^2-t_n^k(\lambda_n^k)^2\vert}{\lambda_n^j\lambda_n^k}\to\infty\quad \text{as}\ n\to\infty.\end{equation}

$\bullet$ For any $J\geq 1$, we have the decoupling properties: 
\begin{equation}\label{strong decoupling}\lim_{n\to\infty}\bigg[\norm{\nsc u_n}_2^2-\sum_{j=1}^J\norm{\nsc\phi^j}_2^2-\norm{\nsc w_n^J}_2^2\bigg]=0,\end{equation}

and for any $1\leq j\leq J$, 
\begin{equation}\label{weak decoupling}e^{-it_n^j\Delta}[(g_n^j)^{-1}w_n^J]\rightharpoonup 0\quad\text{weakly in }\dot{H}_x^{s_c}\ \text{as }n\to\infty.\end{equation}
\end{lemma}
\begin{remark}\label{scaling parameters} In this linear profile decomposition, we may always choose the scaling parameters $\lambda_n^k$ so that they belong to $2^{\mathbb{Z}}$.
\end{remark}
\section{Local well-posedness}\label{local theory}

In this section, we develop a local theory for \eqref{nls}.  We begin by recording a standard local well-posedness result proven by Cazenave--Weissler \cite{cw}; see also \cite{cazenave book, KV, tao book}. We will also need to establish a stability result (appearing as Theorem~\ref{stability}), which will be essential in the reduction to almost periodic solutions in Section~\ref{reduction section}. For stability results in the mass- and energy-critical settings, see \cite{CKSTT:gwp, RV, TV:stability, TVZ}.

For the following local well-posedness result, one must assume that the initial data belongs to the inhomogeneous Sobolev space $H_x^{s_c}(\R^d).$ This assumption serves to simplify the proof (allowing for a contraction mapping argument in mass-critical spaces); we can remove it a posteriori by using the stability result we prove below. 
\begin{theorem}[Standard local well-posedness \cite{cw}]\label{standard lwp} Let $d\geq 1$, $0< s_c< 1$, and $u_0\in H_x^{s_c}(\R^d).$ Then there exists $\eta_0>0$ so that if $0<\eta\leq\eta_0$ and $I$ is an interval containing zero such that 
\begin{equation}\label{standard lwp smallness} \xnorm{\nsc e^{it\Delta}u_0}{p+2}{\frac{2d(p+2)}{2(d-2)+dp}}{I}\leq\eta,
\end{equation} then there exists a unique solution $u$ to \eqref{nls} that obeys the following bounds:
\begin{align*}
\xnorm{\nsc u}{p+2}{\frac{2d(p+2)}{2(d-2)+dp}}{I}&\leq 2\eta,
\\ \norm{\nsc u}_{S^0(I)}&\lesssim\norm{\nsc u_0}_{L_x^2(\R^d)}+\eta^{p+1},
\\ \norm{u}_{S^0(I)}&\lesssim\norm{u_0}_{L_x^2(\R^d)}.
\end{align*}
\end{theorem} 

\begin{remark} By Strichartz, we have $$\xnorm{\nsc e^{it\Delta}u_0}{p+2}{\frac{2d(p+2)}{2(d-2)+dp}}{I}\lesssim\norm{\nsc u_0}_{L_x^2(\R^d)},$$ so that \eqref{standard lwp smallness} holds with $I=\R$ for sufficiently small initial data. One can also guarantee that \eqref{standard lwp smallness} holds by taking $\vert I\vert$ sufficiently small (cf. monotone convergence).
\end{remark} 


We now turn to the question of stability for \eqref{nls}. We will prove a stability result for $(d,s_c)$ satisfying \eqref{constraints2}, in which case we always have $p\geq1$. As we will see, this assumption allows for a very simple stability theory. On the other hand, when $p<1$, developing a stability theory can become quite delicate. For a discussion in the energy-critical case, see \cite[Section 3.4]{KV} and the references cited therein. See also \cite{KV:supercritical} for a stability theory in the energy-supercritical regime, as well as \cite{DZ} for a stability theory in the inter-critical regime in high dimensions.

Following the arguments in \cite{KV}, we begin with the following
\begin{lemma}[Short-time perturbations]\label{short-time} Fix $(d,s_c)$ satisfying \eqref{constraints2}. Let $I$ be a compact interval and $\wwt{u}:I\times\R^d\to\C$ a solution to $$(i\partial_t+\Delta)\wwt{u}=\vert\wwt{u}\vert^p\wwt{u}+e$$ for some function $e$. Assume that \begin{equation}\label{stability hypoth}\norm{\wwt{u}}_{L_t^\infty\dot{H}_x^{s_c}(I\times\R^d)}\leq E.\end{equation} Let $t_0\in I$ and $u_0\in\dot{H}_x^{s_c}(\R^d)$. Then there exist $\eps_0,$ $\delta>0$ (depending on $E$) such that for all $0<\eps<\eps_0$, if
\begin{align}\label{delta}\xnorm{\nsc \wwt{u}}{\frac{p(d+2)}{2}}{\frac{2dp(d+2)}{d^2p+2dp-8}}{I}&\leq\delta,
\\ \label{initial data close}\norm{u_0-\wwt{u}(t_0)}_{\dot{H}_x^{s_c}(\R^d)}&\leq\eps,
\\ \label{error small}\norm{\nsc e}_{N^0(I)}&\leq \eps,\end{align}
then there exists $u:I\times\R^d\to\C$ solving $(i\partial_t+\Delta)u=\vert u\vert^p u$ with $u(t_0)=u_0$ satisfying
\begin{align}\label{strichartz norms close}\norm{\nsc(u-\wwt{u})}_{S^0(I)}&\lesssim\eps,
\\ \label{strichartz norms under control}\norm{\nsc u}_{S^0(I)}&\lesssim E,
\\ \label{nonlinearities close}\norm{\nsc(\vert u\vert^p u-\vert\wwt{u}\vert^p\wwt{u})}_{N^0(I)}&\lesssim\eps.\end{align} 
\end{lemma} 
\begin{proof} We prove the lemma under the additional hypothesis $u_0\in L_x^2(\R^d)$; this allows us (by Theorem \ref{standard lwp}) to find a solution $u$, so that we are left to prove all of the estimates as a priori estimates. Once the lemma is proven, we can use approximation by $H_x^{s_c}(\R^d)$ functions (along with the lemma itself) to see that the lemma holds for $u_0\in\dot{H}_x^{s_c}(\R^d)$.

Define $w=u-\wwt{u}$, so that $(i\partial_t+\Delta)w=\vert u\vert^p u-\vert\wwt{u}\vert^p\wwt{u}-e$. Without loss of generality, assume $t_0=\inf I$, and define $$A(t)=\norm{\nsc(\vert u\vert^p u-\vert\wwt{u}\vert^p \wwt{u})}_{N^0([t_0,t))}.$$ We first note that by Duhamel, Strichartz, \eqref{initial data close}, and \eqref{error small}, we get
			\begin{align} 
			\norm{&\nsc w}_{S^0([t_0,t))}		\nonumber
			\\ &\lesssim \norm{\nsc w(t_0)}_{L_x^2(\R^d)} \nonumber
			+\norm{\nsc(\vert u\vert^p u-\vert\wwt{u}\vert^p\wwt{u})}_{N^0([t_0,t))}+\norm{\nsc e}_{N^0(I)}
			\\ &\lesssim \eps+A(t).			\label{short-time bootstrap}
			\end{align}
Using this fact, together with Lemma \ref{derivatives of differences}, \eqref{delta}, and Sobolev embedding, we can estimate (with all spacetime norms over $[t_0,t)\times\R^d$)
\begin{align*} A(t)&\lesssim\xonorm{\nsc(\vert\wwt{u}+w\vert^p(\wwt{u}+w)-\vert\wwt{u}\vert^p\wwt{u})}{\frac{p(d+2)}{2(p+1)}}{\frac{2dp(d+2)}{d^2p+6dp-8}}
\\ 
&\lesssim \xonorm{\nsc\wwt{u}}{\frac{p(d+2)}{2}}{\frac{2dp(d+2)}{d^2p+2dp-8}}\xonorms{w}{\frac{p(d+2)}{2}}^p
\\ & \ \ \ \ \ \ \ +\xonorm{\nsc w}{\frac{p(d+2)}{2}}{\frac{2dp(d+2)}{d^2p+2dp-8}}\xonorms{\wwt{u}+w}{\frac{p(d+2)}{2}}^p
\\ & \lesssim\delta[\eps+A(t)]^p+[\eps+A(t)][\delta^p+(\eps+A(t))^p].
\end{align*}
Thus, recalling $p\geq 1$ and choosing $\delta,\eps$ sufficiently small, we conclude $A(t)\lesssim\eps$ for all $t\in I$, which gives \eqref{nonlinearities close}. Combining \eqref{nonlinearities close} with \eqref{short-time bootstrap}, we also get \eqref{strichartz norms close}. Finally, we can prove \eqref{strichartz norms under control} as follows: by Strichartz, \eqref{strichartz norms close}, \eqref{stability hypoth}, \eqref{error small}, \eqref{delta}, the fractional chain rule, and Sobolev embedding, 
\begin{align*} \norm{\nsc u}_{S^0(I)}&\lesssim\norm{\nsc(u-\wwt{u})}_{S^0(I)}+\norm{\nsc\wwt{u}}_{S^0(I)}
\\ &\lesssim \eps+\norm{\nsc \wwt{u}(t_0)}_{L_x^2(\R^d)}+\norm{\nsc(\vert\wwt{u}\vert^p\wwt{u})}_{N^0(I)}+\norm{\nsc e}_{N^0(I)}
\\ &\lesssim \eps+E+\xnorm{\nsc\wwt{u}}{\frac{p(d+2)}{2}}{\frac{2dp(d+2)}{d^2p+2dp-8}}{I}\xnorms{\wt{u}}{\frac{p(d+2)}{2}}{I}^p
\\ &\lesssim E+\eps+\delta^{p+1}
\\ &\lesssim E
\end{align*} for $\eps$ and $\delta$ sufficiently small depending on $E$.
\end{proof}

With Lemma \ref{short-time} established, we now turn to 
\begin{theorem}[Stability]\label{stability} Fix $(d,s_c)$ satisfying \eqref{constraints2}. Let $I$ be a compact time interval and $\wwt{u}:I\times\R^d\to\C$ a solution to $$(i\partial_t+\Delta)\wwt{u}=\vert\wwt{u}\vert^p\wwt{u}+e$$ for some function $e$. Assume that 
\begin{align}\label{Stability hypoth 1}\norm{\wwt{u}}_{L_t^\infty\dot{H}_x^{s_c}(I\times\R^d)}&\leq E,
\\ \label{Stability hypoth 2}S_I(\wwt{u})&\leq L.\end{align} 
Let $t_0\in I$ and $u_0\in\dot{H}_x^{s_c}(\R^d)$. Then there exists $\eps_1=\eps_1(E,L)$ such that if 
\begin{align}\label{Initial data close}\norm{u_0-\wwt{u}(t_0)}_{\dot{H}_x^{s_c}(\R^d)}&\leq\eps,
\\ \label{Error small}\norm{\nsc e}_{N^0(I)}&\leq\eps\end{align}
for some $0<\eps<\eps_1$, then there exists a solution $u:I\times\R^d\to\C$ to $(i\partial_t+\Delta)u=\vert u\vert^p u$ with $u(t_0)=u_0$ satisfying 
\begin{align}\label{Strichartz norms close}\norm{\nsc(u-\wwt{u})}_{S^0(I)}&\leq C(E,L)\eps,
\\ \label{Strichartz norms ok}\norm{\nsc u}_{S^0(I)}&\leq C(E,L).\end{align}
\end{theorem}
\begin{proof} Once again, we may assume $t_0=\inf I$. To begin, we let $\eta>0$ be a small parameter to be determined shortly. By \eqref{Stability hypoth 2}, we may subdivide $I$ into (finitely many, depending on $\eta$ and $L$) intervals $J_k=[t_k,t_{k+1})$ so that $$\xnorms{\wwt{u}}{\frac{p(d+2)}{2}}{J_k}\sim\eta$$ for each $k$. Then by Strichartz, \eqref{Stability hypoth 1}, \eqref{Error small}, and the fractional chain rule, we have
\begin{align*} \norm{\nsc\wwt{u}}_{S^0(J_k)}&\lesssim \norm{\nsc \wwt{u}(t_k)}_{L_x^2(\R^d)}+\norm{\nsc(\vert\wwt{u}\vert^p\wwt{u})}_{N^0(J_k)}+\norm{\nsc e}_{N^0(J_k)}
\\ &\lesssim E+\norm{\nsc\wwt{u}}_{S^0(I)}\xnorms{\wwt{u}}{\frac{p(d+2)}{2}}{J_k}^p+\eps
\\ &\lesssim E+\eps+\eta^p\norm{\nsc\wwt{u}}_{S^0(I)}.
\end{align*}
Thus for $\eps\leq E$ and $\eta$ sufficiently small, we find $$\norm{\nsc\wwt{u}}_{S^0(J_k)}\lesssim E.$$ Adding these bounds, we find 
			\begin{equation}\label{stabbb}\norm{\nsc \wwt{u}}_{S^0(I)}\leq C(E,L).\end{equation}

Now, we take $\delta>0$ as in Lemma \ref{short-time} and subdivide $I$ into finitely many, say $J_0=J_0(C(E,L),\delta)$ intervals $I_j=[t_j,t_{j+1})$ so that 
		$$\xnorm{\nsc\wwt{u}}{\frac{p(d+2)}{2}}{\frac{2dp(d+2)}{d^2p+2dp-8}}{I_j}\leq\delta$$ 
for each $j$. We now wish to proceed inductively. We may apply Lemma \ref{short-time} on each $I_j$, provided we can guarantee 
				\begin{equation}\label{stability hypothesis} 
				\norm{u(t_j)-\wwt{u}(t_j)}_{\dot{H}_x^{s_c}(\R^d)}\leq\eps
				\end{equation}
for some $0<\eps<\eps_0$ and each $j$ (where $\eps_0$ is as in Lemma \ref{short-time}). In the event that \eqref{stability hypothesis} holds for some $j$, applying Lemma \ref{short-time} on $I_j=[t_j,t_{j+1})$ gives
				\begin{align} 
				\norm{\nsc(u-\wwt{u})}_{S^0(I_j)}&\leq C(j)\eps, \label{inductive 1}
				\\ \norm{\nsc u}_{S^0(I_j)}&\leq C(j)E, \label{inductive 2}
				\\ \norm{\nsc(\vert u\vert^p u-\vert\wwt{u}\vert^p\wwt{u})}_{N^0(I_j)}
						&\leq C(j)\eps. \label{inductive 3}
				\end{align}
Now, we first note that \eqref{stability hypothesis} holds for $j=0$, provided we take $\eps_1<\eps_0$. Next, assuming that \eqref{stability hypothesis} holds for $0\leq k\leq j-1$, we can use Strichartz, \eqref{Initial data close}, \eqref{Error small}, and the inductive hypothesis \eqref{inductive 3} to estimate
\begin{align*} \norm{u&(t_j)-\wwt{u}(t_j)}_{\dot{H}^{s_c}_x(\R^d)}
\\ &\lesssim\norm{u(t_0)-\wwt{u}(t_0)}_{\dot{H}_x^{s_c}(\R^d)}\!+\norm{\nsc(\vert u\vert^pu-\vert\wwt{u}\vert^p\wwt{u})}_{N^0([t_0,t_j))}\! +\norm{\nsc e}_{N^0([t_0,t_j))}
\\ & \lesssim \eps+\sum_{k=0}^{j-1} C(k)\eps+\eps
\\ &< \eps_0,
\end{align*}
provided $\eps_1=\eps_1(\eps_0,J_0)$ is taken sufficiently small. Thus, by induction, we get \eqref{inductive 1} and \eqref{inductive 2} on each $I_j$. Adding these bounds over the $I_j$ yields \eqref{Strichartz norms close} and \eqref{Strichartz norms ok}. 
\end{proof} 
\begin{remark} As mentioned above, with this stability result in hand, we can see that Theorem \ref{standard lwp} holds without the assumption $u_0\in L_x^2(\R^d)$. Using this updated version of Theorem \ref{standard lwp} (along with the original proof of Theorem \ref{standard lwp}), one can then derive Theorem \ref{local theorem}. We omit the standard arguments; one can refer instead to \cite{cw, cazenave book}.
\end{remark} 

\section{Reduction to almost periodic solutions}\label{reduction section}
The goal of this section is to sketch a proof of Theorem \ref{reduction}. We will follow the argument presented in \cite[Section 3]{KV5}. By now, the general procedure is fairly standard; see, for example, \cite{kenig merle, KM, KM NLW, KV, TVZ} for other instances in different contexts. Thus, we will merely outline the main steps of the argument, providing full details only when significant new difficulties arise in our setting. 

We suppose that Theorem \ref{swamprat} fails for some $(d,s_c)$ satisfying \eqref{constraints2}. We then define the function $L:[0,\infty)\to[0,\infty]$ by
\begin{align*}L(E):=\sup\{S_I(u):u:I\!\times\!\R^d\to\C\text{ solving } \eqref{nls} \text{ with }\sup_{t\in I}\norm{u(t)}_{\dot{H}_x^{s_c}(\R^d)}^2\leq E\},
\end{align*} where $S_I(u)$ is defined as in \eqref{scattering size}. We note that $L$ is a non-decreasing function, and that Theorem \ref{local theorem} implies 
				\begin{equation}\label{L small-data}
				L(E)\lesssim E^{\frac{p(d+2)}{4}}\indent\text{for}\indent E<\eta_0,
				\end{equation} 
where $\eta_0$ is the small-data threshold. Thus, there exists a unique `critical' threshold $E_c\in(0,\infty]$ such that $L(E)<\infty$ for $E<E_c$ and $L(E)=\infty$ for $E>E_c$. The failure of Theorem \ref{swamprat} implies that $0<E_c<\infty.$

The key ingredient to proving Theorem \ref{reduction} is a Palais--Smale condition modulo the symmetries of the equation; indeed, once the following proposition is proven, deriving Theorem \ref{reduction} is standard (see \cite{KV5}).

\begin{proposition}[Palais--Smale condition modulo symmetries]\label{palais smale} Let $(d,s_c)$ satisfy \eqref{constraints2}. Let $u_n:I_n\times\R^d\to\C$ be a sequence of solutions to \eqref{nls} such that 
		$$\limsup_{n\to\infty}\sup_{t\in I_n}\norm{u_n(t)}_{\dot{H}_x^{s_c}(\R^d)}^2=E_c,$$ 
and suppose $t_n\in I_n$ are such that 
		\begin{equation}		\label{ps blowup}
		\lim_{n\to\infty} S_{[t_n,\sup I_n)}(u_n)=\lim_{n\to\infty} S_{(\inf I_n, t_n]}(u_n)=\infty.
		\end{equation}
Then the sequence $u_n(t_n)$ has a subsequence that converges in $\dot{H}_x^{s_c}(\R^d)$ modulo symmetries; that is, there exist $g_n\in G$ such that $g_n[u_n(t_n)]$ converges along a subsequence in $\dot{H}_x^{s_c}(\R^d)$, where $G$ is as in Definition \ref{symmetry group}. 
\end{proposition}
We now sketch the proof of this proposition, following the argument as it appears in \cite{KV5}. As in that setting, the main ingredients will be a linear profile decomposition (Lemma \ref{linear profile decomposition} in our case) and a stability result (Theorem \ref{stability} in our case). However, as we will see, combining fractional derivatives with non-polynomial nonlinearities will present some significant new difficulties in our setting when it comes time to apply the stability result.
 
We begin by translating so that each $t_n=0$, and apply the linear profile decomposition Lemma \ref{linear profile decomposition} to write 
			\begin{equation}\label{lpd} 
			u_n(0)=\sum_{j=1}^J g_n^j e^{it_n^j\Delta}\phi^j +w_n^J
			\end{equation} 
along some subsequence. Refining the subsequence for each $j$ and diagonalizing, we may assume that for each $j$, we have $t_n^j\to t^j\in[-\infty,\infty].$ If $t^j\in(-\infty,\infty)$, then we replace $\phi^j$ by $e^{it^j\Delta}\phi^j$, so that we may take $t^j=0$. Moreover, we can absorb the error $e^{it_n^j\Delta}\phi^j-\phi^j$ into the error term $w_n^J$, and so we may take $t_n^j\equiv 0$. Thus, without loss of generality, either $t_n^j\equiv 0$ or $t_n^j\to\pm\infty$. 

Next, appealing to Theorem \ref{local theorem}, for each $j$ we define $v^j:I^j\times\R^d\to\C$ to be the maximal-lifespan solution to \eqref{nls} such that 
			$$\left\{\begin{array}{ll}
			v^j(0)=\phi^j & \text{if }t_n^j\equiv 0,
			\\ \\ v^j\text{ scatters to }\phi^j\text{ as }t\to\pm\infty & \text{if }t_n^j\to\pm\infty.
			\end{array}\right.$$
We now define the nonlinear profiles $v_n^j:I_n^j\times\R^d\to\C$ by 
				$$v_n^j(t)=g_n^j v^j\left((\lambda_n^j)^{-2}t+t_n^j\right),$$ 
where $I_n^j=\{t:(\lambda_n^j)^{-2}t+t_n^j\in I^j\}.$ Now, the $\dot{H}_x^{s_c}$ decoupling of the profiles $\phi^j$, \eqref{strong decoupling}, immediately tells us that the $v_n^j$ are global and scatter for $j$ sufficiently large, say $j\geq J_0$; indeed, for large enough $j$, we are in the small-data regime. We want to show that there exists some $1\leq j_0<J_0$ such that 
			\begin{equation}
			\limsup_{n\to\infty} S_{[0,\sup I_n^{j_0})}(v_n^{j_0})=\infty.		\label{one bad profile}
			\end{equation}

Once we obtain at least one such `bad' nonlinear profile, we can show that in fact, there can only be one profile. To see this, one needs to adapt the argument in \cite[Lemma 3.3]{KV5} to see that the $\dot{H}_x^{s_c}$ decoupling of the profiles persists in time (this does not follow immediately, as the $\dot{H}_x^{s_c}$-norm is not a conserved quantity for \eqref{nls}). Then, the `critical' nature of $E_c$ can be used to rule out the possibility of multiple profiles. 

Comparing with \eqref{lpd}, one sees that once we show that there is only one profile $\phi^{j_0}$, the proof of Proposition \ref{palais smale} is nearly complete; one needs only to rule out the cases $t_n^{j_0}\to\pm\infty$. This can be easily done by applying the stability theory; we omit the details and instead refer the reader to \cite{KV5}.

We turn now to proving that there is at least one bad profile. We suppose towards a contradiction that there are no bad nonlinear profiles. In this case, we can show 
				\begin{equation}								\label{okkkk}
				\sum_{j\geq 1}S_{[0,\infty)}(v_n^j)\lesssim_{E_c} 1
				\end{equation} 
for $n$ sufficiently large (to control the tail of the sum, for example, we recall that for $j\geq J_0$, we are in the small-data regime; thus we can use \eqref{strong decoupling} and \eqref{L small-data} to bound the tail by $E_c$ for $n$ sufficiently large). We would like to use \eqref{okkkk} and the stability result (Theorem~\ref{stability}) to deduce a bound on the scattering size of the $u_n$, thus contradicting \eqref{ps blowup}. 

To this end, we define the approximations 
			$$u_n^J(t)=\sum_{j=1}^J v_n^j(t)+e^{it\Delta}w_n^J.$$
By the construction of the $v_n^j$, it is easy to see that 
			\begin{equation}
			\limsup_{n\to\infty}\norm{u_n(0)-u_n^J(0)}_{\dot{H}_x^{s_c}(\R^d)}=0.
			\label{initial data close approximation}
			\end{equation}
We also claim that we have
			\begin{equation}
			\lim_{J\to\infty}\limsup_{n\to\infty} S_{[0,\infty)}(u_n^J)\lesssim_{E_c} 1. \label{ok version2}
			\end{equation}
To see why \eqref{ok version2} holds, first note that by \eqref{asymptotic vanishing} and \eqref{okkkk}, it suffices to show
			\begin{equation}
			\lim_{J\to\infty}\limsup_{n\to\infty}\bigg\vert S_{[0,\infty)}\bigg(\sum_{j=1}^J v_n^j\bigg)-\sum_{j=1}^J S_{[0,\infty)}(v_n^j)\bigg\vert=0.
			\label{ok version3}
			\end{equation}
To establish \eqref{ok version3}, we can first use the pointwise inequality
		$$\bigg\vert\ \bigg\vert \sum_{j=1}^J v_n^j\bigg\vert^{\frac{p(d+2)}{2}}-\sum_{j=1}^J\vert v_n^j\vert^{\frac{p(d+2)}{2}}\bigg\vert\lesssim_J \sum_{j\neq k}\vert v_n^j\vert^{\frac{p(d+2)}{2}-1}\vert v_n^k\vert$$
along with H\"older's inequality to see
		\begin{align}
		\text{LHS}\eqref{ok version3} &\lesssim_J \sum_{j\neq k}\xnorms{v_n^j}{\frac{p(d+2)}{2}}{[0,\infty)}^{\frac{p(d+2)}{2}-2}\xnorms{v_n^j v_n^k}{\frac{p(d+2)}{4}}{[0,\infty)}.
		\label{then gives}
		\end{align}
Then, following an argument of Keraani (cf. \cite[Lemma 2.7]{keraani}), given $j\neq k$, we can approximate $v^j$ and $v^k$ by compactly supported functions in $\R\times\R^d$ and use the asymptotic orthogonality of parameters \eqref{asymptotic orthogonality} to show
		\begin{equation}
		\limsup_{n\to\infty}\xnorms{v_n^j v_n^k}{\frac{p(d+2)}{4}}{[0,\infty)}=0.
		\label{yay keraani}
		\end{equation}
Thus, continuing from \eqref{then gives}, we get that \eqref{ok version3} (and therefore \eqref{ok version2}) holds.
 

With \eqref{initial data close approximation} and \eqref{ok version2} in place, we see that if we can show that the $u_n^J$ asymptotically solve \eqref{nls}, that is, 
			$$\lim_{J\to\infty}\limsup_{n\to\infty}
			\norm{\nsc\big[(i\partial_t+\Delta)u_n^J-\vert u_n^J\vert^pu_n^J\big]}_{N^0([0,\infty))}
			=0,$$ 
then we will be able to apply Theorem \ref{stability} to deduce bounds on the scattering size of the $u_n$. Writing  $F(z)=\vert z\vert^p z$, the proof of Proposition \ref{palais smale} therefore reduces to showing the following
\begin{lemma}[Decoupling of nonlinear profiles]\label{decouple lemma}
				\begin{equation}\label{decouple 1}
				\lim_{J\to\infty}\limsup_{n\to\infty} 
				\bigg\|\nsc\bigg( F\big(\sum_{j=1}^J v_n^j\big)-\sum_{j=1}^J F(v_n^j)\bigg)\bigg\|_{N^0([0,\infty))}=0,
				\end{equation}
				
			\begin{equation}\label{decouple 2}
			\lim_{J\to\infty}\limsup_{n\to\infty}
			\norm{\nsc\left( F(u_n^J-e^{it\Delta}w_n^J)-F(u_n^J)\right)}_{N^0([0,\infty))}=0.
					\end{equation}
\end{lemma}
While many of the ideas needed to establish this lemma may be found in \cite{keraani}, we will see that new difficulties appear in our setting. Consider, for example, \eqref{decouple 1}. In the mass-critical setting, i.e. $s_c=0$, one has the pointwise estimate 
\begin{equation}\label{pw 1} \bigg\vert\ F\big(\sum_{j=1}^J v_n^j\big)-\sum_{j=1}^J F(v_n^j)\bigg\vert\lesssim_J \sum_{j\neq k}\vert v_n^j\vert\,\vert v_n^k\vert^p.
\end{equation} 
To see that the contribution of the terms on the the right-hand side of \eqref{pw 1} is acceptable, one can follow the argument of Keraani just described above; that is, one can use the asymptotic orthogonality of parameters to derive an estimate like \eqref{yay keraani}, which in turn gives \eqref{decouple 1}. 

In the energy-critical setting, i.e. $s_c=1$, one can instead use the pointwise estimate 
		$$\bigg\vert \nabla\bigg(F\big(\sum_{j=1}^J v_n^j\big)-\sum_{j=1}^J F(v_n^j)\bigg)\bigg\vert
		\lesssim_J \sum_{j\neq k}\vert\nabla v_n^j\vert\,\vert v_n^k\vert^{p}.$$ 
A similar argument can then be used to prove \eqref{decouple 1}; the key in both cases is to exhibit terms that all contain some $v_n^j$ paired against some $v_n^k$ for $j\neq k$. 

In the energy-supercritical case, the authors of \cite{KV:supercritical} were able to establish analogous pointwise estimates (in terms of the Hardy--Littlewood maximal function) for a square function of Strichartz that shares estimates with fractional differentiation operators (see \cite{strichartz square}). With the appropriate pointwise estimates in place, the usual arguments can then be applied; in this way, a potentially complicated analysis is handled quite efficiently. The approach of \cite{KV:supercritical}, however, does not work in our setting, as it relies fundamentally on the fact that $s_c>1$. 

See also \cite{KM}, which deals with the case $d=3$ and $s_c=\tfrac12$ (in which case $p=2$). In that setting, one also has to face the nonlocal nature of $\vert\nabla\vert^{\frac12}$; however, by using the polynomial nature of the nonlinearity, along with the well-developed theory of paraproducts (see \cite{CMpara, taylor}), the authors are able to place themselves back into a situation where the usual arguments apply. In this way, they are able to overcome the difficulty of fractional derivatives while still providing a very clean analysis.

In our case, we must deal simultaneously with a non-polynomial nonlinearity and a fractional number of derivatives; as we will see, this necessitates a fairly delicate and technical analysis. The main difficulty of our task stems from the fact the nonlocal operator $\vert\nabla\vert^{s_c}$ does not respect pointwise estimates in the spirit of \eqref{pw 1}. We will deal with this problem by opening up the proof of the fractional chain rule (Lemma \ref{fractional chain rule}) as given in \cite[$\S 2.4$]{taylor}; in particular, we will employ the Littlewood--Paley square function (specifically, Lemma \ref{square function estimates}), which allows us to work at the level of individual frequencies. By making use of maximal function and vector maximal function estimates, we can then find a way to adapt the standard arguments. 

\begin{proof}[Proof of \eqref{decouple 1}]  By induction, it will suffice to treat the case of two summands; to simplify notation, we write $f=v_n^j$ and $g=v_n^k$ for some $j\neq k$, and we are left to show 
\begin{align}\norm{\nsc\big(\vert f+g\vert^p(f+g)-\vert f\vert^p f-\vert g\vert^p g\big)}_{N^0([0,\infty))}\to 0\label{decoupling to zero}\end{align} as $n\to\infty$. 

As alluded to above, the key will be to perform a decomposition in such a way that all of the resulting terms we need to estimate have $f$ paired against $g$ inside of a single integrand; for such terms, we will be able to use the asymptotic orthogonality of parameters \eqref{asymptotic orthogonality} to our advantage. 

We first rewrite 
\begin{align*}\vert f+g&\vert^p(f+g)-\vert f\vert^pf-\vert g\vert^p g\\ &=\big(\vert f+g\vert^p-\vert f\vert^p\big)f+\big(\vert f+g\vert^p -\vert g\vert^p\big)g.\end{align*} 
By symmetry, it will suffice to treat the first term. We turn therefore to estimating
			$$\xonorms{\nsc\big[(\vert f+g\vert^p-\vert f\vert^p)f\big]}{\frac{2(d+2)}{d+4}}.$$
By Lemma \ref{square function estimates}, it will suffice to consider
				\begin{equation}\label{sq}
				\bxonorms{\left(\sum\big\vert N^{s_c} P_N\big[ (\vert f+g\vert^p-\vert f\vert^p)f\big]\big\vert^2\right)^{1/2} }{\frac{2(d+2)}{d+4}}.
				\end{equation}
Thus, we restrict our attention to a single frequency $N\in 2^{\mathbb{Z}}$. We let $\delta_y f(x):=f(x-y)-f(x)$, and let $\psih$ denote the convolution kernel of the Littlewood--Paley projection $P_1$. As $\psi(0)=0$, we have 
			$$\smallint\psih(y)\,dy=0,$$ 
so that exploiting cancellation, we can write
\begin{align} & P_N\big(\big[\vert f(x)+g(x)\vert^p-\vert f(x)\vert^p\big]f(x)\big) \nonumber
\\ &\quad =\smallint N^d \psih(Ny)\delta_y\big(\big[\vert f(x)+g(x)\vert^p-\vert f(x)\vert^p\big]f(x)\big)\,dy.\label{single freq}
\end{align}
We now rewrite
\begin{align}
\delta_y&\big(\big[\vert f(x)+g(x)\vert^p-\vert f(x)\vert^p\big]f(x)\big) \nonumber
\\ & =\delta_y f(x)\big[\vert f(x-y)+g(x-y)\vert^p-\vert f(x-y)\vert^p\big] \label{expansion1}
\\ &\ +\!f(x)\big[\vert f(x)+g(x-y)\vert^p-\vert f(x)+g(x)\vert^p\big] \label{expansion2}
\\ &\ +\!f(x)\big[\vert f(x-y)\!+\!g(x-y)\vert^p\!-\!\vert f(x-y)\vert^p\!+\!\vert f(x)\vert^p\!-\!\vert f(x)\!+\!g(x-y)\vert^p\big]. \label{expansion3}
\end{align}
We estimate each term individually. First, we have
\begin{align*} \vert\eqref{expansion1}\vert\lesssim \vert\delta_y f(x)\vert\,\vert g(x-y)\vert\left\{\vert f(x-y)\vert^{p-1}+\vert g(x-y)\vert^{p-1}\right\}.\end{align*}
Next, we see
\begin{align*}\vert\eqref{expansion2}\vert\lesssim\vert f(x)\vert\,\vert \delta_y g(x)\vert\left\{\vert f(x)\vert^{p-1}+\vert g(x)\vert^{p-1}+\vert g(x-y)\vert^{p-1}\right\}.\end{align*}
We now turn to \eqref{expansion3}. First, if $1<p\leq 2$, a simple argument using the fundamental theorem of calculus implies
		\begin{align*} 
				\vert\eqref{expansion3}\vert
				\lesssim \vert f(x)\vert\,\vert \delta_yf(x)\vert\,\vert g(x-y)\vert^{p-1}
		\end{align*}
(see Lemma \ref{stupid lemma} for details). For $p>2$, one instead finds
	\begin{align*}
			\vert\eqref{expansion3}\vert\lesssim
			\vert f(x)\vert\,\vert\delta_yf(x)\vert\,\vert g(x-y)\vert\left\{\vert f(x)\vert^{p-2}
			+\vert f(x-y)\vert^{p-2}+\vert g(x-y)\vert^{p-2}\right\}.
		\end{align*} 

\begin{remark}\label{goodbye p=1} Let us pause here to note that if $p=1$, the approach above breaks down. Notice that each term in the bounds for \eqref{expansion1}, \eqref{expansion2}, and \eqref{expansion3} has two essential properties: (i) it features $f$ paired against powers of $g$, and (ii) the derivative (in the form of $\delta_y$) lands on either $f$ or $g$. When $p=1$, the same approach does not yield a decomposition that is satisfactory in this sense; it is for this reason that we have excluded the case $(d,s_c)=(5,\tfrac12)$ from this paper.
\end{remark} 
To ease the exposition, we will restrict our attention here and below to the more difficult case $1<p\leq 2$; once we have dealt with this case, it should be clear how to proceed when $p>2$. 

Collecting terms, we continue from \eqref{single freq} to see 
\begin{align}\big\vert& P_N\big(\big[\vert f(x)+g(x)\vert^p-\vert f(x)\vert^p\big]f(x)\big)\big\vert \nonumber
\\ &\lesssim \int N^d\vert\psih(Ny)\vert\,\vert\delta_y f(x)\vert\,\vert g(x-y)\vert\left\{\vert f(x-y)\vert^{p-1}+\vert g(x-y)\vert^{p-1}\right\}\!\,dy \label{11}
\\ &\ + \int N^d\vert \psih(Ny)\vert\,\vert f(x)\vert\,\vert\delta_y g(x)\vert\left\{\vert f(x)\vert^{p-1}\!+\!\vert g(x)\vert^{p-1}\!+\!\vert g(x-y)\vert^{p-1}\right\}\,\!dy\label{13}
\\ &\ + \int N^d\vert\psih(Ny)\vert\,\vert f(x)\vert\,\vert \delta_y f(x)\vert\vert g(x-y)\vert^{p-1}\,dy.\label{12}
\end{align}

One can see that we are already faced with several terms to estimate; moreover, to estimate any single term will require further decomposition. However, in the end, the same set of tools will suffice to handle every term that appears. Thus, let us deal with only \eqref{11} in detail; once we have seen how to handle this term, it should be clear that the same techniques apply to handle \eqref{13} and \eqref{12}. 

Turning to \eqref{11}, we first write
\begin{align}\eqref{11}&=\smallint N^d\vert\psih(Ny)\vert\,\vert\delta_y f(x)\vert\,\vert g(x-y)\vert\,\vert f(x-y)\vert^{p-1}\,dy \label{11f}
\\ &\ \ +\smallint N^d\vert\psih(Ny)\vert\,\vert\delta_y f(x)\vert\,\vert g(x-y)\vert^p\,dy. \label{11g}
\end{align}
For both of these terms, we will need to make use of some auxiliary inequalities in the spirit of \cite[$\S$2.3]{taylor}, which we record in Lemma \ref{aux lemma}.

We turn to \eqref{11f}. If we first write
\begin{align} \vert \delta_y &f(x)\vert\lesssim \vert f_{>N}(x)\vert+\vert f_{>N}(x-y)\vert+\sum_{K\leq N}\vert\delta_y f_K(x)\vert,\label{delta breakdown 1}
\end{align}
then putting Lemma \ref{aux lemma} to use, we arrive at
\begin{align}
\eqref{11f}\lesssim&\ \vert f_{>N}(x)\vert\, \M(g\,\vert f\vert^{p-1})(x) \label{111}
\\ &+\M(f_{>N}\,g\,\vert f\vert^{p-1})(x) \label{112}
\\ &+\sum_{K\leq N}\tfrac{K}{N}\M(f_K)(x)\M(g\,\vert f\vert^{p-1})(x) \label{113}
\\ &+\sum_{K\leq N}\tfrac{K}{N}\M(\M(f_K)\,g\,\vert f\vert^{p-1})(x) \label{114}.
\end{align}
Similarly, we can decompose
\begin{align}
\eqref{11g}\lesssim&\ \vert f_{>N}(x)\vert\,\M(\vert g\vert^{p})(x) \label{115}
\\ &+\M(f_{>N}\vert g\vert^{p})(x) \label{116}
\\ &+\sum_{K\leq N}\tfrac{K}{N} \M(f_K)(x)\M(\vert g\vert^p)(x) \label{117}
\\ &+\sum_{K\leq N}\tfrac{K}{N} \M(\M(f_K)\vert g\vert^p)(x)\label{118}.
\end{align}

Let us now consider the contribution of \eqref{111} to the left-hand side of \eqref{decoupling to zero}. Comparing with \eqref{sq}, we see it will suffice to estimate
\begin{align*} \xonorms{\big(\sum_N \big\vert N^{s_c}f_{>N} \M(g\, \vert f\vert^{p-1})\big\vert^2\big)^{1/2}}{\frac{2(d+2)}{d+4}}.\end{align*} 
Using H\"older's inequality and maximal function estimates, we can control this term by \begin{align*}\xonorms{\big(\sum_N \big\vert N^{s_c}f_{>N}\big\vert^2\big)^{1/2}}{\frac{2(d+2)}{d}}\xonorms{\vert g\vert\,\vert f\vert^{p-1}}{\frac{d+2}{2}}.
\end{align*}
We now recall that $f=v_n^j$ and $g=v_n^k$ for some $j\neq k$. Then, the first term is controlled by $\norm{\nsc v_n^j}_{S^0}$ (cf. Lemma \ref{square function estimates}), which in turn is bounded (recall that by assumption, all of the $v_n^j$ have scattering size $\lesssim E_c$). The second term can be handled in the standard way; that is, this term vanishes in the limit due to the asymptotic orthogonality of parameters \eqref{asymptotic orthogonality} (cf. \cite[Lemma 2.7]{keraani}). 
Thus, we see that \eqref{111} is under control. A similar approach (this time using the vector maximal inequality) handles \eqref{112}.

To estimate the contribution of \eqref{113} to the left-hand side of \eqref{decoupling to zero}, we need to estimate 
\begin{align} \norm{\big(\sum_N \big\vert N^{s_c}\sum_{K\leq N} \tfrac{K}{N}\M(f_K)\M(g\,\vert f\vert^{p-1})\big\vert^2\big)^{1/2}}_{N^0([0,\infty))}.\label{113 again}\end{align}
For this term, we need to make use of the following basic inequality: for a nonnegative sequence $\{a_K\}_{K\in 2^{\mathbb{Z}}}$ and $0<s<1$, one has 
					\begin{equation}
					\label{a general inequality}
					\sum_{N\in 2^{\mathbb{Z}}}N^{2s}\big\vert\sum_{K\leq N}\tfrac{K}{N}a_K\big\vert^2
					\lesssim \sum_{K\in 2^{\mathbb{Z}}} K^{2s}\vert a_K\vert^2
					\end{equation}
(cf. \cite[Lemma 4.2]{taylor}). Using this inequality, along with H\"older, we can estimate
\begin{align*} \eqref{113 again}&\lesssim  \xonorms{\big(\sum_K \big\vert K^{s_c}\M(f_K)\big\vert^2\big)^{1/2} \M(g\,\vert f\vert^{p-1})}{\frac{2(d+2)}{d+4}}
\\ &\lesssim \xonorms{\big(\sum_K \vert K^{s_c} \M(f_K)\vert^2\big)^{1/2}}{\frac{2(d+2)}{d}}\xonorms{\vert g\vert\,\vert f\vert^{p-1}}{\frac{d+2}{2}}\to 0
\end{align*} as $n\to\infty$, exactly as before. Thus, \eqref{113} is under control; the same approach handles \eqref{114} (after an application of the vector maximal inequality).

Let us now turn to \eqref{115}. As before, we sum over $N\in 2^{\mathbb{Z}}$ and find that we need to estimate
\begin{equation}\label{115again}\bxonorms{\left(\sum \big\vert N^{s_c} f_{>N}\big\vert^2\right)^{1/2}\M(\vert g\vert^p)}{\frac{2(d+2)}{d+4}}.\end{equation} 
Recalling that  $f=v_n^j$ and $g=v_n^k$ for some $j\neq k$, we see that we are once again in a position to use the argument from \cite{keraani}. 
To begin, we may assume without loss of generality that both
					$$\Phi_1:=\left(\sum\vert N^{s_c}P_{>N}v^j\vert^2\right)^{1/2}\quad\text{and}\quad \Phi_2:=M(\vert v^k\vert^p)$$
belong to $C_c^\infty(\R\times\R^d)$; indeed, $C_c^\infty$-functions are dense in both $L_{t,x}^{\frac{2(d+2)}{d}}$ and $L_{t,x}^{\frac{d+2}{2}}.$ We now wish to use the asymptotic orthogonality of parameters, that is,
				\begin{equation}
				\frac{\lambda_n^j}{\lambda_n^k}+\frac{\lambda_n^k}{\lambda_n^j}
				+\frac{\vert x_n^j-x_n^k\vert^2}{\lambda_n^j\lambda_n^k}
				+\frac{\vert t_n^j(\lambda_n^j)^2-t_n^k(\lambda_n^k)^2\vert}{\lambda_n^j\lambda_n^k}
				\to\infty\quad \text{as}\ n\to\infty,\label{aop}
				\end{equation} 
 to show \eqref{115again}$\to 0$. 

Consider first the case $\frac{\lambda_n^j}{\lambda_n^k}\to c> 0$ (along a subsequence, say). If we unravel the definition of the nonlinear profiles and change variables to move the symmetries onto $\Phi_2$, we arrive at
		\begin{align*}
		&\eqref{115again}^{\frac{2(d+2)}{d+4}} 
		\\ &=\left(\tfrac{\lambda_n^j}{\lambda_n^k}\right)^{\frac{4(d+2)}{d+4}}\iint\bigg\vert \Phi_1(s,y)\Phi_2(\big(t_n^k+\big(\tfrac{\lambda_n^j}{\lambda_n^k}\big)^2(s-t_n^j),\big(\tfrac{\lambda_n^j}{\lambda_n^k}\big)y+\tfrac{x_n^j-x_n^k}{\lambda_n^k}\big)\bigg\vert^{\frac{2(d+2)}{d+4}}\,dy\,ds.
		\end{align*}
Then, recalling \eqref{aop}, we see that as $n\to\infty$, either the spatial or temporal argument of $\Phi_2$ must escape the support of $\Phi_1$. Thus, in this case, we get \eqref{115again}$\to 0$.

If instead we have $\frac{\lambda_n^j}{\lambda_n^k}\to 0$, then continuing from above, we can estimate
			\begin{align*}
			\eqref{115again}&\lesssim \big(\tfrac{\lambda_n^j}{\lambda_n^k}\big)^2\xonorms{\Phi_1}{\frac{2(d+2)}{d+4}}\xonorms{\Phi_2}{\infty}.
			\end{align*}
As $\Phi_1,\Phi_2\in C_c^\infty(\R\times\R^d)$, we see that \eqref{115again}$\to 0$ in this case, as well. 

Finally, we can treat the case $\frac{\lambda_n^j}{\lambda_n^k}\to\infty$ just like the previous case; the only difference is that we change variables to move the symmetries onto $\Phi_1$, instead of $\Phi_2$. Thus, we have that \eqref{115again}$\to 0$ in this third and final case. 

We have now shown that \eqref{115} is under control. The same ideas can be used to handle \eqref{116}, \eqref{117}, and \eqref{118}.

As mentioned above, this same set of ideas suffices to deal with all the remaining terms stemming from \eqref{decouple 1}.
\end{proof}
\begin{proof}[Proof of \eqref{decouple 2}] For this term, we will need to make use \eqref{asymptotic vanishing}. As we will see, the terms in which $e^{it\Delta}w_n^J$ appears without derivatives will be relatively easy to handle, as \eqref{asymptotic vanishing} will apply directly. On the other hand, the terms that only contain $\nsc e^{it\Delta}w_n^J$ will require a more careful analysis; in particular, we will need to carry out a local smoothing argument before we can make effective use of \eqref{asymptotic vanishing}.

Defining $g:=\sum_{j=1}^Jv_n^j$ and $h:=e^{it\Delta}w_n^J$, we are left to show
 \begin{equation} \lim_{J\to\infty}\limsup_{n\to\infty}\norm{\nsc\left(\vert g+h\vert^p(g+h)-\vert g\vert^p g\right)}_{N^0([0,\infty))}=0. \label{decouple to zero 2}
 \end{equation}
We write
 				\begin{align} 
				\vert g+h\vert^p(g+h)-\vert g\vert^p g=&\ \vert g+h\vert^ph \label{22}
				 \\ &+(\vert g+h\vert^p-\vert g\vert^p)g \label{21}
				 \end{align}
and first restrict our attention to \eqref{22}. We proceed as before, working at a single frequency and exploiting cancellation to write
 			\begin{align}
 		\vert P_N(\vert g+h\vert^ph)(x)\vert&=\big\vert\smallint N^d\psih(Ny)\delta_y\big[\vert g(x)+h(x)\vert^ph(x)\big]\,dy			\big\vert \nonumber
			 \\ &\leq \smallint N^d\vert\psih(Ny)\vert\,\vert g(x-y)+h(x-y)\vert^p\,\vert\delta_y h(x)\vert\,dy \label{221}
 			\\ &\quad + \smallint N^d\vert\psih(Ny)\vert\,\vert\delta_y\big[\vert g(x)+h(x)\vert^p\big]\,\vert h(x)\vert\,dy.			\label{222}
 			\end{align}

We will deal only with \eqref{221}, which is the more difficult term. Indeed, in all of the terms that stem from \eqref{222}, we will have a copy of $e^{it\Delta}w_n^J$ appearing without derivatives, so that \eqref{asymptotic vanishing} will suffice. (For completeness, we will later show how to handle such a term; cf. \eqref{one such term} below.)

Proceeding as in \eqref{delta breakdown 1}, we write
\begin{align}
\eqref{221}&\lesssim \smallint N^d\vert\psih(Ny)\vert\,\vert g(x-y)+h(x-y)\vert^p\,\vert h_{>N}(x)\vert\,dy \label{2211}
\\ &+\smallint N^d\vert\psih(Ny)\vert\,\vert g(x-y)+h(x-y)\vert^p\,\vert h_{>N}(x-y)\vert\,dy \label{2212}
\\ &+\sum_{K\leq N}\smallint N^d\vert\psih(Ny)\vert\,\vert g(x-y)+h(x-y)\vert^p\,\vert\delta_y h_K(x)\vert\,dy.\label{2213}
\end{align}

Let us now deal only with \eqref{2213}; in doing so, we will see all of the ideas necessary to handle \eqref{2211} and \eqref{2212}, as well. We first write
\begin{align}
\eqref{2213}&\lesssim\sum_{K\leq N}\smallint N^d\vert\psih(Ny)\vert\,\vert g(x-y)\vert^p\vert\delta_y h_K(x)\vert\,dy \label{22131}
\\ &+\sum_{K\leq N}\smallint N^d\vert\psih(Ny)\vert\,\vert h(x-y)\vert^p\vert\delta_y h_K(x)\vert\,dy.\label{22132}
\end{align}

We only consider \eqref{22131}, as the contribution of \eqref{22132} is easier to estimate (again, due to the presence of $e^{it\Delta}w_n^J$ without derivatives). Employing the inequalities of Lemma \ref{aux lemma}, we find
\begin{align}
\eqref{22131} \lesssim \sum_{K\leq N}\tfrac{K}{N}\M(\vert g\vert^p)(x)\M(h_K)(x) +\sum_{K\leq N}\tfrac{K}{N}\M(\vert g\vert^p \M(h_K))(x).
\nonumber
\end{align}

Let us now concern ourselves only with the first term above, as the second is similar. As before, to estimate the contribution of this term to \eqref{decouple to zero 2} (and thereby complete our treatment of \eqref{22}), we need to sum over $N\in 2^{\mathbb{Z}}$. Using \eqref{a general inequality} and recalling the definitions of $g$ and $h$, we write
\begin{align}\norm{\big(&\sum_N\big\vert N^{s_c}\sum_{K\leq N}\tfrac{K}{N}\M(\vert g\vert^p)\M(h_K)\big\vert^2\big)^{1/2}}_{N^0}\nonumber
\\ &\lesssim \norm{\big(\sum_N\big\vert N^{s_c} \M(h_N)\big\vert^2\big)^{1/2}\M(\vert g\vert^p)}_{L_{t,x}^{\frac{2(d+2)}{d+4}}} \nonumber
\\ &\lesssim \norm{\big(\sum_N\big\vert N^{s_c}\M(P_N e^{it\Delta}w_n^J)\vert^2\big)^{1/2}M\big(\big\vert\sum_{j=1}^J v_n^j\big\vert^p\big)}_{L_{t,x}^{\frac{2(d+2)}{d+4}}}\nonumber.
\end{align}
Thus, to complete our treatment of \eqref{22}, we are left to show
						\begin{equation}
						\label{down to this}
				\lim_{J\to\infty}\limsup_{n\to\infty}
				\norm{\big(\sum_N\big\vert N^{s_c}\M(P_N e^{it\Delta}w_n^J)\vert^2\big)^{1/2}
				M\big(\big\vert\sum_{j=1}^J v_n^j\big\vert^p\big)}_{L_{t,x}^{\frac{2(d+2)}{d+4}}}^{\frac{d+2}{2}}=0.
					\end{equation}

To begin, we let $\eta>0$; then using \eqref{okkkk}, we see that there exists some $J_1=J_1(\eta)$ so that
					\begin{equation}
					\sum_{j\geq J_1}\xonorms{v_n^j}{\frac{p(d+2)}{2}}^{\frac{p(d+2)}{2}}<\eta.
					\nonumber
					\end{equation}
Using H\"older's inequality, maximal function and vector maximal function estimates, and Lemma~\ref{square function estimates}, we can argue as we did to obtain \eqref{ok version3} to see
		\begin{align*}
		\limsup_{n\to\infty} \norm{\big(\sum_N&\big\vert N^{s_c}\M(P_N e^{it\Delta}w_n^J)\vert^2\big)^{1/2}M\big(\big\vert\sum_{j\geq J_1} v_n^j\big\vert^p\big)}_{L_{t,x}^{\frac{2(d+2)}{d+4}}}^{\frac{d+2}{2}}
		\\ &\lesssim \limsup_{n\to\infty}\norm{\nsc e^{it\Delta}w_n^J}_{L_{t,x}^{\frac{2(d+2)}{d}}}^{\frac{d+2}{2}}\sum_{j\geq J_1}\xonorms{v_n^j}{\frac{p(d+2)}{2}}^{\frac{p(d+2)}{2}}
		\\ &\lesssim \eta.
		\end{align*}
As $\eta>0$ was arbitrary, we see that to establish \eqref{down to this}, it will suffice to show 
				\begin{equation}\label{single j}
				\lim_{J\to\infty}\limsup_{n\to\infty} \norm{\big(\sum_N \big\vert N^{s_c} 
				\M(P_N e^{it\Delta}w_n^J)\big\vert^2\big)^{1/2} \M(\vert v_n^j\vert^p)}_{L_{t,x}^{\frac{2(d+2)}{d+4}}}
				=0
				\end{equation}
for $1\leq j< J_1$. 

Restricting our attention to a single $j$ and recalling the definition of $v_n^j$, we change variables and find we need to estimate
$$\norm{\big(\sum_N \big\vert(\lambda_n^j)^{\frac{2}{p}}N^{s_c}\M P_N\big[e^{i[(\lambda_n^j)^2(t-t_n^j)]\Delta}w_n^J(\lambda_n^jx+x_n^j)\big]\big\vert^2\big)^{1/2}\M(\vert v^j\vert^p)}_{L_{t,x}^{\frac{2(d+2)}{d+4}}}.$$

We will now carry out some reductions, inspired by the proof of \cite[Proposition~3.4]{keraani}: as $\M(\vert v^j\vert^p)$ shares bounds with $\vert v^j\vert^p$, and $v^j$ obeys good bounds (it has scattering size $\lesssim E_c$), we may replace $\M(\vert v^j\vert^p)$ with some function $\Phi$ in $C_c^\infty(\R\times\R^d)$. If we then use H\"older's inequality, we find it suffices to estimate the first term in $L_{t,x}^2(K)$, where $K$ is the (compact) support of this function $\Phi$. The next step will be to use a local smoothing estimate on this (fixed) set $K$. Now, the norms that will appear in these estimates will have critical scaling; that is, they will be invariant under the change of variables that eliminates the parameters $\lambda_n^j$, $x_n^j,$ and $t_n^j$.  Thus, without loss of generality, we will ignore them from the start. 

To establish \eqref{single j} and complete our treatment of \eqref{22}, we are therefore left to show
\begin{equation}\label{weird local smoothing} \lim_{J\to\infty}\limsup_{n\to\infty} \norm{\big(\sum_N \big\vert \M(N^{s_c}P_N e^{it\Delta}w_n^J)\big\vert^2\big)^{1/2}}_{L_{t,x}^2(K)}=0
\end{equation}
for a fixed compact set $K\subset \R\times\R^d$. 

To establish \eqref{weird local smoothing}, we will need to rely on the fact that we are working on a compact set, so that we can carry out a local smoothing argument. Indeed, the term appearing above is morally like $\nsc e^{it\Delta}w_n^J$, over which we do not have sufficient control (cf. \eqref{weak decoupling}). However, we do have good control over $e^{it\Delta}w_n^J$, in the form of \eqref{asymptotic vanishing}. Thus, to succeed, we need to find a way to estimate the term above using fewer than $s_c$ derivatives; this is exactly the role of local smoothing.

For the proof of \eqref{weird local smoothing}, we will use a standard local smoothing result for the free propagator (Lemma \ref{local smoothing}), along with a few results from \cite[Chapter V]{stein}. In particular, we need the following: if we choose $\eps>0$ so that $-d<-1-\eps$, then $\vert x\vert^{-1-\eps}$ is an $A_2$ weight, so that $\M$ is bounded on $L^2(\vert x\vert^{-1-\eps}\,dx)$.

\begin{proof}[Proof of \eqref{weird local smoothing}] We can write $K\subset[-T,T]\times\{\vert x\vert\leq R\}$ for some $T,R>0$. We fix some $N_0\in 2^{\mathbb{Z}}$ and break into low and high frequencies:
\begin{align*} \iint_K \sum_N\big\vert N^{s_c}\M(P_N e^{it\Delta}w_n^J)\big\vert^2\,dx\,dt\lesssim &\sum_{N\leq N_0}\iint_K\big\vert \M(N^{s_c}P_N e^{it\Delta}w_n^J)\big\vert^2\,dx\,dt
\\ &+ \sum_{N>N_0}\iint_K \big\vert \M(N^{s_c}P_N e^{it\Delta}w_n^J)\big\vert^2\,dx\,dt.
\end{align*}
For the low frequencies, we use H\"older and maximal function estimates to write
\begin{align*} \sum_{N\leq N_0}\iint_K\big\vert &\M(N^{s_c}P_N e^{it\Delta}w_n^J)\big\vert^2\,dx\,dt 
\\ &\lesssim \sum_{N\leq N_0} T^{\frac{p(d+2)-4}{p(d+2)}}R^{\frac{d(p(d+2)-4)}{p(d+2)}}\xonorms{\M(N^{s_c}P_N e^{it\Delta}w_n^J)}{\frac{p(d+2)}{2}}^2
\\ &\lesssim_K \sum_{N\leq N_0}N^{2s_c}\xonorms{e^{it\Delta}w_n^J}{\frac{p(d+2)}{2}}^2
\\ &\lesssim_K N_0^{2s_c}\xonorms{e^{it\Delta}w_n^J}{\frac{p(d+2)}{2}}^2.
\end{align*}
For the high frequencies, we choose $\eps>0$ so that $-d<-1-\eps$. Then, using Lemma \ref{local smoothing}, Bernstein, and the fact that $\vert x\vert^{-1-\eps}\in A_2$, we can estimate 
\begin{align*}
\sum_{N>N_0}\iint_K \big\vert &\M(N^{s_c}P_N e^{it\Delta}w_n^J)\big\vert^2\,dx\,dt
\\ &\lesssim R^{1+\eps}\sum_{N>N_0}\int_\R\int_{\R^d}\big\vert \M(N^{s_c}P_N e^{it\Delta}w_n^J)\vert^2\langle x\rangle^{-1-\eps}\,dx\,dt
\\ & \lesssim_K \sum_{N>N_0}N^{2s_c}\int_\R\int_{\R^d}\big\vert P_N e^{it\Delta}w_n^J\big\vert^2\langle x\rangle^{-1-\eps}\,dx\,dt
\\ &\lesssim_K \sum_{N>N_0}N^{2s_c}\norm{\vert\nabla\vert^{-\frac12}P_N w_n^J}_{L_x^2(\R^d)}^2
\\ &\lesssim_K \sum_{N>N_0}N^{-1}\norm{\nsc w_n^J}_{L_x^2(\R^d)}^2
\\ &\lesssim_K N_0^{-1}\norm{\nsc w_n^J}_{L_x^2(\R^d)}^2.  
\end{align*} 
Optimizing in the choice of $N_0$ now yields
$$\norm{\big(\sum_N \big\vert \M(N^{s_c}P_N e^{it\Delta}w_n^J)\big\vert^2\big)^{1/2}}_{L_{t,x}^2(K)}\lesssim_K \norm{e^{it\Delta}w_n^J}_{L_{t,x}^{\frac{p(d+2)}{2}}}^{\frac{1}{2s_c+1}}\norm{w_n^J}_{\dot{H}_x^{s_c}(\R^d)}^{\frac{2s_c}{2s_c+1}},$$ which, by \eqref{asymptotic vanishing}, gives \eqref{weird local smoothing}. 
\end{proof}

We have now dealt with \eqref{22}, and so we finally turn to \eqref{21}. As usual, we first restrict our attention to a single frequency $N$. We have dealt with a term of this form before (cf. \eqref{single freq}); proceeding in exactly the same way, we arrive at

\begin{align}\big\vert &P_N\big([\vert g(x)+h(x)\vert^p-\vert g(x)\vert^p]g(x)\big)\big\vert\nonumber
\\ &\lesssim\smallint N^d\vert\psih(Ny)\vert\,\vert\delta_y g(x)\vert\,\vert h(x-y)\vert\big\{\vert g(x-y)\vert^{p-1}+\vert h(x-y)\vert^{p-1}\big\}\,dy \label{211}
\\ &+\smallint N^d\vert\psih(Ny)\vert\,\vert g(x)\vert\,\vert\delta_y g(x)\vert\,\vert h(x-y)\vert^{p-1}\,dy \label{212}
\\ &+\smallint N^d\vert\psih(Ny)\vert\,\vert g(x)\vert\,\vert\delta_y h(x)\vert\,\big\{\vert g(x)\vert^{p-1}+\vert h(x)\vert^{p-1}+\vert h(x-y)\vert^{p-1}\big\}\,dy,\label{213}
\end{align}
at least in the case $p\leq 2$ (as above, we will only consider this case). 

Note that all of the terms above are similar to terms we have handled before. Thus, we proceed in the same way, decomposing terms exactly as before. Whenever a term includes a copy of $e^{it\Delta}w_n^J$ without derivatives, things will be relatively straightforward, as one can rely on \eqref{asymptotic vanishing} (see \eqref{one such term} below for details); for the one term stemming from \eqref{213} in which $e^{it\Delta}w_n^J$ only appears with derivatives, we have to go through the same local smoothing argument given above (cf. the proof of \eqref{weird local smoothing}).

Thus, to conclude the proof of \eqref{decouple 2}, we will see how to estimate the contribution of the term 
				\begin{equation}\label{one such term}
\smallint N^d\vert\psih(Ny)\vert\,\vert\delta_y g(x)\vert\,\vert h(x-y)\vert\,\vert g(x-y)\vert^{p-1}\,dy.
				\end{equation}
Estimating $\vert\delta_y g(x)\vert$ as before, we find we need to bound the terms
\begin{align}
\M(h\vert g&\vert^{p-1})g_{>N}+\M(h\vert g\vert^{p-1}g_{>N}) \nonumber
\\ &+\sum_{K\leq N}\tfrac{K}{N}\M(h\vert g\vert^{p-1})\M(g_K)+\sum_{K\leq N}\M(h\vert g\vert^{p-1}\M(g_K)).\nonumber
\end{align}

Let us now see how to handle the contribution of the first term only, as the other three are similar. We begin by summing over $N\in 2^{\mathbb{Z}}$ and recalling the definitions of $g$ and $h$; then, using H\"older, maximal function estimates, and Lemma \ref{square function estimates}, we can argue as we did to obtain \eqref{ok version3} to see
				\begin{align}
				\lim_{J\to\infty}&\limsup_{n\to\infty} \norm{\big(\sum_N \big\vert N^{s_c}g_{>N}\big\vert^2\big)^{1/2}\M(h\vert g\vert^{p-1})}_{L_{t,x}^{\frac{2(d+2)}{d+4}}}^{\frac{p(d+2)}{2(p-1)}} \nonumber
			\\ &\lesssim\lim_{J\to\infty}\limsup_{n\to\infty} \label{the last thing}
			\norm{\nsc\big(\sum_{j=1}^J v_n^j\big)}^{\frac{p(d+2)}{2(p-1)}}_{L_{t,x}^\frac{2(d+2)}{d}}
			\norm{e^{it\Delta}w_n^J}_{L_{t,x}^{\frac{p(d+2)}{2}}}^{\frac{p(d+2)}{2(p-1)}}
			\sum_{j=1}^J \norm{v_n^j}_{L_{t,x}^{\frac{p(d+2)}{2}}}^{\frac{p(d+2)}{2}}.
			\end{align}
We turn to estimating the first term above. We first write
	\begin{align}
	\lim_{J\to\infty}&\limsup_{n\to\infty}\norm{\nsc\big(\sum_{j=1}^J v_n^j\big)}_{L_{t,x}^{\frac{2(d+2)}{d}}}^2 \nonumber
	\\&\lesssim\lim_{J\to\infty}\limsup_{n\to\infty}\bigg(\sum_{j=1}^J \norm{\nsc v_n^j}_{L_{t,x}^{\frac{2(d+2)}{d}}}^2+\sum_{j\neq k}\norm{\nsc v_n^j\nsc v_n^k}_{L_{t,x}^{\frac{d+2}{d}}}\bigg). \label{continue from here}
	\end{align}
Arguing as we did to obtain \eqref{yay keraani}, we immediately get that
	\begin{equation}\label{this part is fine}
	\lim_{J\to\infty}\limsup_{n\to\infty} \sum_{j\neq k}\norm{\nsc v_n^j\nsc v_n^k}_{L_{t,x}^{\frac{d+2}{d}}}=0.
	\end{equation}
Next, we let $\eta>0$; then, using \eqref{strong decoupling}, we can find $J(\eta)>0$ so that
		$$\sum_{j> J(\eta)}\norm{\nsc \phi^j}_{L_x^2}^2 <\eta.$$ 
Taking $\eta$ sufficiently small and applying a standard bootstrap argument, we find
		\begin{equation}\label{this part two}
		\sum_{j> J(\eta)}\norm{\nsc v_n^j}_{L_{t,x}^{\frac{2(d+2)}{d}}}^2\lesssim\sum_{j> J(\eta)}\norm{\nsc \phi^j}_{L_x^2}^2\lesssim\eta.
		\end{equation}
On the other hand, the fact that each $v_n^j$ has scattering size $\lesssim E_c$ implies
		\begin{equation}\label{this part three}
		\sum_{j=1}^{J(\eta)}\norm{\nsc v_n^j}_{L_{t,x}^{\frac{2(d+2)}{d}}}^2\lesssim_{E_c} 1.
		\end{equation}
Combining \eqref{this part is fine}, \eqref{this part two}, and \eqref{this part three}, we can continue from \eqref{continue from here} to see
		$$\lim_{J\to\infty}\limsup_{n\to\infty}\norm{\nsc\big(\sum_{j=1}^J v_n^j\big)}_{L_{t,x}^{\frac{2(d+2)}{d}}}^2\lesssim_{E_c} 1.$$ 

Thus, continuing from \eqref{the last thing} and using \eqref{okkkk} and \eqref{asymptotic vanishing}, we find  
$$\lim_{J\to\infty}\limsup_{n\to\infty}
			\norm{\nsc\big(\sum_{j=1}^J v_n^j\big)}^{\frac{p(d+2)}{2(p-1)}}_{L_{t,x}^{\frac{2(d+2)}{d}}}
			\norm{e^{it\Delta}w_n^J}_{L_{t,x}^{\frac{p(d+2)}{2}}}^{\frac{p(d+2)}{2(p-1)}}\sum_{j=1}^J \norm{v_n^j}_{L_{t,x}^{\frac{p(d+2)}{2}}}^{\frac{p(d+2)}{2}}=0,$$ as needed. This completes the proof of \eqref{decouple 2}.
\end{proof} 
Having established \eqref{decouple 1} and \eqref{decouple 2}, we are now done with the proof of Lemma \ref{decouple lemma}, as well as the sketch of the proof of Proposition \ref{palais smale}.
\section{Long-time Strichartz estimates}
In this section, we prove a long-time Strichartz estimate. Such estimates were first developed by Dodson \cite{dodson:3} in the study of the mass-critical NLS, but have since appeared in the energy-critical setting (see \cite{revisit2, revisit1}). In this paper, we establish a long-time Strichartz estimate for the first time in the inter-critical setting, modeling our approach after \cite{dodson:3, revisit2, revisit1}. The long-time Strichartz estimate will be an important technical tool in Section \ref{frequency-cascade section}, in which we rule out rapid frequency-cascade solutions, as well as in Section \ref{flim section}, in which we establish a frequency-localized interaction Morawetz inequality.

We will prove long-time Strichartz estimates for $(d,s_c)$ satisfying \eqref{constraints2}. This guarantees $p>1$, which simplifies the proof. Actually, as we will point out below, the same ideas can be used to handle $(d,s_c)=(5,\tfrac12)$, in which case $p=1$. 
 
\begin{proposition}[Long-time Strichartz estimates]\label{lts lemma} Take $(d,s_c)$ satisfying \eqref{constraints2}. Let $u:[0,T_{max})\times\R^d\to\mathbb{C}$ be an almost periodic solution to \eqref{nls} with $N(t)\equiv N_k\geq 1$ on each characteristic subinterval $J_k\subset[0,T_{max})$. Then on any compact time interval $I\subset[0,T_{max})$, which is a union of contiguous characteristic subintervals $J_k$, and for any $N>0$, we have

\begin{equation}\label{lts} \xnorm{\nsc u_{\leq N}}{2}{\frac{2d}{d-2}}{I}\lesssim_u 1+N^{2s_c-\frac12}K^{\frac12},
\end{equation}

\noindent where $K:=\int_I N(t)^{3-4s_c}\, dt.$ Moreover, for any $\varepsilon>0$, there exists $N_0=N_0(\varepsilon)$ such that for all $N\leq N_0$,

\begin{equation}\label{lts small} \xnorm{\nsc u_{\leq N}}{2}{\frac{2d}{d-2}}{I}\lesssim_u \varepsilon(1+N^{2s_c-\frac12}K^{\frac12}).
\end{equation}

\noindent We also note that the implicit constants in \eqref{lts} and \eqref{lts small} are independent of $I$.
\end{proposition} 
\begin{proof} Fix a compact interval $I\subset[0,T_{max})$, which is a contiguous union of characteristic subintervals $J_k$; throughout the proof, all spacetime norms will be taken over $I\times\R^d$ unless stated otherwise. Let $\eta_0>0$ and $\eta>0$ be small parameters to be chosen later, and note that by Remark~\ref{another consequence}, we may find $c=c(\eta)$ so that 

\begin{equation}\label{eta} \xonorm{\nsc u_{\leq cN(t)}}{\infty}{2}\leq\eta.
\end{equation}

\noindent For $N>0$, we define 
			$$A(N):=\xnorm{\nsc u_{\leq N}}{2}{\frac{2d}{d-2}}{I}
			\quad\text{and}\quad 
			A_k(N):=\xnorm{\nsc u_{\leq N}}{2}{\frac{2d}{d-2}}{J_k}$$
for an individual characteristic subinterval $J_k$. We first note that by Lemma \ref{spacetime bounds}, \eqref{lts} holds whenever $N\geq \sup_{J_k\subset I} N_k$. Indeed, in this case, we have
\begin{align*}A(N) & \lesssim_u 1+\left(\int_I N(t)^2\, dt\right)^{\frac12}
\\ &\lesssim_u 1+\left(\int_I N(t)^{3-4s_c} N^{4s_c-1}\, dt\right)^{\frac12}
\\ &\lesssim_u 1+N^{2s_c-\frac12}K^{\frac12}.
\end{align*} 
We will establish \eqref{lts} for arbitrary $N>0$ by induction, beginning by establishing a recurrence relation for $A(N)$:
\begin{lemma}[Recurrence relation for $A(N)$]\label{recurrence relation for A(N)}
\begin{align}\label{the recurrence}
A(N)&\lesssim_u \ \inf_{t\in I}\norm{\nsc u_{\leq N}(t)}_{L_x^2(\R^d)}+B(\eta,\eta_0)N^{2s_c-\frac12}K^{\frac12}\nonumber
\\ &\quad + \eta^\nu A\big(\tfrac{N}{\eta_0}\big)+\sum_{M>N/\eta_0} \left(\tfrac{N}{M}\right)^{\frac32s_c}A(M)
\end{align}
uniformly in $N$, for some positive constants $B(\eta,\eta_0)$ and $\nu$. 
\end{lemma}
\begin{proof}[Proof of Lemma \ref{recurrence relation for A(N)}] We first apply Strichartz to see
\begin{equation}\label{lts strichartz} A(N)\lesssim \inf_{t\in I}\norm{\nsc u_{\leq N}(t)}_{L_x^{2}}+\xonorm{\nsc P_{\leq N}(\vert u\vert^p u)}{2}{\frac{2d}{d+2}}.
\end{equation} 
Our next step is to decompose the nonlinearity $\vert u\vert^pu$ and estimate the resulting pieces; the particular decomposition we choose depends on the ambient dimension.

\textbf{Case 1.} When $d=3$, we have $2\leq p<4$, and we decompose as follows:
\begin{align}
 \vert u\vert^p u = &\ (\vert u\vert^p+\vert u\vert^{p-2}\bar{u}u_{\leq N/\eta_0})u_{> N/\eta_0} \nonumber
\\ &+\vert u\vert^{p-2}\bar{u}(P_{> cN(t)}u_{\leq N/\eta_0})u_{\leq N/\eta_0} \label{decomposition d=3}
\\ &+\vert u\vert^{p-2}\bar{u}(P_{\leq cN(t)}u_{\leq N/\eta_0})u_{\leq N/\eta_0}.\nonumber
\end{align}

To estimate the contribution of the first term on the right-hand side of \eqref{decomposition d=3} to \eqref{lts strichartz}, we let 
$$G:=\vert u\vert^p+\vert u\vert^{p-2}\bar{u}u_{\leq N/\eta_0}$$
and use Bernstein, Lemma \ref{paraproduct}, and H\"older to estimate 
\begin{align} \xonorm{\nsc P_{\leq N}(Gu_{> N/\eta_0})}{2}{6/5} & \lesssim N^{\frac32s_c}\xonorm{\vert\nabla\vert^{-\frac12s_c}(Gu_{> N/\eta_0})}{2}{6/5} \nonumber
 \\ &\lesssim N^{\frac32s_c}\xonorm{\vert\nabla\vert^{\frac12s_c} G}{\infty}{\frac{12p}{11p-4}}\xonorm{\vert\nabla\vert^{-\frac12s_c}u_{>N/\eta_0}}{2}{6}\nonumber
 \\ \label{paraproduct 3 start} &\lesssim \xonorm{\vert\nabla\vert^{\frac12s_c} G}{\infty}{\frac{12p}{11p-4}}\displaystyle\sum_{M>N/\eta_0} \left(\tfrac{N}{M}\right)^{\frac32s_c} A(M).
\end{align}
To estimate the contribution of the first term above, we first use the fractional chain rule and Sobolev embedding to see
$$\xonorm{\vert\nabla\vert^{\frac12s_c}\vert u\vert^p}{\infty}{\frac{12p}{11p-4}}\lesssim \xonorm{u}{\infty}{\frac{3p}{2}}^{p-1}\xonorm{\vert\nabla\vert^{\frac12s_c}u}{\infty}{\frac{12p}{3p+4}}\lesssim \xonorm{\nsc u}{\infty}{2}^p \lesssim_u 1,$$ while by the fractional product rule, the fractional chain rule, and Sobolev embedding we get
\begin{align*} 
\xonorm{\vert\nabla\vert^{\frac12s_c}&(\vert u\vert^{p-2}\bar{u} u_{\leq N/\eta_0})}{\infty}{\frac{12p}{11p-4}}
\\ &\lesssim\ \xonorm{u}{\infty}{\frac{3p}{2}}\xonorm{\vert\nabla\vert^{\frac12s_c}(\vert u\vert^{p-2}\bar{u})}{\infty}{\frac{12p}{11p-12}}+\xonorm{u}{\infty}{\frac{3p}{2}}^{p-1}\xonorm{\vert\nabla\vert^{\frac12s_c}u}{\infty}{\frac{12p}{3p+4}}
\\ &\lesssim\ \xonorm{\nsc u}{\infty}{2}\xonorm{u}{\infty}{\frac{3p}{2}}^{p-2}\xonorm{\vert \nabla\vert^{\frac12s_c}u}{\infty}{\frac{12p}{3p+4}}+\xonorm{\nsc u}{\infty}{2}^p
\\ &\lesssim\ \xonorm{\nsc u}{\infty}{2}^p
\\ &\lesssim_u\ 1.
\end{align*}
Thus, continuing from \eqref{paraproduct 3 start}, we see 
\begin{align} \xonorm{\nsc P_{\leq N}&\big((\vert u\vert^p+\vert u\vert^{p-2}\bar{u}u_{\leq N/\eta_0})u_{>N/\eta_0}\big)}{2}{6/5}\nonumber
\\ &\lesssim_u\!\! \sum_{M>N/\eta_0}\!\! \left(\tfrac{N}{M}\right)^{\frac32s_c}\! A(M).\label{paraproduct 3}
\end{align} 

Next, we turn to estimating the contribution of the second term in \eqref{decomposition d=3} to \eqref{lts strichartz}. We begin by restricting our attention to an individual $J_k\times\R^d$. Note that we only need to consider the case $cN_k\leq N/\eta_0$; in this case, we can use Bernstein, H\"older, Sobolev embedding, Lemma \ref{spacetime bounds}, and the fact that $s_c\geq\tfrac12$ to estimate
	\begin{align}
	\xonorm{\nsc P_{\leq N}&(\vert u\vert^{p-2}\bar{u}(P_{>cN_k}u_{\leq N/\eta_0})u_{\leq N/\eta_0}}{2}{6/5} \nonumber
	\\ &\lesssim N^{s_c}\xonorm{\vert u\vert^{p-2}\bar{u}(P_{>cN_k}u_{\leq N/\eta_0})u_{\leq N/\eta_0}}{2}{6/5} \nonumber
	\\ &\lesssim N^{s_c}\xonorm{u}{\infty}{\frac{3p}{2}}^{p-1}\xonorm{P_{>cN_k}u_{\leq N/\eta_0}}{4}{3}
	\xonorm{u_{\leq N/\eta_0}}{4}{\frac{6p}{4-p}} \nonumber
	\\ &\lesssim_u N^{s_c}(cN_k)^{-s_c}\xonorm{\nsc u_{\leq N/\eta_0}}{4}{3}^2 \nonumber
	\\ &\lesssim_u B(\eta,\eta_0)\left(\tfrac{N}{N_k}\right)^{2s_c-\frac12} \label{bilinear pain}
	\end{align}
for some positive constant $B(\eta,\eta_0)$. Summing the estimates \eqref{bilinear pain} over the characteristic subintervals $J_k\subset I$ then gives
\begin{align}
\xonorm{\nsc P_{\leq N}&(\vert u\vert^{p-2} \bar{u}(P_{> cN(t)}u_{\leq N/\eta_0})u_{\leq N/\eta_0})}{2}{6/5}\nonumber
\\ &\lesssim_u B(\eta,\eta_0)N^{2s_c-\frac12}K^{\frac12}.\label{bilinear 3}
\end{align} 
\noindent Before proceeding to the next term in \eqref{decomposition d=3}, we note that in obtaining estimate \eqref{bilinear pain}, we could have held onto the term $\xonorm{\nsc u_{\leq N/\eta_0}}{4}{3},$ which (by interpolation) we can estimate by 
	\begin{align*} 
	\xonorm{\nsc u_{\leq N/\eta_0}}{4}{3}&\lesssim\xonorm{\nsc u_{\leq N/\eta_0}}{\infty}{2}^{\frac12}\xonorm{\nsc u_{\leq N/\eta_0}}{2}{6}^{\frac12}
	\\ &\lesssim_u \xonorm{\nsc u_{\leq N/\eta_0}}{\infty}{2}^{\frac12}.
	\end{align*}
In this case, summing the estimates yields
\begin{align}\xonorm{\nsc P_{\leq N}&(\vert u\vert^{p-2} \bar{u}(P_{> cN(t)}u_{\leq N/\eta_0})u_{\leq N/\eta_0})}{2}{6/5} \nonumber
\\ &\label{bilinear 3 small}\lesssim_u \sup_{J_k\subset I} \xnorm{\nsc u_{\leq N/\eta_0}}{\infty}{2}{J_k}^{\frac12} B(\eta,\eta_0)N^{2s_c-\frac12}K^{\frac12}.
\end{align} 
This variant of \eqref{bilinear 3} will be important when we eventually need to exhibit smallness in \eqref{lts small}. 

To estimate the contribution of the final term in \eqref{decomposition d=3} to \eqref{lts strichartz}, we begin with an application of the fractional product rule and H\"older to see

\begin{align}\xonorm{\nsc &P_{\leq N}(\vert u\vert^{p-2}\bar{u}(P_{\leq cN(t)}u_{\leq N/\eta_0})u_{\leq N/\eta_0})}{2}{6/5}\nonumber
\\ \lesssim&\label{tricky1} \ \xonorm{\nsc (\vert u\vert^{p-2}\bar{u})}{\infty}{\frac{6p}{7p-8}}\xonorm{P_{\leq cN(t)}u_{\leq N/\eta_0}}{4}{\frac{6p}{4-p}}\xonorm{u_{\leq N/\eta_0}}{4}{\frac{6p}{4-p}}
\\ &\label{tricky3}+\xonorm{u}{\infty}{\frac{3p}{2}}^{p-1}\xonorm{\nsc P_{\leq cN(t)}u_{\leq N/\eta_0}}{4}{3}\xonorm{u_{\leq N/\eta_0}}{4}{\frac{6p}{4-p}} 
\\ &\label{tricky2}+\xonorm{u}{\infty}{\frac{3p}{2}}^{p-1}\xonorm{P_{\leq cN(t)}u_{\leq N/\eta_0}}{\infty}{\frac{3p}{2}}\xonorm{\nsc u_{\leq N/\eta_0}}{2}{6}.
\end{align}
We first note that by the fractional chain rule and Sobolev embedding, we get
$$\xonorm{\nsc(\vert u\vert^{p-2}\bar{u})}{\infty}{\frac{6p}{7p-8}}\lesssim \xonorm{u}{\infty}{\frac{3p}{2}}^{p-2}\xonorm{\nsc u}{\infty}{2}\lesssim_u 1.$$ 
Using Sobolev embedding, interpolation, and \eqref{eta}, we also see
\begin{align*}
\xonorm{P_{\leq cN(t)}u_{\leq N/\eta_0}}{4}{\frac{6p}{4-p}} &\lesssim \xonorm{\nsc P_{\leq cN(t)}u_{\leq N/\eta_0}}{4}{3}
\\ &\lesssim \xonorm{\nsc P_{\leq cN(t)}u_{\leq N/\eta_0}}{\infty}{2}^{\frac12}\xonorm{\nsc P_{\leq cN(t)}u_{\leq N/\eta_0}}{2}{6}^{\frac12}
\\ &\lesssim \eta^{\frac12} A\big(\tfrac{N}{\eta_0}\big)^{\frac12}.
\end{align*}
Estimating similarly gives $$\xonorm{u_{\leq N/\eta_0}}{4}{\frac{6p}{4-p}}\lesssim_u A\big(\tfrac{N}{\eta_0}\big)^{\frac12}.$$ 
Plugging these last three estimates into \eqref{tricky1}, \eqref{tricky3}, and \eqref{tricky2} and employing a few more instances of Sobolev embedding and \eqref{eta} finally gives 
\begin{equation}\label{tricky 3}\xonorm{\nsc P_{\leq N}(\vert u\vert^{p-2}\bar{u}(P_{\leq cN(t)}u_{\leq N/\eta_0})u_{\leq N/\eta_0})}{2}{6/5}\lesssim_u \eta^{\frac12}A\big(\tfrac{N}{\eta_0}\big).
\end{equation}
Collecting the estimates \eqref{paraproduct 3}, \eqref{bilinear 3}, and \eqref{tricky 3}, we see that in the case $d=3$, the estimate \eqref{lts strichartz} becomes
		\begin{align}
		A(N)\lesssim_u  & \inf_{t\in I}\norm{\nsc u_{\leq N}(t)}_{L_x^2(\R^d)}+B(\eta,\eta_0)N^{2s_c-\frac12}K^{\frac12}\nonumber
		\\ \label {lts recurrence 3} &\quad+\eta^{\frac12}A\big(\tfrac{N}{\eta_0}\big)
		+\sum_{M>N/\eta_0}\left(\tfrac{N}{M}\right)^{\frac32s_c}A(M).
		\end{align}
Comparing \eqref{lts recurrence 3} to \eqref{the recurrence}, we see that Lemma \ref{recurrence relation for A(N)} holds for $d=3$.

\textbf{Case 2.} In this case, we have $d\in\{4,5\}$ and $\tfrac{4}{d-1}\leq p<\tfrac{4}{d-2}$, with both inequalities strict for $d=5$. In particular, we have $1<p<2$. 

Again, we wish to decompose the nonlinearity and continue from \eqref{lts strichartz}. This time, we decompose as follows:
\begin{align}
\vert u\vert^p u & =  \vert u\vert^p u_{>N/\eta_0} \nonumber
\\ & \quad+\vert u_{>cN(t)}\vert^p P_{\leq cN(t)} u_{\leq N/\eta_0} \label{decomposition d=d} 
\\ & \quad+ \vert u_{>cN(t)}\vert^p P_{>cN(t)} u_{\leq N/\eta_0} \nonumber
\\ & \quad+ (\vert u\vert^p-\vert u_{>cN(t)}\vert^p)u_{\leq N/\eta_0}.\nonumber
\end{align}

We estimate the contribution of the first term on the right-hand side of \eqref{decomposition d=d} to \eqref{lts strichartz} similarly to the case $d=3$; in particular, by Bernstein, H\"older, and Lemma~\ref{paraproduct}, we have
\begin{align}
\xonorm{\nsc & P_{\leq N}(\vert u\vert^p u_{>N/\eta_0})}{2}{\frac{2d}{d+2}}\nonumber 
\\ &\lesssim N^{\frac32s_c}\xonorm{\vert\nabla\vert^{-\frac12s_c}(\vert u\vert^p u_{>N/\eta_0})}{2}{\frac{2d}{d+2}} \nonumber
\\ &\lesssim N^{\frac32s_c}\xonorm{\vert\nabla\vert^{\frac12s_c} \vert u\vert^p}{\infty}{\frac{4dp}{p(d+8)-4}}\xonorm{\vert\nabla\vert^{-\frac12s_c} u_{>N/\eta_0} }{2}{\frac{2d}{d-2}}
\nonumber
\\ &\lesssim  \xonorm{\vert\nabla\vert^{\frac12s_c} \vert u\vert^p}{\infty}{\frac{4dp}{p(d+8)-4}} \sum_{M>N/\eta_0} \left(\tfrac{N}{M}\right)^{\frac32s_c}A(M). \label{paraproduct d start}
\end{align} 
As we can use the fractional chain rule and Sobolev embedding to estimate
$$ \xonorm{\vert\nabla\vert^{\frac12s_c} \vert u\vert^p}{\infty}{\frac{4dp}{p(d+8)-4}}\lesssim \xonorm{u}{\infty}{\frac{dp}{2}}^{p-1}\xonorm{\vert\nabla\vert^{\frac12s_c} u}{\infty}{\frac{4dp}{dp+4}}\lesssim \xonorm{\nsc u}{\infty}{2}^p\lesssim_u 1,$$ we can continue from \eqref{paraproduct d start} to get 
\begin{equation}\label{paraproduct d} \xonorm{\nsc P_{\leq N}(\vert u\vert^p u_{>N/\eta_0})}{2}{\frac{2d}{d+2}}\lesssim_u \sum_{M>N/\eta_0} \left(\tfrac{N}{M}\right)^{\frac32s_c}A(M).
\end{equation}

Next, we turn to estimating the second term in \eqref{decomposition d=d} to \eqref{lts strichartz}. Restricting our attention to an individual characteristic subinterval $J_k$, we first apply Bernstein, H\"older, and the fractional product rule to see
\begin{align}\xonorm{\nsc &P_{\leq N}(\vert u_{>cN_k}\vert^p P_{\leq cN_k}u_{\leq N/\eta_0})}{2}{\frac{2d}{d+2}}\nonumber
\\ &\lesssim N^{s_c-\frac14}\xonorm{\vert\nabla\vert^{\frac14}(\vert u_{>cN_k}\vert^p P_{\leq cN_k}u_{\leq N/\eta_0})}{2}{\frac{2d}{d+2}} \nonumber
\\ &\lesssim N^{s_c-\frac14}\xonorm{\vert\nabla\vert^{\frac14}\vert u_{>cN_k}\vert^p}{4}{\frac{2dp}{p(d+3)-4}}\xonorm{P_{\leq cN_k}u_{\leq N/\eta_0}}{4}{\frac{2dp}{4-p}} \label{young 1}
\\ &\ \ \ +\!N^{s_c-\frac14}\xonorm{u_{>cN_k}}{4p}{\frac{4dp^2}{p(2d+5)-8}}^p\xonorm{\vert\nabla\vert^{\frac14}P_{\leq cN_k}u_{\leq N/\eta_0}}{4}{\frac{4dp}{8-p}}.\label{young 2}
\end{align}
Using H\"older, the fractional chain rule, Sobolev embedding, Bernstein, interpolation, \eqref{eta}, and Young's inequality, we can estimate
\begin{align*}
\eqref{young 1} &\lesssim N^{s_c-\frac14}\xonorm{u_{>cN_k}}{\infty}{\frac{dp}{2}}^{p-1}\xonorm{\vert\nabla\vert^{\frac14}u_{>cN_k}}{4}{\frac{2d}{d-1}}\xonorm{\nsc P_{\leq cN_k}u_{\leq N/\eta_0}}{4}{\frac{2d}{d-1}}
\\ &\lesssim_u N^{s_c-\frac14}(cN_k)^{\frac14-s_c}\xonorm{\nsc u_{>cN_k}}{4}{\frac{2d}{d-1}}
\\ &\quad\times\xonorm{\nsc P_{\leq cN_k}u_{\leq N/\eta_0}}{\infty}{2}^{\frac12}\xonorm{\nsc u_{\leq N/\eta_0}}{2}{\frac{2d}{d-2}}^{\frac12}
\\ &\lesssim_u B(\eta)\left(\tfrac{N}{N_k}\right)^{s_c-\frac14}\eta^{\frac12}A_k\big(\tfrac{N}{\eta_0}\big)^{\frac12}
\\ &\lesssim_u B(\eta)\left(\tfrac{N}{N_k}\right)^{2s_c-\frac12}+\eta A_k\big(\tfrac{N}{\eta_0}\big),
\end{align*} 
for some positive constant $B(\eta)$. Using Lemma \ref{spacetime bounds} as well, we can estimate similarly 
\begin{align*}
\eqref{young 2}&\lesssim N^{s_c-\frac14}(cN_k)^{\frac14-s_c}\xonorm{\vert\nabla\vert^{(s_c-\frac14)/p}u_{>cN_k}}{4p}{\frac{4dp^2}{p(2d+5)-8}}^p
\\ &\quad \times\xonorm{\nsc P_{\leq cN_k}u_{\leq N/\eta_0}}{4}{\frac{2d}{d-1}}
\\ &\lesssim B(\eta)\left(\tfrac{N}{N_k}\right)^{s_c-\frac14}\xonorm{\nsc u_{>cN_k}}{4p}{\frac{2dp}{dp-1}}^p
\\ &\quad \times\xonorm{\nsc P_{\leq cN_k} u_{\leq N/\eta_0}}{\infty}{2}^{\frac12}\xonorm{\nsc u_{\leq N/\eta_0}}{2}{\frac{2d}{d-2}}^{\frac12}
\\ &\lesssim_u B(\eta)\left(\tfrac{N}{N_k}\right)^{2s_c-\frac12}+\eta A_k\big(\tfrac{N}{\eta_0}\big).
\end{align*}
Collecting the estimates for \eqref{young 1} and \eqref{young 2} and summing over the intervals $J_k\subset I$, we arrive at 
\begin{align}
\xonorm{\nsc P_{\leq N}&(\vert u_{>cN(t)}\vert^p  P_{\leq cN(t)}u_{\leq N/\eta_0})}{2}{\frac{2d}{d+2}}\nonumber
\\ &\lesssim_uB(\eta)N^{2s_c-\frac12}K^{\frac12}+\eta A\big(\tfrac{N}{\eta_0}\big).\label{bilinear1 d=d}
\end{align}
Before proceeding, we note that for both \eqref{young 1} and \eqref{young 2}, we could have instead estimated
\begin{align*} \xnorm{\vert\nabla\vert^{s_c}&P_{\leq cN_k} u_{\leq N/\eta_0}}{\infty}{2}{J_k}^{\frac12}
\\ & \lesssim \xnorm{\nsc u_{\leq cN_k}}{\infty}{2}{J_k}^{\frac14}\xnorm{\nsc u_{\leq N/\eta_0}}{\infty}{2}{J_k}^{\frac14}
\\ & \lesssim \eta^{\frac14} \xnorm{\nsc u_{\leq N/\eta_0}}{\infty}{2}{J_k}^{\frac14}.
\end{align*} 
If we had done this, upon summing we could have ended up with the alternate estimate 
\begin{align}\xonorm{\vert\nabla&\vert^{s_c} P_{\leq N}(\vert u_{>cN(t)}\vert^p P_{\leq cN(t)}u_{\leq N/\eta_0})}{2}{\frac{2d}{d+2}}\nonumber
\\ &
\lesssim \sup_{J_k\subset I}\xnorm{\nsc u_{\leq N/\eta_0}}{\infty}{2}{J_k}^{\frac12} B(\eta)N^{2s_c-\frac12}K^{\frac12}+\eta^{\frac12} A\big(\tfrac{N}{\eta_0}\big).\label{bilinear1 d=d smallness}
\end{align}
This variant of \eqref{bilinear1 d=d} will be important when we need to exhibit smallness in \eqref{lts small}. 

To estimate the contribution of the third term in \eqref{decomposition d=d} to \eqref{lts strichartz}, we first define the following:
$$\left\{\begin{array}{ll} \theta:=\tfrac{dp-4-p}{4-p}\in[0,1), & \sigma:=\tfrac{p^2(d^2+2d-2)-4p(4d+1)+48}{4p(dp-8)}\in(0,s_c),
\\ \\ r_1:=\tfrac{4dp(dp-8)}{p^2(d^2-2d-2)+p(28-8d)-16}, & r_2:=\tfrac{4dp(dp-8)}{p^2(d^2+2d-2)-4p(2d+1)-16}.
\end{array}\right.$$ 
With this choice of parameters, we have $$\left\{\begin{array}{ll} s_c+\theta(\tfrac{d-1}{2}-s_c)=2s_c-\tfrac12, \\ \\  -\theta(s_c+\tfrac12)-2\sigma(1-\theta)=-(2s_c-\tfrac12)
\end{array}\right.$$ 
and (by Sobolev embedding)
$$\begin{array}{ll}
\dot{H}^{s_c,\frac{2d}{d-2}}\hookrightarrow \dot{H}^{\sigma,r_1}, & \dot{H}^{s_c,2}\hookrightarrow\dot{H}^{\sigma,r_2}.
\end{array}$$

\noindent Then restricting our attention to an individual $J_k$, we can use Bernstein, H\"older, the bilinear Strichartz estimate (Lemma \ref{bilinear strichartz}), and Sobolev embedding to estimate
\begin{align}\xonorm{\nsc &P_{\leq N}(\vert u_{>cN_k}\vert^pP_{>cN_k}u_{\leq N/\eta_0})}{2}{\frac{2d}{d+2}}\nonumber
\\ &\lesssim N^{s_c}\xonorm{u_{>cN_k}}{\infty}{\frac{dp}{2}}^{p-1}\xonorms{u_{> cN_k}P_{>cN_k}u_{\leq N/\eta_0}}{2}^\theta\nonumber
\\ &\quad \times\xonorm{u_{>cN_k}}{2}{r_1}^{1-\theta}\xonorm{P_{>cN_k}u_{\leq N/\eta_0}}{\infty}{r_2}^{1-\theta} \nonumber
\\ &\lesssim_u N^{s_c}\left(\tfrac{N}{\eta_0}\right)^{\theta(\frac{d-1}{2}-s_c)}(cN_k)^{-\theta(s_c+\frac12)}\nonumber
\\ & \quad \times\xonorm{u_{>cN_k}}{2}{r_1}^{1-\theta}\xonorm{P_{> cN_k}u_{\leq N/\eta_0}}{\infty}{r_2}^{1-\theta} \nonumber
\\ &\lesssim_u B(\eta_0)N^{2s_c-\frac12}(cN_k)^{-\theta(s_c+\frac12)-2\sigma(1-\theta)}\nonumber
\\ & \quad \times\xonorm{\vert\nabla\vert^{\sigma}u_{>cN_k}}{2}{r_1}^{1-\theta}\xonorm{\vert\nabla\vert^{\sigma}P_{>cN_k}u_{\leq N/\eta_0}}{\infty}{r_2}^{1-\theta} \nonumber
\\ &\lesssim_u B(\eta,\eta_0)\left(\tfrac{N}{N_k}\right)^{2s_c-\frac12}\xonorm{\nsc u_{>cN_k}}{2}{\frac{2d}{d-2}}^{1-\theta}\xonorm{\nsc u_{\leq N/\eta_0}}{\infty}{2}^{1-\theta} \nonumber
\\ &\lesssim_u B(\eta,\eta_0)\left(\tfrac{N}{N_k}\right)^{2s_c-\frac12}\label{bilinear d=d start}
\end{align}
for some positive constant $B(\eta,\eta_0)$. If we sum the estimates \eqref{bilinear d=d start} over the intervals $J_k\subset I$, we arrive at
\begin{equation}\label{bilinear d=d}
\xonorm{\nsc P_{\leq N}(\vert u_{>cN(t)}\vert^pP_{>cN(t)}u_{\leq N/\eta_0})}{2}{\frac{2d}{d+2}}\lesssim_u B(\eta,\eta_0)N^{2s_c-\frac12}K^{\frac12}.
\end{equation}

Before moving on to the fourth (and final) term in \eqref{decomposition d=d}, we note that if we had held on to the term $\xonorm{\nsc u_{\leq N/\eta_0}}{\infty}{2}^{1-\theta}$ when deriving \eqref{bilinear d=d start}, then upon summing we would get
\begin{align}\xonorm{\nsc &P_{\leq N}(\vert u_{>cN(t)}\vert^pP_{>cN(t)}u_{\leq N/\eta_0})}{2}{\frac{2d}{d+2}}
\nonumber
\\ &\lesssim_u \sup_{J_k\subset I}\xnorm{\nsc u_{\leq N/\eta_0}}{\infty}{2}{J_k}^{1-\theta} B(\eta,\eta_0)N^{2s_c-\frac12}K^{\frac12}.\label{bilinear d=d smallness}
\end{align}
This variant of \eqref{bilinear d=d} will be important when we eventually need to exhibit smallness in \eqref{lts small}.

We now turn to the final term in \eqref{decomposition d=d}, beginning with an application of the fractional product rule and H\"older:
\begin{align}\xonorm{\nsc& P_{\leq N}\left((\vert u\vert^p -\vert u_{>cN(t)}\vert^p)u_{\leq N/\eta_0}\right)}{2}{\frac{2d}{d+2}}\nonumber
\\ &\lesssim \xonorm{\nsc(\vert u\vert^p-\vert u_{>cN(t)}\vert^p)}{\infty}{\frac{2dp}{p(d+4)-4}}\xonorm{u_{\leq N/\eta_0}}{2}{\frac{dp}{2-p}} \label{differences 1}
\\ &\quad +\xonorm{\vert u\vert^p -\vert u_{>cN(t)}\vert^p}{\infty}{\frac{d}{2}}\xonorm{\nsc u_{\leq N/\eta_0}}{2}{\frac{2d}{d-2}}\label{differences 2}.
\end{align}
By Lemma \ref{derivatives of differences}, Sobolev embedding, and \eqref{eta}, we first estimate
\begin{align*} 
\eqref{differences 1}&\lesssim\xonorm{\nsc u_{>cN(t)}}{\infty}{2}\xonorm{u_{\leq cN(t)}}{\infty}{\frac{dp}{2}}^{p-1}A\big(\tfrac{N}{\eta_0}\big)
\\ &\quad +\xonorm{\nsc u_{\leq cN(t)}}{\infty}{2}\xonorm{u}{\infty}{\frac{dp}{2}}^{p-1}A\big(\tfrac{N}{\eta_0}\big)
\\ &\lesssim_u (\eta^{p-1}+\eta)A\big(\tfrac{N}{\eta_0}\big).
\end{align*} 
On the other hand, by Sobolev embedding, H\"older, and \eqref{eta}, we get
\begin{align*} 
\eqref{differences 2}&\lesssim \big(\xonorm{u}{\infty}{\frac{dp}{2}}^{p-1}+\xonorm{u_{>cN(t)}}{\infty}{\frac{dp}{2}}^{p-1}\big)\xonorm{u_{\leq cN(t)}}{\infty}{\frac{dp}{2}} A\big(\tfrac{N}{\eta_0}\big)\lesssim_u \eta A\big(\tfrac{N}{\eta_0}\big).
\end{align*}
Thus we can estimate the contribution of the final term in \eqref{decomposition d=d} by
\begin{equation}\label{differences term}
\xonorm{\nsc P_{\leq N}\left((\vert u\vert^p -\vert u_{>cN(t)}\vert^p)u_{\leq N/\eta_0}\right)}{2}{\frac{2d}{d+2}}\lesssim_u \eta^{p-1}A\big(\tfrac{N}{\eta_0}\big).
\end{equation}
Collecting the estimates \eqref{paraproduct d}, \eqref{bilinear1 d=d}, \eqref{bilinear d=d}, and \eqref{differences term}, we see that in Case 2, the estimate \eqref{lts strichartz} becomes
\begin{align}
A(N)&\lesssim_u \inf_{t\in I}\norm{\nsc u_{\leq N}(t)}_{L_x^2(\R^d)}+B(\eta,\eta_0)N^{2s_c-\frac12}K^{\frac12}\nonumber
\\ & \quad +\eta^{\min\{\frac12,p-1\}}A\big(\tfrac{N}{\eta_0}\big)+\sum_{M>N/\eta_0} \left(\tfrac{N}{M}\right)^{\frac32s_c}A(M). \label{lts recurrence d}
\end{align}
Comparing \eqref{lts recurrence d} to \eqref{the recurrence}, we see that Lemma \ref{recurrence relation for A(N)} holds for $d\in\{4,5\}$. \end{proof}
\begin{remark} We have omitted the case $(d,s_c)=(5,\tfrac12)$, in which $p=1$; this scenario is not handled under Case 2 due to the use of Lemma \ref{derivatives of differences}. However, by using the alternate decomposition
$$\vert u\vert u=\vert u\vert u_{>N/\eta_0}+(\vert u\vert-\vert u_{\leq cN(t)}\vert)u_{\leq N/\eta_0}+\vert u_{\leq cN(t)}\vert u_{\leq N/\eta_0},$$ 
one can use the same ideas as above to establish the recurrence relation in this case. 
\end{remark}
With the recurrence relation \eqref{the recurrence} in hand, we can now use induction to complete the proof of Proposition \ref{lts lemma}. First, recall that \eqref{lts} holds for $N\geq \sup_{J_k\subset I}N_k.$; i.e. we have 
\begin{equation}\label{lts C}
A(N)\leq C(u)\left[1+N^{2s_c-1/2}K^{1/2}\right]
\end{equation} for $N\geq \sup_{J_k\subset I}N_k.$ Of course, this inequality remains true if we replace $C(u)$ by any larger constant. 

We now suppose \eqref{lts C} holds at frequency $N$ and use the recurrence relation \eqref{the recurrence} to show it holds at frequency $N/2$. Let us first rewrite \eqref{the recurrence} as
\begin{align}
A(N)& \leq \wt{C}(u)\big[1+B(\eta,\eta_0)N^{2s_c-\frac12}K^{\frac12} +\eta^{\nu}A(\tfrac{N}{\eta_0})+\!\!\!\!\!\!\sum_{M>N/\eta_0}\!\!\!\left(\tfrac{N}{M}\right)^{\frac32s_c}A(M)\big].\label{the recurrence C}
\end{align}
To simplify notation, we will let $\alpha:=2s_c-\frac12$ and write $B(\eta,\eta_0)=B$; then, if we take $\eta_0<\tfrac12$ and use the inductive hypothesis, \eqref{the recurrence C} becomes
\begin{align}
A(\tfrac{N}{2})&\leq \wt{C}(u)\big[1+B(\tfrac{N}{2})^{\alpha}K^{\frac12}+\eta^\nu C(u)(1+\eta_0^{-\alpha}(\tfrac{N}{2})^\alpha K^{\frac12})\nonumber
\\ & \quad+C(u)\!\!\!\!\!\!\sum_{M>N/2\eta_0}\!\!\!\left(\tfrac{N}{2M}\right)^{\frac32s_c}(1+M^\alpha K^{\frac12})\big] \nonumber
\\ &\leq \wt{C}(u)\big[1+B(\tfrac{N}{2})^\alpha K^{\frac12}+\eta^\nu C(u)(1+\eta_0^{-\alpha}(\tfrac{N}{2})^\alpha K^{\frac12})\nonumber
\\ & \quad+C(u)\eta_0^{\frac32s_c}+C(u)\eta_0^{\frac12(1-s_c)}(\tfrac{N}{2})^\alpha K^{\frac12}\big] \nonumber
\\ &= \wt{C}(u)\left[1+B(\tfrac{N}{2})^\alpha K^{\frac 12}\right]+C(u)\big[(\eta^\nu+\eta_0^{\frac32s_c})\wt{C}(u)\nonumber
\\ &\quad+\big(\eta_0^{-\alpha}\eta^\nu+\eta_0^{\frac12(1-s_c)}\big)\wt{C}(u)(\tfrac{N}{2})^\alpha K^{\frac12}\big].\label{induction term}
\end{align}
Notice that we had convergence of the sum above precisely because $s_c<1$. If we choose $\eta_0$ possibly even smaller depending on $\wt{C}(u)$, and $\eta$ sufficiently small depending on $\wt{C}(u)$ and $\eta_0$, we can guarantee
\begin{align*} \eqref{induction term}\leq \wt{C}(u)\left[1+B(\eta,\eta_0)(\tfrac{N}{2})^\alpha K^{\frac 12}\right]+\tfrac12 C(u)\left[1+(\tfrac{N}{2})^\alpha K^{\frac12}\right].\end{align*}
If we now choose $C(u)$ possibly larger so that $C(u)\geq 2(1+B(\eta,\eta_0))\wt{C}(u)$, then this inequality implies that \eqref{lts C} holds at $N/2$, as was needed to show. This completes the proof of \eqref{lts}. 

It remains to establish \eqref{lts small}. To begin, fix $\varepsilon>0$. To exhibit the smallness in \eqref{lts small}, we need to revisit the proof of the recurrence relation for $A(N)$, paying closer attention to the terms that gave rise to the expression $N^{2s_c-\frac12}K^{\frac12}.$ More precisely, if we use \eqref{bilinear 3 small} instead of \eqref{bilinear 3}; \eqref{bilinear1 d=d smallness} instead of \eqref{bilinear1 d=d}; and \eqref{bilinear d=d smallness} instead of \eqref{bilinear d=d}; 
then continuing from \eqref{lts strichartz}, the recurrence relation for $A(N)$ takes the form
\begin{align}
A(N)\lesssim_u &\ f(N)+f(N)N^{2s_c-\frac12}K^{\frac12}+\eta^\nu A(\tfrac{N}{\eta_0})+\sum_{M>N/\eta_0}(\tfrac{N}{M})^{\frac32s_c}A(M),\label{small recurrence}
\end{align}
where $f(N)$ has the form 
\begin{align} 
f(N)=\  &\xnorm{\nsc u_{\leq N}}{\infty}{2}{I}\nonumber
\\ &+B(\eta,\eta_0)\sum_{i=1}^4\sup_{J_k\subset I} \xnorm{\nsc u_{\leq N/\eta_0}}{\infty}{2}{J_k}^{\theta_i}\label{vanishing}
\end{align} for some $\theta_i\in(0,1]$.  Here the particular values of the $\theta_i$ are not important; we will only need the fact that each $\theta_i>0$. Combining the updated recurrence relation \eqref{small recurrence} with the newly proven estimate \eqref{lts} and once again simplifying notation via $\alpha=2s_c-\frac12$, we see
\begin{align} \nonumber A(N)&\lesssim_u f(N)+f(N)N^{\alpha}K^{\frac12}+\eta^\nu(1+\eta_0^{-\alpha}N^{\alpha}K^{\frac12})+\eta_0^{\frac32s_c}(1+\eta_0^{-\alpha}N^\alpha K^{\frac12})
\\ \label{almost there} &\lesssim_u f(N)+\eta^\nu+\eta_0^{\frac32s_c}+\left[f(N)+\eta^\nu\eta_0^{-\alpha}+\eta_0^{\frac12(1-s_c)}\right]N^\alpha K^{\frac12}.
\end{align} 

To complete the argument, we will need the fact that for fixed $\eta,\eta_0>0$, we have 
\begin{equation}\label{vanishingg} \lim_{N\to 0} f(N)=0,
\end{equation} 
which is a consequence of almost periodicity and the fact that $$\inf_{t\in[0,T_{max})}N(t)\geq 1.$$ 

Then, continuing from \eqref{almost there}, we choose $\eta_0$ small enough that $\eta_0^{\frac32s_c}+\eta_0^{\frac12(1-s_c)}<\varepsilon$, and choose $\eta$ sufficiently small depending on $\eta_0$ so that $\eta^\nu+\eta_0^{-\alpha}\eta^\nu<\varepsilon$. Finally, using \eqref{vanishingg}, we choose $N_0=N_0(\varepsilon)$ so that $f(N)< \varepsilon$ for $N\leq N_0$. With this choice of parameters, \eqref{almost there} becomes $$A(N)\lesssim_u \varepsilon(1+N^{2s_c-\frac12}K^{\frac12})$$ for $N\leq N_0,$ which completes the proof of \eqref{lts small}. \end{proof}
\section{The rapid frequency-cascade scenario}\label{frequency-cascade section}
In this section, we preclude the existence of almost periodic solutions as in Theorem \ref{special scenarios} for which $\int_0^{T_{max}}N(t)^{3-4s_c}\,dt<\infty.$ We show that their existence is inconsistent with the conservation of mass. The main tool we will use is the long-time Strichartz estimate established in the previous section; as such, we will prove the following result for $(d,s_c)$ satisfying \eqref{constraints2}. 
\begin{theorem}[No rapid frequency-cascades]\label{frequency cascade} Let $(d,s_c)$ satisfy \eqref{constraints2}. Then there are no almost periodic solutions $u:[0,T_{max})\times\R^d\to\mathbb{C}$ to \eqref{nls} with $N(t)\equiv N_k\geq 1$ on each characteristic subinterval $J_k\subset[0,T_{max})$ such that 
\begin{equation}\xnorms{u}{\frac{p(d+2)}{2}}{[0,T_{max})}=\infty\label{u blows up}
\end{equation}
and
\begin{equation}\label{K is finite} \int_0^{T_{max}} N(t)^{3-4s_c}\, dt<\infty.
\end{equation} 
\end{theorem}
\begin{proof} We argue by contradiction. Suppose $u$ were such a solution; then by Corollary \ref{N at blowup}, we have $$\lim_{t\to T_{max}} N(t)=\infty,$$ whether $T_{max}$ is finite or infinite (recall $s_c>\tfrac14$). Thus by Remark~\ref{another consequence}, we see
\begin{equation}\label{inf to zero}\lim_{t\to T_{max}} \norm{\nsc u_{\leq N}(t)}_{L_x^2(\R^d)}=0\indent \text{for any }N>0.
\end{equation} 
\noindent Now, we let $I_n$ be a nested sequence of compact subintervals of $[0,T_{max})$, each of which is a contiguous union of characteristics intervals $J_k$. On each $I_n$, we will now apply Proposition \ref{lts lemma}; specifically, for fixed $\eta,\eta_0>0$, we use the recurrence relation \eqref{the recurrence}, the estimate \eqref{lts}, and the hypothesis \eqref{K is finite} to see
\begin{align*}
A_n(N)&:= \xnorm{\nsc u_{\leq N}}{2}{\frac{2d}{d-2}}{I_n}
\\ &\lesssim_u \inf_{t\in I_n} \norm{\nsc u_{\leq N}(t)}_{L_x^2(\R^d)} +B(\eta,\eta_0)N^{2s_c-\frac12}\left(\int_{I_n} N(t)^{3-4s_c}\, dt\right)^{\frac12}
\\ &\quad +\sum_{M>N/\eta_0}\left(\tfrac{N}{M}\right)^{\frac32 s_c}A_n(M)
\\ & \lesssim_u \inf_{t\in I_n} \norm{\nsc u_{\leq N}(t)}_{L_x^2(\R^d)} +B(\eta,\eta_0)N^{2s_c-\frac12}\bigg(\int_0^{T_{max}} N(t)^{3-4s_c}\, dt\bigg)^{\frac12}
\\ &\quad +\sum_{M>N/\eta_0}\left(\tfrac{N}{M}\right)^{\frac32 s_c}A_n(M)
\\ & \lesssim_u \inf_{t\in I_n} \norm{\nsc u_{\leq N}(t)}_{L_x^2(\R^d)} +B(\eta,\eta_0)N^{2s_c-\frac12}+\sum_{M>N/\eta_0}\left(\tfrac{N}{M}\right)^{\frac32 s_c}A_n(M).
\end{align*}
Arguing as we did to obtain \eqref{lts}, we conclude
$$A_n(N)\lesssim_u \inf_{t\in I_n} \norm{\nsc u_{\leq N}(t)}_{L_x^2(\R^d)} +N^{2s_c-\frac12}.$$ 
Letting $n\to\infty$ and using \eqref{inf to zero} then gives
\begin{equation}\label{good estimate for endpoint}
\xnorm{\nsc u_{\leq N}}{2}{\frac{2d}{d-2}}{[0,T_{max})}\lesssim_u N^{2s_c-\frac12}\indent\text{for all }N>0.
\end{equation}

We now claim that \eqref{good estimate for endpoint} implies
\begin{lemma}\label{freq lemma 1}
\begin{equation}\label{good estimate for infty}
\xnorm{\nsc u_{\leq N}}{\infty}{2}{[0,T_{max})}\lesssim_u N^{2s_c-\frac12}\indent\text{for all }N>0.
\end{equation}
\end{lemma}
\begin{proof}[Proof of Lemma \ref{freq lemma 1}] Fix $N>0$; we first use Proposition \ref{no waste} and Strichartz to estimate
\begin{equation}\label{after no waste}
\xnorm{\nsc u_{\leq N}}{\infty}{2}{[0,T_{max})}\lesssim \xnorm{\nsc P_{\leq N}(\vert u\vert^p u)}{2}{\frac{2d}{d+2}}{[0,T_{max})}. 
\end{equation}
To proceed, we decompose the nonlinearity and estimate the individual pieces; as before, the particular decomposition we use depends on the ambient dimension. In the estimates that follow, spacetime norms will be taken over $[0,T_{max})\times\R^d$.

\textbf{Case 1.} When $d=3$, we decompose
$$\vert u\vert^p u=\vert u\vert^{p-2}\bar{u}u_{\leq N}^2+(\vert u\vert^{p-2}\bar{u}u_{>N}+2\vert u\vert^{p-2}\bar{u}u_{\leq N})u_{>N}.$$

We can use H\"older, the fractional product rule, fractional chain rule, Sobolev embedding, interpolation, and \eqref{good estimate for endpoint} to estimate the contribution of the first piece as follows:

\begin{align*}\xonorm{\nsc& P_{\leq N}(\vert u\vert^{p-2}\bar{u}u_{\leq N}^2)}{2}{6/5}
\\ &\lesssim\xonorm{\nsc(\vert u\vert^{p-2}\bar{u})}{\infty}{\frac{6p}{7p-8}}\xonorm{u_{\leq N}}{4}{\frac{6p}{4-p}}^2
\\ &\quad+\xonorm{u}{\infty}{\frac{3p}{2}}^{p-1}\xonorm{\nsc (u_{\leq N}^2)}{2}{\frac{6p}{4+p}}
\\ &\lesssim \xonorm{u}{\infty}{\frac{3p}{2}}^{p-2}\xonorm{\nsc u}{\infty}{2}\xonorm{\nsc u_{\leq N}}{4}{3}^2
\\ &\quad +\xonorm{\nsc u}{\infty}{2}^{p-1}\xonorm{u_{\leq N}}{\infty}{\frac{3p}{2}}\xonorm{\nsc u_{\leq N}}{2}{6}
\\ &\lesssim_u \xonorm{\nsc u}{\infty}{2}^{p-1}\xonorm{\nsc u_{\leq N}}{\infty}{2}\xonorm{\nsc u_{\leq N}}{2}{6}+N^{2s_c-\frac12}
\\ &\lesssim_u N^{2s_c-\frac12}. 
\end{align*}
To estimate the contribution of the second piece, we denote $$G=\vert u\vert^{p-2}\bar{u}u_{>N}+2\vert u\vert^{p-2}\bar{u}u_{\leq N}$$ and use Bernstein, H\"older, Lemma \ref{paraproduct}, and \eqref{good estimate for endpoint} to see
\begin{align}
\xonorm{\nsc P_{\leq N}&(G u_{>N})}{2}{6/5} \nonumber
\\ & \lesssim N^{\frac32 s_c}\xonorm{\vert\nabla\vert^{-\frac12s_c}(G u_{>N})}{2}{6/5} \nonumber
\\ &\lesssim N^{\frac32 s_c}\xonorm{\vert\nabla\vert^{\frac12 s_c}G}{\infty}{\frac{12p}{11p-4}}\xonorm{\vert\nabla\vert^{-\frac12 s_c}u_{>N}}{2}{6} \nonumber
\\ &\lesssim \xonorm{\vert\nabla\vert^{\frac12 s_c}G}{\infty}{\frac{12p}{11p-4}}\sum_{M>N}\left(\tfrac{N}{M}\right)^{\frac32s_c}\xonorm{\nsc u_M}{2}{6}\nonumber 
\\ &\lesssim_u \xonorm{\vert\nabla\vert^{\frac12 s_c}G}{\infty}{\frac{12p}{11p-4}} N^{2s_c-\frac12}.\label{freq d3 para piece}
\end{align}
A few applications of the fractional product rule, fractional chain rule, and Sobolev embedding give 
$$\xonorm{\vert\nabla\vert^{\frac12s_c}G}{\infty}{\frac{12p}{11p-4}}\lesssim \xonorm{\nsc u}{\infty}{2}^p\lesssim_u 1,$$ 
so that continuing from \eqref{freq d3 para piece}, we get 
$$\xonorm{\nsc P_{\leq N}\big((\vert u\vert^{p-2}\bar{u}u_{>N}+2\vert u\vert^{p-2}\bar{u}u_{\leq N})u_{>N}\big)}{2}{6/5}\lesssim_u N^{2s_c-1/2}.$$ 
Thus we see that the claim holds in this first case.

\textbf{Case 2.} When $d\in\{4,5\}$, we decompose $$\vert u\vert^p u =\vert u\vert^p u_{\leq N}+\vert u\vert^p u_{>N}.$$ 

We employ H\"older, the fractional product rule, the fractional chain rule, Sobolev embedding, and \eqref{good estimate for endpoint} to estimate the contribution of the first piece as follows:
\begin{align*} \xonorm{\vert\nabla&\vert^{s_c} P_{\leq N}(\vert u\vert^p u_{\leq N})}{2}{\frac{2d}{d+2}} 
\\ &\lesssim\xonorm{\nsc \vert u\vert^p}{\infty}{\frac{2dp}{p(d+4)-4}}\xonorm{u_{\leq N}}{2}{\frac{dp}{2-p}}+\xonorm{u}{\infty}{\frac{dp}{2}}^p\xonorm{\nsc u_{\leq N}}{2}{\frac{2d}{d-2}}
\\ &\lesssim_u\xonorm{u}{\infty}{\frac{dp}{2}}^{p-1}\xonorm{\nsc u}{\infty}{2}\xonorm{\nsc u_{\leq N}}{2}{\frac{2d}{d-2}}+N^{2s_c-\frac12}
\\ &\lesssim_u N^{2s_c-\frac12}.
\end{align*}

For the second piece, we use H\"older, Bernstein, Lemma \ref{paraproduct}, the fractional chain rule, and Sobolev embedding to see
\begin{align*} \xonorm{\nsc &P_{\leq N}(\vert u\vert^p u_{>N})}{2}{\frac{2d}{d+2}}
\\ &\lesssim N^{\frac32s_c}\xonorm{\vert\nabla\vert^{-\frac12 s_c}(\vert u\vert^p u_{>N})}{2}{\frac{2d}{d+2}}
\\ &\lesssim N^{\frac32 s_c}\xonorm{\vert\nabla\vert^{\frac12 s_c}\vert u\vert^p }{\infty}{\frac{4dp}{p(d+8)-4}}\xonorm{\vert \nabla\vert^{-\frac12 s_c}u_{>N}}{2}{\frac{2d}{d-2}}
\\ &\lesssim \xonorm{u}{\infty}{\frac{dp}{2}}^{p-1}\xonorm{\vert\nabla\vert^{\frac12s_c}u}{\infty}{\frac{4dp}{dp+4}}\sum_{M>N}\left(\tfrac{N}{M}\right)^{\frac32s_c}\xonorm{\nsc u_{M}}{2}{\frac{2d}{d-2}}
\\ &\lesssim_u \xonorm{\nsc u}{\infty}{2}^p N^{2s_c-\frac12}
\\ &\lesssim_u N^{2s_c-\frac12}.
\end{align*}
Thus we see that the claim holds in this second case, completing the proof of Lemma~\ref{freq lemma 1}. \end{proof}
We now wish to use \eqref{good estimate for infty} to prove 
\begin{lemma}\label{negative regularity}
$$u\in L_t^\infty \dot{H}_x^{-\varepsilon}([0,T_{max})\times\R^d)\indent\text{for some }\varepsilon>0.$$
\end{lemma}
\begin{proof}[Proof of Lemma \ref{negative regularity}] For $s_c>\tfrac12,$ this is easy; indeed, choosing $\varepsilon>0$ such that $s_c-\tfrac12-\varepsilon>0$, we can use Bernstein and \eqref{good estimate for infty} to see 
\begin{align} \xonorm{\vert\nabla\vert^{-\varepsilon}u}{\infty}{2}&\lesssim \sum_{N\leq 1} N^{-s_c-\varepsilon}\xonorm{\nsc u_N}{\infty}{2}+\sum_{N>1} N^{-s_c-\varepsilon}\xonorm{\nsc u_N}{\infty}{2}\nonumber
\\ &\lesssim_u \sum_{N\leq 1}N^{-s_c-\varepsilon}N^{2s_c-\frac12}+1 \nonumber
\\ &\lesssim_u 1.\nonumber  
\end{align}

When $s_c=\tfrac12$ (that is, $p=\tfrac{4}{d-1}$), we need to work a bit harder. To begin, we note that by Bernstein and \eqref{good estimate for infty}, we have
		\begin{align}
		\xonorm{\ntw{\frac25}u}{\infty}{2}&\lesssim\sum_{N\leq 1}N^{-\frac{1}{10}}\xonorm{\ntw{\frac12}u_N}{\infty}{2}
		+\sum_{N>1}N^{-\frac{1}{10}}\xonorm{\ntw{\frac12}u_N}{\infty}{2} \nonumber
		\\ &\lesssim_u \sum_{N\leq 1}N^{\frac25}+1 \nonumber
		\\ &\lesssim_u1. \label{additional input}
		\end{align}
We wish to show that in fact, we have the more quantitative statement
\begin{equation}\label{better estimate for infty} \xnorm{\vert\nabla\vert^{\frac25} u_{\leq N}}{\infty}{2}{[0,T_{max})}\lesssim_u N^{\frac12}\indent\text{for all }N>0. 
\end{equation} 

Once we have established \eqref{better estimate for infty}, we can complete the proof of Lemma \ref{negative regularity} as follows: choosing $0<\eps<\tfrac{1}{10}$, we use Bernstein, \eqref{additional input}, and \eqref{better estimate for infty} to estimate
	\begin{align*}
	\xonorm{\ntw{-\eps}u}{\infty}{2}&\lesssim\sum_{N\leq 1}N^{-\frac25-\eps}\xonorm{\ntw{\frac25}u_N}{\infty}{2}
	+\sum_{N>1}N^{-\frac25-\eps}\xonorm{\ntw{\frac25}u_N}{\infty}{2}
	\\ &\lesssim_u\sum_{N\leq 1}N^{\frac{1}{10}-\eps}+1
	\\ &\lesssim_u 1.
	\end{align*}

Thus, to complete the proof of Lemma \ref{negative regularity}, it remains to establish \eqref{better estimate for infty}. We begin by fixing $N>0$. The proof of \eqref{better estimate for infty} will be a second iteration of the arguments that gave \eqref{good estimate for infty}, this time using \eqref{additional input} as additional input. 

We first use \eqref{additional input} (and the uniqueness of weak limits) to see that the no-waste Duhamel formula (Proposition \ref{no waste}) also holds in the weak $\dot{H}_x^{\frac25}$ topology; thus, using Strichartz as well, we can estimate 
$$\xnorm{\vert\nabla\vert^{\frac25} u_{\leq N} }{\infty}{2}{[0,T_{max})}\lesssim \xnorm{\vert\nabla\vert^{\frac25} P_{\leq N}(\vert u\vert^{\frac{4}{d-1}} u)}{2}{\frac{2d}{d+2}}{[0,T_{max})}.$$
Once again, we decompose the nonlinearity, and again the decomposition depends on the ambient dimension. The estimates that follow will be very similar in spirit to the estimates that gave \eqref{good estimate for infty}; all estimates will be taken over $[0,T_{max}).$ 

\textbf{Case 1.} When $d=3$, we decompose
$$\vert u\vert^2 u=\bar{u}u_{\leq N}^2+(\bar{u}u_{>N}+2\bar{u}u_{\leq N})u_{>N}.$$

We estimate the first piece as follows: by H\"older, the fractional product rule, the fractional chain rule, Sobolev embedding, interpolation, \eqref{good estimate for endpoint} 
and \eqref{additional input}, 

\begin{align*} \xonorm{&\ntw{\frac25} P_{\leq N}(\bar{u} u_{\leq N}^2)}{2}{6/5}
 \\ & \lesssim\xonorm{\ntw{\frac25}{u}}{\infty}{2}\xonorm{u_{\leq N}}{4}{6}^2 +\xonorm{u}{\infty}{30/11}\xonorm{\ntw{\frac25} (u_{\leq N})^2}{2}{15/7}
\\ &\lesssim \xonorm{\ntw{\frac25}u}{\infty}{2}\xonorm{\ntw{\frac12} u_{\leq N}}{4}{3}^2 +\xonorm{\ntw{\frac25}u}{\infty}{2}\xonorm{u_{\leq N}}{\infty}{3}\xonorm{\ntw{\frac25}u_{\leq N}}{2}{15/2}
\\ &\lesssim\xonorm{\ntw{\frac25}u}{\infty}{2}\xonorm{\ntw{\frac12}u}{\infty}{2}\xonorm{\ntw{\frac12}u_{\leq N}}{2}{6}
\\ &\quad +\xonorm{\ntw{\frac25}u}{\infty}{2}\xonorm{\ntw{\frac12}u}{\infty}{2}\xonorm{\ntw{\frac12}u_{\leq N}}{2}{6}
\\ &\lesssim_u N^{\frac12}.
\end{align*}

For the second piece, we first let $G:=\bar{u}u_{>N}+2\bar{u}u_{\leq N},$ and use Bernstein, H\"older, Lemma \ref{paraproduct}, Sobolev embedding, and \eqref{good estimate for endpoint} to see
\begin{align}
\xonorm{\ntw{\frac25}P_{\leq N}(Gu_{>N})}{2}{6/5}&\lesssim N^{\frac45}\xonorm{\ntw{-\frac25}(Gu_{>N})}{2}{6/5}\nonumber
\\ &\lesssim N^{\frac45}\xonorm{\ntw{\frac25}G}{\infty}{6/5}\xonorm{\ntw{-\frac25}u_{>N}}{2}{15/2}\nonumber
\\ &\lesssim \xonorm{\ntw{\frac25}G}{\infty}{6/5}\sum_{M>N} \left(\tfrac{N}{M}\right)^{\frac45}\xonorm{\ntw{\frac25}u_M}{2}{15/2} \nonumber
\\ &\lesssim \xonorm{\ntw{\frac25}G}{\infty}{6/5}\sum_{M>N}\left(\tfrac{N}{M}\right)^{\frac45}\xonorm{\ntw{\frac12}u_M}{2}{6} \nonumber
\\ &\lesssim \xonorm{\ntw{\frac25}G}{\infty}{6/5}N^{\frac12}.\label{2/5}
\end{align}
A few applications of the fractional product rule, Sobolev embedding, and \eqref{additional input} give 
\begin{align*}\xonorm{\ntw{\frac25}G}{\infty}{6/5}&\lesssim \xonorm{u}{\infty}{3}\xonorm{\ntw{\frac25}u}{\infty}{2}
\\ &\lesssim_u \xonorm{\ntw{\frac12}u}{\infty}{2}\xonorm{\ntw{\frac25}u}{\infty}{2}
\\ &\lesssim_u 1,
\end{align*} 
so that \eqref{2/5} becomes $$\xonorm{\ntw{\frac25}P_{\leq N}(Gu_{>N})}{2}{6/5}\lesssim_u N^{\frac12}.$$ We see that \eqref{better estimate for infty} holds in this first case.  

\textbf{Case 2.} When $d\in\{4,5\}$, we decompose
$$ \vert u\vert^{\frac{4}{d-1}}u=\vert u\vert^{\frac{4}{d-1}}u_{\leq N}+\vert u\vert^{\frac{4}{d-1}}u_{>N}.$$

We first note that interpolating between $u\in L_t^\infty \dot{H}_x^{\frac25}$ and $u\in L_t^\infty\dot{H}_x^{\frac12},$ we have \begin{align}\label{interpolated freq} u\in L_t^\infty\dot{H}_x^{\frac{21-d}{40}}.
\end{align}
Thus, to estimate the contribution of the first piece, we can use H\"older, the fractional product rule, the fractional chain rule, Sobolev embedding, \eqref{good estimate for endpoint}, \eqref{good estimate for infty}, and \eqref{interpolated freq} to see 

\begin{align*}\xonorm{&\ntw{\frac25}(\vert u\vert^{\frac{4}{d-1}}u_{\leq N})}{2}{\frac{2d}{d+2}}
\\ &\lesssim \xonorm{\ntw{\frac25}\vert u\vert^{\frac{4}{d-1}}}{\infty}{\frac{2d}{5}}\xonorm{u_{\leq N}}{2}{\frac{2d}{d-3}} +\xonorm{u}{\infty}{\frac{40d}{21(d-1)}}^{\frac{4}{d-1}}\xonorm{\ntw{\frac25}u_{\leq N}}{2}{\frac{10d}{5d-11}}
\\ &\lesssim \xonorm{u}{\infty}{\frac{2d}{d-1}}^{\frac{5-d}{d-1}}\!\xonorm{\ntw{\frac25}u}{\infty}{2}\xonorm{\ntw{\frac12}u_{\leq N}}{2}{\frac{2d}{d-2}}\! +\!\xonorm{\ntw{\frac{21-d}{40}}u}{\infty}{2}^{\frac{4}{d-1}}\xonorm{\ntw{\frac12}u}{2}{\frac{2d}{d-2}}
\\ &\lesssim_u \xonorm{\ntw{\frac12}u}{\infty}{2}^{\frac{5-d}{d-1}}\xonorm{\ntw{\frac25}u}{\infty}{2} N^{\frac12}+N^{\frac12}
\\ &\lesssim_u N^{\frac12}.
\end{align*}

For the second piece, we use Bernstein, H\"older, Lemma \ref{paraproduct}, the fractional chain rule, Sobolev embedding, \eqref{good estimate for endpoint}, \eqref{good estimate for infty}, and \eqref{additional input} to see
\begin{align*}
\xonorm{\ntw{\frac25}&P_{\leq N}(\vert u\vert^{\frac{4}{d-1}}u)}{2}{\frac{2d}{d+2}}
\\ &\lesssim N^{\frac45}\xonorm{\ntw{-\frac25}(\vert u\vert^{\frac{4}{d-1}}u_{\geq N})}{2}{\frac{2d}{d+2}}
\\ &\lesssim N^{\frac45}\xonorm{\ntw{\frac25}\vert u\vert^{\frac{4}{d-1}}}{\infty}{\frac{2d}{5}}\xonorm{\ntw{-\frac25}u_{>N}}{2}{\frac{10d}{5d-11}}
\\ &\lesssim \xonorm{u}{\infty}{\frac{2d}{d-1}}^{\frac{5-d}{d-1}}\xonorm{\ntw{\frac25}u}{\infty}{2}\sum_{M>N}\left(\tfrac{N}{M}\right)^{\frac45}\xonorm{\ntw{\frac25}u_M}{2}{\frac{10d}{5d-11}}
\\ &\lesssim_u \sum_{M>N}\left(\tfrac{N}{M}\right)^{\frac45}\xonorm{\ntw{\frac12}u_M}{2}{\frac{2d}{d-2}} 
\\ &\lesssim_u N^{\frac12}.
\end{align*}
Thus \eqref{better estimate for infty} holds in this second case. This completes the proof of Lemma \ref{negative regularity}.\end{proof}

With Lemma \ref{negative regularity} at hand, we are ready to complete the proof of Theorem \ref{frequency cascade}. Fix $t\in[0,T_{max})$ and $\eta>0$. By Remark~\ref{another consequence}, we may find $c(\eta)>0$ so that $$\int_{\vert\xi\vert\leq c(\eta)N(t)}\vert\xi\vert^{2s_c}\vert \wh{u}(t,\xi)\vert^2\, d\xi\leq\eta.$$ Interpolating with $u\in L_t^\infty\dot{H}_x^{-\varepsilon},$ we get 
$$\int_{\vert \xi\vert\leq c(\eta)N(t)}\vert \wh{u}(t,\xi)\vert^2\ d\xi\lesssim_u \eta^{\frac{\varepsilon}{s_c+\varepsilon}}.$$
On the other hand, we have
$$\int_{\vert\xi\vert\geq c(\eta)N(t)}\vert \wh{u}(\xi,t)\vert^2\, d\xi\leq (c(\eta)N(t))^{-2s_c}\int\vert\xi\vert^{2s_c}\vert \wh{u}(t,\xi)\vert^2\, d\xi\lesssim_u (c(\eta)N(t))^{-2s_c}.$$

Adding these last estimates and using Plancherel, we conclude that for all $t\in[0,T_{max}),$ we have
$$0\leq{M}(u(t)):=\int\vert u(t,x)\vert^2\, dx\lesssim_u \eta^{\frac{\varepsilon}{s_c+\varepsilon}}+(c(\eta)N(t))^{-2s_c}.$$ Thus, recalling $\lim_{t\to T_{max}}N(t)=\infty$, we can conclude that for all $\eta>0$, we may find $t_0\in [0,T_{max})$ so that for all $t\in(t_0,T_{max}),$ we have  ${M}(u(t))\leq\eta.$ But by conservation of mass, ${M}(u(t))\equiv{M}(u_0),$ and so we find that ${M}(u_0)\leq\eta$ for all $\eta>0$. Of course, this gives $u\equiv 0$, which contradicts that $u$ blows up (cf. \eqref{u blows up}).
\end{proof}
\begin{remark} We have omitted the case $(d,s_c)=(5,\tfrac12)$ from Theorem~\ref{frequency cascade} only because we omitted this case from the long-time Strichartz estimate, Proposition~\ref{lts lemma}. Of course, as remarked in the proof of Proposition~\ref{lts lemma}, the long-time Strichartz estimates continue to hold when $(d,s_c)=(5,\tfrac12)$; thus we see that Theorem~\ref{frequency cascade} holds in this case as well.
\end{remark}


\section{The frequency-localized interaction Morawetz inequality}\label{flim section}
In this section, we prove spacetime bounds for the high-frequency portions of almost periodic solutions to \eqref{nls}; these bounds can be used to preclude the existence of quasi-solitons (Section \ref{Quasi}). As we will see, establishing these bounds will lead to the most non-trivial restrictions on the set of $(d,s_c)$ to which our main theorem (Theorem \ref{swamprat}) applies; see the proof below for a more detailed discussion. The main result of this section is the following
\begin{proposition}[Frequency-localized interaction Morawetz inquality]\label{flim} Let $(d,s_c)$ satisfy \eqref{constraints}. Suppose $u:[0,T_{max})\times\R^d\to\C$ is an almost periodic solution to \eqref{nls} such that $N(t)\equiv N_k\geq 1$ on each characteristic subinterval $J_k\subset[0,T_{max})$, and let $I\subset[0,T_{max})$ be a compact time interval, which is a union of contiguous subintervals $J_k$. Then for any $\eta>0$, there exists $N_0=N_0(\eta)$ such that for any $N\leq N_0$, we have
\begin{equation}\label{flim D}
-\int_I\iint_{\R^d\times\R^d}\vert u_{\geq N}(t,y)\vert^2\Delta(\tfrac{1}{\vert\cdot\vert})(x-y)\vert u_{\geq N}(t,x)\vert^2\, dx\, dy \, dt\lesssim_u\eta(N^{1-4s_c}+K),\end{equation}
where $K:=\int_I N(t)^{3-4s_c}\,dt.$ Furthermore, $N_0$ and the implicit constants above do not depend on $I$. 
\end{proposition} 
\noindent Before we begin the proof of Proposition \ref{flim}, we recall a general form of the interaction Morawetz inequality, introduced originally in \cite{CKSTT} (for more discussion, see also \cite{KV} and the references cited therein). We will essentially follow the presentation in \cite[Section 5]{Monica:thesis art}.

For a fixed function $a:\R^d\to\R$ and $\varphi$ solving $(i\partial_t+\Delta)\varphi=\mathcal{N}$, we define the interaction Morawetz action by 
$$M(t)=2\,\Im\iint_{\R^d\times\R^d}\vert\varphi(t,y)\vert^2 a_k(x-y)(\varphi_k\bar{\varphi})(t,x)\, dx\, dy,$$ 
where subscripts denote spatial derivatives and repeated indices are summed. If we define the mass bracket $$\{f,g\}_m:=\Im(f\bar{g})$$ and the momentum bracket $$\{f,g\}_{\p}:=\Re(f\nabla \bar{g}-g\nabla\bar{f}),$$ then one can show
\begin{align} \partial_t M(t)=&-\iint_{\R^d\times\R^d}\vert\varphi(t,y)\vert^2a_{jjkk}(x-y)\vert\varphi(t,x)\vert^2\, dx\, dy\nonumber
\\ &+\iint_{\R^d\times\R^d}\vert\varphi(t,y)\vert^24a_{jk}(x-y)\Re(\bar{\varphi}_j\varphi_k)(t,x)\, dx\, dy\label{mor2}
\\ &-\iint_{\R^d\times\R^d}\indent 2\, \Im(\bar{\varphi}\varphi_k)(t,y)a_{jk}(x-y)2\, \Im(\bar{\varphi}\varphi_j)(t,x)\, dx\, dy\label{mor3}
\\ &+\iint_{\R^d\times\R^d}2\{\N,\varphi\}_m(t,y)a_j(x-y)\ 2\, \Im(\bar{\varphi}\varphi_j)(t,x)\, dx\, dy\nonumber
\\ &+\iint_{\R^d\times\R^d}\vert\varphi(t,y)\vert^2\ 2\,\nabla a(x-y)\cdot\{N,\varphi\}_{\p}(t,x)\, dx\, dy.\nonumber
\end{align}
To prove Proposition \ref{flim}, we will use $a(x)=\vert x\vert$. Note that in this case, we have
$$\left\{\begin{array}{ll} a_j(x)=\tfrac{x_j}{\vert x\vert},
\\ \\ a_{jk}(x)=\tfrac{\delta_{jk}}{\vert x\vert}-\tfrac{x_jx_k}{\vert x\vert^3},
\\ \\ \Delta a(x)=\tfrac{d-1}{\vert x\vert},
\\ \\ \Delta\Delta a(x)=-(d-1)\Delta(\tfrac{1}{\vert x\vert}).
\end{array}\right.$$

For this choice of $a$, one can also show $\eqref{mor2}+\eqref{mor3}\geq 0$ (for details, see for example \cite[Lemma 5.4]{Monica:thesis art}). Thus, integrating $\partial_t M$ over $I$, we arrive at the following

\begin{lemma}[Interaction Morawetz inequality]\label{interaction morawetz}
\begin{align*} &-\int_I\iint_{\R^d\times\R^d} \vert \varphi(t,y)\vert^2\Delta(\tfrac{1}{\vert \cdot\vert})(x-y)\vert \varphi(t,x)\vert^2\, dx\, dy\, dt 
\\ &+\int_I\iint_{\R^d\times\R^d}\vert \varphi(t,y)\vert^2\tfrac{x-y}{\vert x-y\vert}\cdot\left\{\N,\varphi\right\}_{\p}(t,x)\, dx\, dy\, dt 
\\ &\ \ \ \lesssim \sup_{t\in I}\iint_{\R^d\times\R^d}\vert \varphi(t,y)\vert^2\tfrac{x-y}{\vert x-y\vert}\cdot\nabla \varphi(t,x)\bar{\varphi}(t,x)\, dx\, dy 
\\ &\ \ \ \ +\bigg\vert\int_I\iint_{\R^d\times\R^d}\left\{\N,\varphi\right\}_m(t,y)\tfrac{x-y}{\vert x-y\vert}\cdot\nabla \varphi(t,x)\bar{\varphi}(t,x)\, dx\, dy\, dt\ \bigg\vert.
\end{align*}
\end{lemma}
To prove Proposition \ref{flim}, we will apply this estimate with $\varphi=u_{\geq N}$, with $N$ chosen small enough to capture `most' of the solution. To make this idea more precise, we first need to record the following corollary of Proposition \ref{lts lemma}. 
 
\begin{corollary}[Low and high frequencies control]\label{lts corollary} Let $(d,s_c)$ satisfy \eqref{constraints2}, and let $u:[0,T_{max})\times\R^d\to\C$ be an almost periodic solution to \eqref{nls} with $N(t)\equiv N_k\geq 1$ on each characteristic subinterval $J_k\subset[0,T_{max}).$ Then on any compact time interval $I\subset[0,T_{max}),$ which is a union of continuous subintervals $J_k$, and for any frequency $N>0$, we have
\begin{equation}\label{uhi control} \xnorm{u_{\geq N}}{q}{r}{I}\lesssim_u N^{-s_c}(1+N^{4s_c-1}K)^{\frac{1}{q}}\end{equation}
for all $\tfrac{2}{q}+\tfrac{d}{r}=\tfrac{d}{2}$ with $q>4-\tfrac{2p}{dp-4},$ where $K:=\int_I N(t)^{3-4s_c}\, dt.$ 

Moreover, for any $\eta>0$, there exists $N_0=N_0(\eta)$ such that for all $N\leq N_0$, we have
\begin{equation}\label{ulo control} \xnorm{\nsc u_{\leq N}}{q}{r}{I}\lesssim_u \eta(1+N^{4s_c-1}K)^{\frac{1}{q}}\end{equation}
 for all $\tfrac{2}{q}+\tfrac{d}{r}=\tfrac{d}{2}$ with $q\geq 2$.

Furthermore, $N_0$ and the implicit constants in \eqref{uhi control} and \eqref{ulo control} do not depend on $I$. 
\end{corollary} 
\begin{proof}[Proof of Corollary \ref{lts corollary}] We first show \eqref{uhi control}. For fixed $\alpha>s_c-\tfrac12$, we can use Bernstein and \eqref{lts} to see
\begin{align} \xnorm{\vert\nabla\vert^{-\alpha}u_{\geq N}}{2}{\frac{2d}{d-2}}{I}&\lesssim\sum_{M\geq N} M^{-\alpha-s_c}\xnorm{\nsc u_M}{2}{\frac{2d}{d-2}}{I} \nonumber
\\ &\lesssim_u \sum_{M\geq N}M^{-\alpha-s_c}(1+M^{2s_c-\frac12}K^{\frac12})\nonumber
\\ &\lesssim_u N^{-\alpha-s_c}(1+N^{4s_c-1}K)^{\frac12}.\label{corollary neg}
\end{align}
Now, take $(q,r)$ with $2< q\leq\infty$ and $\tfrac{2}{q}+\tfrac{d}{r}=\tfrac{d}{2}$, and define $\alpha=\tfrac{(q-2)(dp-4)}{4p}.$ Notice that $\alpha>s_c-\tfrac12$ exactly when $q>4-\tfrac{2p}{dp-4}.$ Thus, in this case, we get by interpolation and \eqref{corollary neg} that
\begin{align*} \xnorm{u_{\geq N}}{q}{r}{I}&\lesssim\xnorm{\vert\nabla\vert^{-\alpha}u_{\geq N}}{2}{\frac{2d}{d-2}}{I}^{\frac{2}{q}}\xnorm{\nsc u_{\geq N}}{\infty}{2}{I}^{1-\frac{2}{q}}
\\ &\lesssim_u \left[ N^{-\frac{qs_c}{2}}(1+N^{4s_c-1}K)^{\frac12}\right]^{\frac{2}{q}},
\end{align*} which gives \eqref{uhi control}. As for \eqref{ulo control}, we first note that since $\inf_{t\in I}N(t)\geq 1$, for any $\eta>0$ we may find $N_0(\eta)$ so that $$\xnorm{\nsc u_{\leq N}}{\infty}{2}{I}\leq\eta$$ for all $N\leq N_0$ (cf. Remark~\ref{another consequence}). The estimate \eqref{ulo control} then follows by interpolating with \eqref{lts}. 
\end{proof}
We are now ready for the
\begin{proof}[Proof of Proposition \ref{flim}] Take $I\subset[0,T_{max})$, a compact time interval, which is a contiguous union of subintervals $J_k$, and let $K:=\int_I N(t)^{3-4s_c}\, dt$. Throughout the proof, all spacetime norms will be taken over $I\times\R^d$.

Fix $\eta>0$, and choose $N_0=N_0(\eta)$ small enough that \eqref{ulo control} holds; recall that \eqref{uhi control} holds without any restriction on $N$. Next, we claim that for $N_0$ possibly even smaller, we can guarantee that for $N\leq N_0$, we have
\begin{equation}\label{uhi small} \xonorm{u_{\geq N}}{\infty}{2}\lesssim_u \eta^{10} N^{-s_c}
\end{equation}
and 
\begin{equation}\label{uhi small2} \xonorm{\vert\nabla\vert^{1-s_c}u_{\geq N}}{\infty}{2}\lesssim_u \eta^{10}N^{1-2s_c}.
\end{equation}
Indeed, by Remark~\ref{another consequence} and the fact that $\inf_{t\in I} N(t)\geq 1$, we may find $c(\eta)>0$ so that $$\xonorm{\nsc u_{\leq c(\eta)}}{\infty}{2}\leq\eta^{10};$$ combining this inequality with Bernstein, we get 
 \begin{align*} N^{s_c}\xonorm{u_{\geq N}}{\infty}{2}&\lesssim N^{s_c}\xonorm{u_{N\leq\cdot\leq c(\eta)}}{\infty}{2}+N^{s_c}\xonorm{u_{>c(\eta)}}{\infty}{2}
\\ &\lesssim \xonorm{\nsc u_{\leq c(\eta)}}{\infty}{2}+\tfrac{N^{s_c}}{c(\eta)^{s_c}}\xonorm{\nsc u_{>c(\eta)}}{\infty}{2}
\\ &\lesssim_u \eta^{10}+N^{s_c}.
\end{align*}
Thus, taking $N$ sufficiently small, we recover \eqref{uhi small}. A similar argument yields \eqref{uhi small2}.

Next, we record the following inequality that will be useful below:
\begin{equation}\label{hardy} \sup_{y\in\R^d} \bigg\vert\int_{\R^d} \tfrac{x-y}{\vert x-y\vert}\cdot\nabla\varphi(x)\bar{\varphi}(x)\ dx\bigg\vert\lesssim \norm{\vert\nabla\vert^{s}\varphi}_2\norm{\vert\nabla\vert^{1-s}\varphi}_2
\end{equation} for $0\leq s\leq 1$. Indeed, for fixed $y\in\R^d$, we can first write
\begin{align*} \bigg\vert\int_{\R^d} \tfrac{x-y}{\vert x-y\vert}\cdot\nabla\varphi(x)\bar{\varphi}(x)\ dx\bigg\vert&\lesssim \norm{\vert\nabla\vert^s\tfrac{x-y}{\vert x-y\vert}\varphi}_2\norm{\vert\nabla\vert^{-s}\nabla\varphi}_2
\\ &\sim \norm{\vert\nabla\vert^s\tfrac{x-y}{\vert x-y\vert}\varphi}_2\norm{\vert \nabla\vert^{1-s}\varphi}_2.
\end{align*}
Thus, to prove \eqref{hardy}, we need to see that the operator $\vert\nabla\vert^s\tfrac{x-y}{\vert x-y\vert}\vert\nabla\vert^{-s}$ is bounded on $L_x^2$ (uniformly in $y$). When $s=0$, this is clear. When $s=1$, this follows from the chain rule, Hardy's inequality, and the boundedness of Riesz transforms. The general case then follows from complex interpolation.

We now wish to apply the interaction Morawetz inequality (Lemma \ref{interaction morawetz}) with $\varphi=u_{\geq N}$ and $\N=P_{\geq N}(\vert u\vert^p u)$, with $N\leq N_0$. Together with \eqref{uhi small}, \eqref{uhi small2},  \eqref{hardy}, Bernstein, and the fact that $u\in L_t^\infty \dot{H}_x^{s_c}(I\times\R^d)$, an application of Lemma \ref{interaction morawetz} gives
\begin{align}
-&\iiint \vert u_{\geq N}(t,y)\vert^2\Delta(\tfrac{1}{\vert\cdot\vert})(x-y)\vert u_{\geq N}(t,x)\vert^2\, dx\, dy \, dt \nonumber
\\ &+\quad \!\!\!\!\!\!\iiint \vert u_{\geq N}(t,y)\vert^2\tfrac{x-y}{\vert x-y\vert}\cdot\{P_{\geq N}(\vert u\vert^p u),u_{\geq N}\}_{\p}(t,x)\, dx\, dy \, dt \nonumber
\\ &\lesssim\xonorm{u_{\geq N}}{\infty}{2}^2\xonorm{\vert\nabla\vert^{1/2}u_{\geq N}}{\infty}{2}^2\nonumber
\\ &\quad +\xonorms{\{P_{\geq N}(\vert u\vert^p u),u_{\geq N}\}_m}{1}\xonorm{\nsc u_{\geq N}}{\infty}{2}\xonorm{\vert\nabla\vert^{1-s_c}u_{\geq N}}{\infty}{2} \nonumber
\\ &\lesssim_u \eta^{20}N^{1-4s_c}+\eta^{10}N^{1-2s_c}\xonorms{\{P_{\geq N}(\vert u\vert^p u),u_{\geq N}\}_m}{1}. \label{interaction morawetz with uhi}
\end{align}
Thus, to prove Proposition \ref{flim}, we need to get sufficient control over the mass and momentum bracket terms appearing above. 

To begin, we consider the contribution of the momentum bracket term. We can write
\begin{align*}
\{P_{\geq N}&(\vert u\vert^p u), u_{\geq N}\}_{\p}
\\ &=\pb{\vert u\vert^p u}{u}-\pb{\vert u_{\leq N}\vert^p u_{\leq N}}{u_{\leq N}}
\\ &\quad\quad -\pb{\vert u\vert^p u-\vert u_{\leq N}\vert^p u_{\leq N}}{u_{\leq N}}-\pb{P_{\leq N}(\vert u\vert^p u)}{u_{\geq N}}
\\ & =-\tfrac{p}{p+2}\nabla(\vert u\vert^{p+2}-\vert u_{\leq N}\vert^{p+2})-\pb{\vert u\vert^p u-\vert u_{\leq N}\vert^p u_{\leq N}}{u_{\leq N}}
\\ &\quad\quad -\pb{P_{\leq N}(\vert u\vert^p u)}{u_{\geq N}}
\\ &=:I+II+III.
\end{align*}

After an integration by parts, we see that term $I$ contributes to the left-hand side of \eqref{interaction morawetz with uhi} a multiple of
\begin{align*}&\iiint \frac{\vert u_{\geq N}(t,y)\vert^2\vert u_{\geq N}(t,x)\vert^{p+2}}{\vert x-y\vert}\, dx\, dy\, dt
\\ &-\iiint \frac{\vert u_{\geq N}(t,y)\vert^2(\vert u\vert^{p+2}-\vert u_{\geq N}\vert^{p+2}-\vert u_{\leq N}\vert^{p+2})(t,x)}{\vert x-y\vert}\, dx\, dy \, dt.\end{align*} 

For term $II$, we use $\pb{f}{g}=\nabla\text{\O}(fg)+\text{\O}(f\nabla g);$ when the derivative hits the product, we integrate by parts, while for the second term we simply bring absolute values inside the integral. In this way, we find that term $II$ contributes to the right-hand side of \eqref{interaction morawetz with uhi} a multiple of
\begin{align*}&\iiint\frac{\vert u_{\geq N}(t,y)\vert^2\vert(\vert u\vert^p u-\vert u_{\leq N}\vert^p u_{\leq N})u_{\leq N}(t,x)\vert}{\vert x-y\vert}\, dx\, dy\, dt
\\ &+\iiint\vert u_{\geq N}(t,y)\vert^2\vert(\vert u\vert^p u-\vert u_{\leq N}\vert^p u_{\leq N})(t,x)\vert\ \vert\nabla u_{\leq N}(t,x)\vert\, dx\, dy \, dt.
\end{align*}

Finally, for term $III$, we integrate by parts when the derivative falls on $u_{\geq N}$; in this way, we see that term $III$ contributes to the right-hand side of \eqref{interaction morawetz with uhi} a multiple of
\begin{align*} &\iiint\frac{\vert u_{\geq N}(t,y)\vert^2\vert u_{\geq N}(t,x)\vert\, \vert P_{\leq N}(\vert u\vert^p u)(t,x)\vert}{\vert x-y\vert}\, dx\, dy\, dt
\\ &+\iiint \vert u_{\geq N}(t,y)\vert^2\vert u_{\geq N}(t,x)\vert\, \vert\nabla P_{\leq N}(\vert u\vert^p u)(t,x)\vert\, dx\, dy\, dt.
\end{align*}

We next consider the mass bracket term in \eqref{interaction morawetz with uhi}. Exploiting the fact that 
$$\mb{\vert u_{\geq N}\vert^p u_{\geq N}}{u_{\geq N}}=0,$$
we can write
\begin{align*} \mb{P_{\geq N}(\vert u\vert^p u)}{u_{\geq N}}
 &=\mb{P_{\geq N}(\vert u\vert^p u)-\vert u_{\geq N}\vert^p u_{\geq N}}{u_{\geq N}}
\\ &=\mb{P_{\geq N}(\vert u\vert^pu-\vert u_{\geq N}\vert^p u_{\geq N}-\vert u_{\leq N}\vert^p u_{\leq N})}{u_{\geq N}}
\\ &\ +\mb{P_{\geq N}(\vert u_{\leq N}\vert^p u_{\leq N})}{u_{\geq N}}-\mb{P_{\leq N}(\vert u_{\geq N}\vert^p u_{\geq N})}{u_{\geq N}}.
\end{align*}

We will now collect the contributions of the mass and momentum bracket terms and insert them back into \eqref{interaction morawetz with uhi}. We will also make use of the pointwise inequalities 
\begin{align}&\big\vert \vert f+g\vert^p(f+g)-\vert f\vert^p f\big\vert\lesssim \vert g\vert^{p+1}+\vert g\vert\ \vert f\vert^{p},\nonumber
\\ &\big\vert \vert f+g\vert^{p+2}-\vert f\vert^{p+2}-\vert g\vert^{p+2}\big\vert\lesssim \vert f\vert\ \vert g\vert^{p+1}+\vert f\vert^{p+1}\vert g\vert.\nonumber
\end{align}
In this way, \eqref{interaction morawetz with uhi} becomes
\begin{align} &-\iiint \vert u_{\geq N}(t,y)\vert^2\Delta(\tfrac{1}{\vert\cdot\vert})(x-y)\vert u_{\geq N}(t,x)\vert^2\, dx\, dy\, dt 
\\ &\quad +\iiint\frac{\vert u_{\geq N}(t,y)\vert^2\vert u_{\geq N}(t,x)\vert^{p+2}}{\vert x-y\vert}\, dx\, dy\, dt \nonumber
\\ \nonumber
\\ &\lesssim_u \eta^{20}N^{1-4s_c} \label{error 1}
\\ & +\eta^{10}N^{1-2s_c}\xonorms{\vert u_{\leq N}\vert^p u_{\geq N}^2}{1} \label{error 2}
\\  & +\eta^{10}N^{1-2s_c}\xonorms{\vert u_{\geq N}\vert^{p+1}u_{\leq N}}{1} \label{error 3}
\\  & +\eta^{10}N^{1-2s_c}\xonorms{P_{\geq N}(\vert u_{\leq N}\vert^pu_{\leq N})u_{\geq N}}{1} \label{error 4}
\\  & +\eta^{10}N^{1-2s_c}\xonorms{P_{\leq N}(\vert u_{\geq N}\vert^pu_{\geq N})u_{\geq N}}{1} \label{error 5}
\\ &+\iiint \frac{\vert u_{\geq N}(t,y)\vert^2\vert u_{\geq N}(t,x)\vert\,\vert u_{\leq N}(t,x)\vert^{p+1}}{\vert x-y\vert}\,dx\,dy\,dt\label{error 6}
\\ &+\iiint\frac{\vert u_{\geq N}(t,y)\vert^2\vert u_{\geq N}(t,x)\vert^{p+1}\vert u_{\leq N}(t,x)\vert}{\vert x-y\vert}\,dx\,dy\,dt\label{error 7}
\\ &+\iiint\frac{\vert u_{\geq N}(t,y)\vert^2\vert P_{\leq N}(\vert u\vert^p u)(t,x)\vert\,\vert u_{\geq N}(t,x)\vert}{\vert x-y\vert}\,dx\,dy\,dt \label{error 8}
\\ &+\iiint\vert u_{\geq N}(t,y)\vert^2\vert u_{\geq N}(t,x)\vert\,\vert u_{\leq N}(t,x)\vert^p\vert\nabla u_{\leq N}(t,x)\vert\,dx\,dy\,dt \label{error 9}
\\ &+\iiint\vert u_{\geq N}(t,y)\vert^2\vert u_{\geq N}(t,x)\vert^{p+1}\vert\nabla u_{\leq N}(t,x)\vert\,dx\,dy\,dt\label{error 10}
\\ &+\iiint\vert u_{\geq N}(t,y)\vert^2\vert u_{\geq N}(t,x)\vert\,\vert\nabla P_{\leq N}(\vert u\vert^p u)(t,x)\vert\,dx\,dy\,dt.\label{error 11}
\end{align}

To complete the proof of Proposition \ref{flim}, we need to show that the error terms \eqref{error 1} through \eqref{error 11} are acceptable, in the sense that they can be controlled by $\eta(N^{1-4s_c}+K).$ Clearly, \eqref{error 1} is acceptable. 

Next, we consider \eqref{error 2}. Using H\"older, Sobolev embedding, \eqref{uhi control}, and \eqref{ulo control}, we get
\begin{align*}
\xonorms{\vert u_{\leq N}\vert^p u_{\geq N}^2}{1}&\lesssim \xonorm{u_{\leq N}}{2p}{dp}^p\xonorm{u_{\geq N}}{4}{\frac{2d}{d-1}}^2
\\ &\lesssim \xonorm{\nsc u_{\leq N}}{2p}{\frac{2dp}{dp-2}}^p\xonorm{u_{\geq N}}{4}{\frac{2d}{d-1}}^2
\\ &\lesssim_u \eta^p N^{-2s_c}(1+N^{4s_c-1} K), 
\end{align*}
which renders \eqref{error 2} acceptable.

We now turn to \eqref{error 3}. For this term, we can again use H\"older, Sobolev embedding, \eqref{uhi control}, and \eqref{ulo control} to see
\begin{align*} 
\xonorms{\vert u_{\geq N}\vert^{p+1}u_{\leq N}}{1}&\lesssim \xonorm{u_{\geq N}}{3}{\frac{6d}{3d-4}}^2\xonorm{u_{\geq N}}{\infty}{\frac{dp}{2}}^{p-1}\xonorm{u_{\leq N}}{3}{\frac{3dp}{6-2p}}
\\ &\lesssim \xonorm{u_{\geq N}}{3}{\frac{6d}{3d-4}}^2\xonorm{\nsc u}{\infty}{2}^{p-1}\xonorm{\nsc u_{\leq N}}{3}{\frac{6d}{3d-4}}
\\ &\lesssim_u \eta N^{-2s_c}(1+N^{4s_c-1}K).
\end{align*}
Thus this term is acceptable as well. Before proceeding, however, we note that it is this term that has forced us to exclude the cases $(d,s_c)\in\{3\}\times(\tfrac34,1)$ from this paper; we postpone further discussion until Remark~\ref{exclusion} below. 

We next turn to \eqref{error 4}; using H\"older, Bernstein, the fractional chain rule, Sobolev embedding, \eqref{uhi control}, and \eqref{ulo control}, we see
\begin{align*}
\xonorms{P_{\geq N}&(\vert u_{\leq N}\vert^p u_{\leq N})u_{\geq N}}{1}
\\ &\lesssim N^{-s_c} \xonorm{u_{\geq N}}{4}{\frac{2d}{d-1}}\xonorm{\nsc(\vert u_{\leq N}\vert^p u_{\leq N})}{\frac{4}{3}}{\frac{2d}{d+1}}
\\ &\lesssim N^{-s_c}\xonorm{u_{\geq N}}{4}{\frac{2d}{d-1}}\xonorm{u_{\leq N}}{4p}{\frac{2dp}{3}}^p\xonorm{\nsc u_{\leq N}}{2}{\frac{2d}{d-2}}
\\ &\lesssim N^{-s_c}\xonorm{u_{\geq N}}{4}{\frac{2d}{d-1}}\xonorm{\nsc u_{\leq N}}{4p}{\frac{2dp}{dp-1}}^p\xonorm{\nsc u_{\leq N}}{2}{\frac{2d}{d-2}}
\\ &\lesssim_u \eta^{p+1}N^{-2s_c}(1+N^{4s_c-1}K),
\end{align*} so that \eqref{error 4} is also acceptable.

For the final term originating from the mass bracket, \eqref{error 5}, we use H\"older, Bernstein, Sobolev embedding, \eqref{uhi control}, and \eqref{uhi small} to see
\begin{align*}
\xonorms{P_{\leq N}(\vert u_{\geq N}\vert^p u_{\geq N})u_{\geq N}}{1}&\lesssim \xonorm{u_{\geq N}}{3}{\frac{6d}{3d-4}}\xonorm{P_{\leq N}(\vert u_{\geq N}\vert^p u_{\geq N})}{\frac{3}{2}}{\frac{6d}{3d+4}}
\\ &\lesssim  N^{s_c}\xonorm{u_{\geq N}}{3}{\frac{6d}{3d-4}}\xonorm{\vert u_{\geq N}\vert^p u_{\geq N}}{\frac{3}{2}}{\frac{3dp}{3dp+2p-6}}
\\ &\lesssim  N^{s_c}\xonorm{u_{\geq N}}{3}{\frac{6d}{3d-4}} \xonorm{u_{\geq N}}{3}{\frac{6d}{3d-4}}^2\xonorm{u_{\geq N}}{\infty}{\frac{dp}{2}}^{p-1}
\\ &\lesssim N^{s_c} \xonorm{u_{\geq N}}{3}{\frac{6d}{3d-4}}^3\xonorm{\nsc u}{\infty}{2}^{p-1}
\\ &\lesssim_u N^{-2s_c}(1+N^{4s_c-1}K),
\end{align*}
which shows that \eqref{error 5} is acceptable. 

We now turn to the terms originating from the momentum bracket. First, consider \eqref{error 6}. By H\"older, Hardy--Littlewood--Sobolev, Sobolev embedding, Bernstein, \eqref{uhi control}, \eqref{ulo control}, and \eqref{uhi small}, we can estimate
\begin{align*} \eqref{error 6} &\lesssim \xonorm{\tfrac{1}{\vert x\vert}\ast\vert u_{\geq N}\vert^2}{3}{3d}\xonorm{u_{\geq N}}{3}{\frac{6d}{3d-4}}
\\ &\quad \times\xonorm{u_{\leq N}}{6}{\frac{3dp}{6-p}}\xonorm{u_{\leq N}}{6}{\frac{6d}{3d-8}}\xonorm{u}{\infty}{\frac{dp}{2}}^{p-1}
\\ &\lesssim \xonorm{\vert u_{\geq N}\vert^2}{3}{\frac{3d}{3d-2}}\xonorm{u_{\geq N}}{3}{\frac{6d}{3d-4}}
\\ &\quad \times\xonorm{\nsc u_{\leq N}}{6}{\frac{6d}{3d-2}}\xonorm{\nabla u_{\leq N}}{6}{\frac{6d}{3d-2}}\xonorm{\nsc u}{\infty}{2}^{p-1}
\\ &\lesssim_u N^{1-s_c}\xonorm{u_{\geq N}}{\infty}{2}\xonorm{u_{\geq N}}{3}{\frac{6d}{3d-4}}^2\xonorm{\nsc u_{\leq N}}{6}{\frac{6d}{3d-2}}^2
\\ &\lesssim_u \eta^{12} N^{1-4s_c}(1+N^{4s_c-1}K),
\end{align*} so that \eqref{error 6} is acceptable.

For \eqref{error 7}, we consider two cases. If $\vert u_{\leq N}\vert\leq 10^{-100}\vert u_{\geq N}\vert,$ then we can absorb this term into the left-hand side of the inequality, provided we can show
\begin{equation}\label{finiteness of bootstrap terms} \iiint\frac{\vert u_{\geq N}(t,y)\vert^2\vert u_{\geq N}(t,x)\vert^{p+2}}{\vert x-y\vert}\, dx\, dy\, dt<\infty.\end{equation}
On the other hand, if $\vert u_{\geq N}\vert\leq 10^{100}\vert u_{\leq N}\vert$, then we are back in the situation of \eqref{error 6}, which we have already handled. Thus, to render \eqref{error 7} acceptable, it remains to prove \eqref{finiteness of bootstrap terms}. To this end, we define 
$$\theta=\tfrac{4dp-16-3p}{2(dp-4)}\in(0,p+2),$$
and use H\"older, Hardy--Littlewood--Sobolev, Sobolev embedding, Lemma \ref{spacetime bounds}, and interpolation to estimate
\begin{align*} \text{LHS}\eqref{finiteness of bootstrap terms}&\lesssim \xonorm{\tfrac{1}{\vert x\vert}\ast\vert u_{\geq N}\vert^2}{4}{4d}\xonorm{u_{\geq N}}{\infty}{2}^\theta\xonorm{u_{\geq N}}{\frac{2p(2dp-5)}{3(dp-4)}}{\frac{dp(2dp-5)}{dp+2}}^{p+2-\theta}
\\ &\lesssim \xonorm{\vert u_{\geq N}\vert^2}{4}{\frac{4d}{4d-3}}\xonorm{u_{\geq N}}{\infty}{2}^{\theta}\xonorm{\nsc u_{\geq N}}{\frac{2p(2dp-5)}{3(dp-4)}}{\frac{2dp(2dp-5)}{2(dp)^2-11dp+24}}^{p+2-\theta}
\\ &\lesssim_u \xonorm{u_{\geq N}}{4}{\frac{4d}{2d-3}}\xonorm{u_{\geq N}}{\infty}{2}^{1+\theta}\left(1+\smallint_I N(t)^2\, dt\right)^{\frac{(p+2-\theta)(3(dp-4))}{2p(2dp-5)}}
\\ &\lesssim_u \xonorm{\vert\nabla\vert^{1/4}u_{\geq N}}{4}{\frac{2d}{d-1}} N^{-s_c(1+\theta)}\left(1+\smallint_I N(t)^2\, dt\right)^{\frac{(p+2-\theta)(3(dp-4))}{2p(2dp-5)}}
\\ &\lesssim_u N^{1/4-s_c(2+\theta)}\xonorm{\nsc u_{\geq N}}{4}{\frac{2d}{d-1}}\left(1+\smallint_I N(t)^2\, dt\right)^{\frac{(p+2-\theta)(3(dp-4))}{2p(2dp-5)}}
\\ &\lesssim_u N^{1-4s_c}\left(1+\smallint_I N(t)^2\,dt\right)^{\frac{1}{4}+\frac{(p+2-\theta)(3(dp-4))}{2p(2dp-5)}}
\\ &\lesssim_u N^{1-4s_c}\left(1+\smallint_I N(t)^2\,dt\right),
\end{align*} 
which gives \eqref{finiteness of bootstrap terms}, and thereby shows that \eqref{error 7} is acceptable.

Next, we turn to \eqref{error 8}. Denoting $$G=\vert u\vert^p u -\vert u_{\leq N}\vert^p u_{\leq N}-\vert u_{\geq N}\vert^p u_{\geq N},$$ we begin by writing
\begin{align}&\eqref{error 8}\nonumber
\\ &\quad= \iiint \frac{\vert u_{\geq N}(t,y)\vert^2\vert P_{\leq N}(\vert u_{\leq N}\vert^p u_{\leq N})(t,x)\vert\,\vert u_{\geq N}(t,x)\vert}{\vert x-y\vert}\,dx\,dy\,dt \label{8a}
\\ &\quad\quad+\iiint \frac{\vert u_{\geq N}(t,y)\vert^2\vert P_{\leq N}(\vert u_{\geq N}\vert^p u_{\geq N})(t,x)\vert\,\vert u_{\geq N}(t,x)\vert}{\vert x-y\vert} \,dx\,dy\,dt\label{8b}
\\ &\quad\quad+\iiint \frac{\vert u_{\geq N}(t,\!y)\vert^2\vert P_{\leq N}G(t,x)\vert\,\vert u_{\geq N}(t,x)\vert}{\vert x-y\vert}\,dx\,dy\,dt.\label{8c}
\end{align}
For \eqref{8a}, we can write
\begin{align*} \eqref{8a}&\lesssim \xonorm{\tfrac{1}{\vert x\vert}\ast \vert u_{\geq N}\vert^2}{3}{3d}\xonorm{u_{\geq N}}{3}{\frac{6d}{3d-4}}\xonorm{u_{\leq N}}{6}{\frac{3dp}{6-p}}\xonorm{u_{\leq N}}{6}{\frac{6d}{3d-8}}\xonorm{u}{\infty}{\frac{dp}{2}}^{p-1}
\\ &\lesssim_u \eta^{12} N^{1-4s_c}(1+N^{4s_c-1}K)
\end{align*} by the same arguments that dealt with \eqref{error 6}. 

For \eqref{8b}, we can use H\"older, Hardy--Littlewood--Sobolev, Bernstein, Sobolev embedding, \eqref{uhi control}, \eqref{ulo control}, and \eqref{uhi small} to estimate
\begin{align*}
\eqref{8b}&\lesssim \xonorm{ u_{\geq N}}{6}{\frac{6d}{3d-2}}^2\xonorm{\tfrac{1}{\vert x\vert}\ast (P_{\leq N}(\vert u_{\geq N}\vert^p u_{\geq N})u_{\geq N})}{\frac{3}{2}}{\frac{3d}{2}}
\\ &\lesssim  \xonorm{ u_{\geq N}}{6}{\frac{6d}{3d-2}}^2\xonorm{P_{\leq N}(\vert u_{\geq N}\vert^p u_{\geq N})u_{\geq N}}{\frac{3}{2}}{\frac{3d}{3d-1}}
\\ &\lesssim  \xonorm{ u_{\geq N}}{6}{\frac{6d}{3d-2}}^2\xonorm{u_{\geq N}}{\infty}{2}\xonorm{P_{\leq N}(\vert u_{\geq N}\vert^p u_{\geq N})}{\frac{3}{2}}{\frac{6d}{3d-2}}
\\ &\lesssim \xonorm{ u_{\geq N}}{6}{\frac{6d}{3d-2}}^2\xonorm{u_{\geq N}}{\infty}{2}N^{1+s_c}\xonorm{P_{\leq N}(\vert u_{\geq N}\vert^p u_{\geq N})}{\frac32}{\frac{3dp}{3dp+2p-6}}
\\ &\lesssim  \xonorm{ u_{\geq N}}{6}{\frac{6d}{3d-2}}^2\xonorm{u_{\geq N}}{\infty}{2}N^{1+s_c}\xonorm{u_{\geq N}}{3}{\frac{6d}{3d-4}}^2\xonorm{u}{\infty}{\frac{dp}{2}}^{p-1}
\\ &\lesssim_u \eta^{10}N^{1-4s_c}(1+N^{4s_c-1}K),
\end{align*} which renders \eqref{8b} acceptable. 

For \eqref{8c}, we first note $$\vert u\vert^p u-\vert u_{\leq N}\vert^p u_{\leq N}-\vert u_{\geq N}\vert^p u_{\geq N}=\text{\O}(u_{\geq N}u_{\leq N}\vert u\vert^{p-1}),$$ so that using H\"older, Hardy--Littlewood--Sobolev, Bernstein, Sobolev embedding, \eqref{uhi control}, \eqref{ulo control}, and \eqref{uhi small}, we get 
\begin{align*}
\eqref{8c}&\lesssim \xonorm{\tfrac{1}{\vert x\vert}\ast \vert u_{\geq N}\vert^2}{3}{3d}\xonorm{P_{\leq N}(\text{\O}(u_{\geq N}u_{\leq N}\vert u\vert^{p-1}))}{3}{\frac{6d}{3d+2}}\xonorm{u_{\geq N}}{3}{\frac{6d}{3d-4}}
\\ &\lesssim N\xonorm{\vert u_{\geq N}\vert^2}{3}{\frac{3d}{3d-2}}\xonorm{\text{\O}(u_{\geq N}u_{\leq N}\vert u\vert^{p-1})}{3}{\frac{6d}{3d+8}}\xonorm{u_{\geq N}}{3}{\frac{6d}{3d-4}}
\\ &\lesssim N\xonorm{u_{\geq N}}{3}{\frac{6d}{3d-4}}^2\xonorm{u_{\geq N}}{\infty}{2}^2\xonorm{u_{\leq N}}{3}{\frac{3dp}{6-2p}}\xonorm{u}{\infty}{\frac{dp}{2}}^{p-1}
\\ &\lesssim N\xonorm{u_{\geq N}}{3}{\frac{6d}{3d-4}}^2\xonorm{u_{\geq N}}{\infty}{2}^2\xonorm{\nsc u_{\leq N}}{3}{\frac{6d}{3d-4}}\xonorm{\nsc u}{\infty}{2}^{p-1}
\\ &\lesssim_u \eta^{21}N^{1-4s_c}(1+N^{4s_c-1}K).
\end{align*} 
Thus \eqref{8c}, and so \eqref{error 8}, is acceptable.

We now turn to \eqref{error 9}. By H\"older, Sobolev embedding, Bernstein, \eqref{ulo control}, and \eqref{uhi small}, we estimate
\begin{align*}
\eqref{error 9}&\lesssim \xonorm{u_{\geq N}}{\infty}{2}^3\xonorm{u_{\leq N}}{2p}{dp}^p\xonorm{\nabla u_{\leq N}}{2}{\frac{2d}{d-2}}
\\ &\lesssim  N^{1-s_c}\xonorm{u_{\geq N}}{\infty}{2}^3\xonorm{\nsc u_{\leq N}}{2p}{\frac{2dp}{dp-2}}^p\xonorm{\nsc u_{\leq N}}{2}{\frac{2d}{d-2}}
\\ &\lesssim_u \eta^{p+31}N^{1-4s_c}(1+N^{4s_c-1}K),
\end{align*}
so that \eqref{error 9} is acceptable. 

For \eqref{error 10}, we use H\"older, Sobolev embedding, Bernstein, \eqref{uhi control}, \eqref{ulo control}, and \eqref{uhi small} to get
\begin{align*}
\eqref{error 10}&\lesssim \xonorm{u_{\geq N}}{\infty}{2}^2\xonorm{u_{\geq N}}{\infty}{\frac{dp}{2}}^{p-1}\xonorm{u_{\geq N}}{3}{\frac{6d}{3d-4}}^2\xonorm{\nabla u_{\leq N}}{3}{\frac{3dp}{6-2p}}
\\ &\lesssim_u N\xonorm{u_{\geq N}}{\infty}{2}^2\xonorm{u_{\geq N}}{3}{\frac{6d}{3d-4}}^2\xonorm{\nsc u_{\leq N}}{3}{\frac{6d}{3d-4}}
\\ &\lesssim_u \eta^{21}N^{1-4s_c}(1+N^{4s_c-1}K),
\end{align*}
which renders \eqref{error 10} acceptable.

Finally, we consider \eqref{error 11}. We begin by writing
\begin{align}
\eqref{error 11}\lesssim &\,\xonorm{u_{\geq N}}{\infty}{2}^2\xonorms{u_{\geq N}\nabla P_{\leq N}(\vert u_{\leq N}\vert^p u_{\leq N})}1 \label{11a}
\\ +& \xonorm{u_{\geq N}}{\infty}{2}^2\xonorms{u_{\geq N}\nabla P_{\leq N}(\vert u_{\geq N}\vert^p u_{\geq N})}{1} \label{11b}
\\ +& \xonorm{u_{\geq N}}{\infty}{2}^2\xonorms{u_{\geq N}\nabla P_{\leq N}(\vert u\vert^p u-\vert u_{\leq N}\vert^p u_{\leq N}-\vert u_{\geq N}\vert^p u_{\geq N})}{1}\label{11c}.
\end{align}
To begin, we use H\"older, the chain rule, and the arguments that gave \eqref{error 9} to see
\begin{align*}
\eqref{11a}&\lesssim \xonorm{u_{\geq N}}{\infty}{2}^3\xonorm{u_{\leq N}}{2p}{dp}^p\xonorm{\nabla u_{\leq N}}{2}{\frac{2d}{d-2}}
\\ &\lesssim_u \eta^{p+31}N^{1-4s_c}(1+N^{4s_c-1}K),
\end{align*}
so that \eqref{11a} is acceptable.

For \eqref{11b}, we argue essentially as we did for \eqref{error 5}. That is, we use H\"older, Bernstein, Sobolev embedding, \eqref{uhi control}, and \eqref{uhi small} to estimate
\begin{align*}
\eqref{11b}&\lesssim \xonorm{u_{\geq N}}{\infty}{2}^2\xonorm{u_{\geq N}}{3}{\frac{6d}{3d-4}}\xonorm{\nabla P_{\leq N}(\vert u_{\geq N}\vert^pu_{\geq N})}{\frac{3}{2}}{\frac{6d}{3d+4}}
\\ &\lesssim N^{1+s_c}\xonorm{u_{\geq N}}{\infty}{2}^2\xonorm{u_{\geq N}}{3}{\frac{6d}{3d-4}}\xonorm{\vert u_{\geq N}\vert^p u_{\geq N}}{\frac{3}{2}}{\frac{3dp}{3dp+2p-6}}
\\ &\lesssim N^{1+s_c}\xonorm{u_{\geq N}}{\infty}{2}^2\xonorm{u_{\geq N}}{3}{\frac{6d}{3d-4}}^3\xonorm{\nsc u}{\infty}{2}^{p-1}
\\ &\lesssim_u \eta^{20} N^{1-4s_c}(1+N^{4s_c-1}K),
\end{align*}
which gives that \eqref{11b} is acceptable.

For \eqref{11c}, we argue similarly to the case of \eqref{8c}. In particular, we use H\"older, Bernstein, Sobolev embedding, \eqref{uhi control}, \eqref{ulo control}, and \eqref{uhi small} to see
\begin{align*}
\eqref{11c}&\lesssim N\xonorm{u_{\geq N}}{\infty}{2}^2\xonorm{u_{\geq N}}{3}{\frac{6d}{3d-4}}\xonorm{\text{\O}(u_{\leq N}u_{\geq N}\vert u\vert^{p-1})}{\frac{3}{2}}{\frac{6d}{3d+4}}
\\ &\lesssim N\xonorm{u_{\geq N}}{\infty}{2}^2\xonorm{u_{\geq N}}{3}{\frac{6d}{3d-4}}^2\xonorm{u_{\leq N}}{3}{\frac{3dp}{6-2p}}\xonorm{u}{\infty}{\frac{dp}{2}}^{p-1}
\\ &\lesssim N\xonorm{u_{\geq N}}{\infty}{2}^2\xonorm{u_{\geq N}}{3}{\frac{6d}{3d-4}}^2\xonorm{\nsc u_{\leq N}}{3}{\frac{6d}{3d-4}}\xonorm{\nsc u}{\infty}{2}^{p-1}
\\ &\lesssim_u \eta^{21} N^{1-4s_c}(1+N^{4s_c-1}K),
\end{align*}
which gives that \eqref{11c}. Collecting the estimates for \eqref{11a}, \eqref{11b}, and \eqref{11c}, we see that \eqref{error 11} is acceptable. This completes the proof of Proposition \ref{flim}.
\end{proof}
\begin{remark}\label{exclusion} Let us discuss why \eqref{error 3} has forced us to exclude the cases $(d,s_c)\in\{3\}\times(\tfrac34,1)$ from this paper. As one can see in the proof above, in the cases we consider, this term is fairly harmless. However, once $s_c>\tfrac34$ in dimension $d=3$ (which corresponds to $p>\tfrac{8}{3}$), this term becomes a problem; put simply, we end up with too many copies of $u_{\geq N}$ to deal with.

This problem has already been encountered in the energy-critical setting ($s_c=1$) in dimension $d=3$; in this case, one can overcome the hurdle by applying a spatial truncation to the weight $a$. One can refer to \cite{CKSTT:gwp} for the original argument, wherein spatial truncation is applied at various levels and subsequently averaged. The authors of \cite{revisit2} revisit the result of \cite{CKSTT:gwp} in the context of minimal counterexamples; at this point in the argument, they choose to work with a more carefully designed spatial truncation, which removes the need for any subsequent averaging argument.

This discussion begs the question: why doesn't spatial truncation work in our setting? To answer this, we need to understand how spatial truncations affect the argument that leads to Proposition \ref{flim}. What we find is that spatial truncations ruin the convexity properties of $a$ that made some of the terms in the proof of Lemma \ref{interaction morawetz} positive; thus, to establish Proposition \ref{flim} with a further spatial truncation, we have to control additional error terms. It turns out that one of these additional error terms requires uniform control over $\norm{u}_{L_x^{p+2}}$, while another requires uniform control over $\norm{\nabla u}_{L_x^2}$ (see \cite[Lemma 6.5 and Lemma 6.6]{revisit2}). In the energy-critical case, one can use the conservation of energy to push the argument through, while in our cases, we cannot proceed without some significant new input. We have therefore abandoned the cases $(d,s_c)\in\{3\}\times(\tfrac34,1)$ in this paper. 

For a further discussion of these issues, refer to \cite{revisit2}, especially Remark~6.9 therein.
\end{remark}
\section{The quasi-soliton scenario}\label{Quasi}
In this section we preclude the existence of almost periodic solutions as in Theorem \ref{special scenarios} for which $\int_0^{T_{max}}N(t)^{3-4s_c}\,dt=\infty.$ We will show that their existence is inconsistent with the frequency-localized interaction Morawetz inequality (Proposition~\ref{flim}). 

Before we begin, we note that
$$-\Delta(\tfrac{1}{\vert x\vert})=\left\{\begin{array}{ll}4\pi\delta & d=3 \\ \tfrac{d-3}{\vert x\vert^3} & d\geq 4.\end{array}\right.$$
Thus, if $d=3$, the conclusion of Proposition \ref{flim} reads
\begin{equation}\label{flim 3D} \int_I\int_{\R^3}\vert u_{\geq N}(t,x)\vert^4\, dx\, dt\lesssim_u\eta(N^{1-4s_c}+K)
\end{equation}
for $N\leq N_0(\eta)$, while if $d\in\{4,5\}$, the conclusion of Proposition \ref{flim} reads
\begin{equation}\label{flim 45D} \int_I\iint_{\R^d\times\R^d}\frac{\vert u_{\geq N}(t,x)\vert^2\vert u_{\geq N}(t,y)\vert^2}{\vert x-y\vert^3}\, dx\, dy\, dt\lesssim_u\eta(N^{1-4s_c}+K)
\end{equation}
for $N\leq N_0(\eta)$.

We now turn to
\begin{theorem}[No quasi-solitons]\label{no quasi-solitons} Let $(d,s_c)$ satisfy \eqref{constraints}. Then there are no almost periodic solutions $u:[0,T_{max})\times\R^d\to\C$ to \eqref{nls} with $N(t)\equiv N_k\geq 1$ on each characteristic subinterval $J_k\subset[0,T_{max})$ that satisfy both $$\xnorms{u}{\frac{p(d+2)}{2}}{[0,T_{max})}=\infty$$ and 
\begin{equation}\label{quasi-soliton} \int_0^{T_{max}}N(t)^{3-4s_c}\, dt=\infty.\end{equation}
\end{theorem}
\begin{proof} Suppose towards a contradiction that such a solution $u$ exists. We will need the following 
\begin{lemma}\label{almost periodic stuff} There exist $N_0>0$ and $C(u)>0$ so that
				\begin{equation}\label{quasi-soliton lower boundd} 
				\inf_{t\in[0,T_{max})}N(t)^{2s_c}\int_{\vert x-x(t)\vert\leq \frac{C(u)}{N(t)}}\vert u_{\geq N}(t,x)\vert^2\, dx
				\gtrsim_u 1\indent\text{for all }N\leq N_0.
				\end{equation}
\end{lemma}
We provide the proof of Lemma \ref{almost periodic stuff} below; let us first take it for granted and use it to complete the proof of Theorem \ref{no quasi-solitons}. We let $I$ be a compact time interval, which is a union of contiguous subintervals $J_k$, and let $\eta>0$ be a small parameter. We take $C(u)$ and $N_0$ as in Lemma \ref{almost periodic stuff}; then, choosing $N_0$ possibly smaller, we can guarantee by Proposition~\ref{flim} that \eqref{flim 3D} or \eqref{flim 45D} holds (depending on the dimension) for all $N\leq N_0$. 

We now consider two cases:

\textbf{Case 1.} When $d=3$, we first note by H\"older's inequality that 
\begin{align*}\big(\smallint_{\vert x-x(t)\vert\leq \frac{C(u)}{N(t)}}&\vert u_{\geq N}(t,x)\vert^2\, dx\big)^2\lesssim_u N(t)^{-3}\smallint_{\vert x-x(t)\vert\leq \frac{C(u)}{N(t)}}\vert u_{\geq N}(t,x)\vert^4\, dx.
\end{align*}
Using this inequality, followed by \eqref{quasi-soliton lower boundd}, we find 
\begin{align} \int_I \int_{\vert x-x(t)\vert\leq \frac{C(u)}{N(t)}}&\vert u_{\geq N}(t,x)\vert^4\, dx\, dt \nonumber
\\ & \gtrsim_u \int_I \left(\int_{\vert x-x(t)\vert\leq \frac{C(u)}{N(t)}}\vert u_{\geq N}(t,x)\vert^2\, dx\right)^2 N(t)^3\, dt \nonumber
\\ &\gtrsim_u \int_I N(t)^{3-4s_c}\, dt \nonumber
\end{align} for all $N\leq N_0$. 
Thus, appealing to \eqref{flim 3D}, we find 
\begin{align} \eta\left[N^{1-4s_c}+\int_I N(t)^{3-4s_c}\right] & \gtrsim_u \int_I\int_{\R^3} \vert u_{\geq N}(t,x)\vert^4\, dx\, dt \nonumber
\\ &\gtrsim_u \int_I \int_{\vert x-x(t)\vert\leq \frac{C(u)}{N(t)}}\vert u_{\geq N}(t,x)\vert^4\, dx\, dt \nonumber
\\ &\gtrsim_u \int_I N(t)^{3-4s_c}\, dt \label{3d quasi}
\end{align}
for all $N\leq N_0$.

\textbf{Case 2.} If $d\in\{4,5\}$, we once again use \eqref{quasi-soliton lower boundd}, but use \eqref{flim 45D} instead of \eqref{flim 3D}. We find
\begin{align} \eta\bigg[N^{1-4s_c}&+\int_I N(t)^{3-4s_c}\,dt\bigg] \nonumber
\\ & \gtrsim_u \int_I\int_{\R^d}\int_{\R^d}\frac{\vert u_{\geq N}(t,x)\vert^2\vert u_{\geq N}(t,y)\vert^2}{\vert x-y\vert^3}\, dx\, dy\, dt \nonumber
\\ & \gtrsim_u \int_I\iint_{\vert x-y\vert\leq \frac{2C(u)}{N(t)}} \big[\tfrac{N(t)}{2C(u)}\big]^3\vert u_{\geq N}(t,x)\vert^2\vert u_{\geq N}(t,y)\vert^2\, dx\, dy\, dt \nonumber
\\ &\gtrsim_u \int_I \big[\tfrac{N(t)}{C(u)}\big]^3\left(\int_{\vert x-x(t)\vert\leq \frac{C(u)}{N(t)}}\vert u_{\geq N}(t,x)\vert^2\, dx\right)^2\,dt \nonumber
\\ &\gtrsim_u \int_I N(t)^{3-4s_c}\, dt \label{45d quasi}
\end{align}
for all $N\leq N_0$. Thus, continuing from \eqref{3d quasi} or \eqref{45d quasi}, we see that in either case, for $\eta$ sufficiently small depending on $u$, we get $$\int_I N(t)^{3-4s_c}\, dt\lesssim_u N^{1-4s_c}\indent\text{for all } I\subset[0,T_{max})\text{ and all }N\leq N_0.$$ 
Recalling \eqref{quasi-soliton}, we now reach a contradiction by taking $I$ sufficiently large inside $[0,T_{max}).$ 
\end{proof}
Finally, we turn to the 
\begin{proof}[Proof of Lemma \ref{almost periodic stuff}] Let $\eta_0>0$ be a small parameter to be determined later. As $\inf_{t\in[0,T_{max})}N(t)\geq 1$, for $N_0=N_0(\eta_0)$ sufficiently small, we can guarantee 
				\begin{equation}
				\nonumber
				\xnorm{u_{\leq N}}{\infty}{\frac{dp}{2}}{[0,T_{max})}\leq\eta_0  \indent\text{for }N\leq N_0.
				\end{equation}
Then, given \emph{any} $C(u)>0$ and $N\leq N_0$, we can use H\"older's inequality and Sobolev embedding to estimate
			\begin{align*}
			\bigg\vert\int_{\vert x-x(t)\vert\leq\frac{C(u)}{N(t)}} \vert u_{>N}(t,x)\vert^2-\vert u(t,x)\vert^2\,dx\bigg\vert
			 &\lesssim_u N(t)^{-2s_c}\xonorm{u_{\leq N}}{\infty}{\frac{dp}{2}}\xonorm{u}{\infty}{\frac{dp}{2}}
			\\ &\lesssim_u \eta_0 N(t)^{-2s_c}
			\end{align*}
for all $t\in[0,T_{max}).$ Thus, if we can show that for $C(u)$ sufficiently large, we have
	\begin{equation}\label{the real deal}
	\inf_{t\in[0,T_{max})} N(t)^{2s_c}\int_{\vert x-x(t)\vert\leq\frac{C(u)}{N(t)}} \vert u(t,x)\vert^2\,dx\gtrsim_u 1,
	\end{equation}
we will have \eqref{quasi-soliton lower boundd} by choosing $\eta_0=\eta_0(u)$ sufficiently small.

We now turn to \eqref{the real deal}. We first choose $c=c(\eta_0)$ sufficiently large that
			\begin{equation}
			\label{so high so small}
			\xnorm{\nsc u_{>cN(t)}}{\infty}{2}{[0,T_{max})}<\eta_0.
			\end{equation}
We then notice that by H\"older, Bernstein, Sobolev embedding, and \eqref{so high so small}, we have
	\begin{align}
	\bigg\vert\int_{\vert x-x(t)\vert\leq\frac{C(u)}{N(t)}}\vert &u(t,x)\vert^2-\vert u_{\leq cN(t)}(t,x)\vert^2\,dx\bigg\vert \nonumber
	\\&\lesssim N(t)^{-s_c}\norm{u_{>cN(t)}(t)}_{L_x^2(\R^d)}\norm{u(t)}_{L_x^{\frac{dp}{2}}(\R^d)}\nonumber
	\\ &\lesssim_u \eta_0 N(t)^{-2s_c}\label{excessive4}
	\end{align}
for all $t\in[0,T_{max})$. Thus, if we can show that for $C(u)$ sufficiently large, we have
	\begin{equation}\label{the real real deal}
	\inf_{t\in[0,T_{max})} N(t)^{2s_c}\int_{\vert x-x(t)\vert\leq\frac{C(u)}{N(t)}} \vert u_{\leq cN(t)}(t,x)\vert^{2}\,dx\gtrsim_u 1,
	\end{equation}
then \eqref{the real deal} will follow by taking $\eta_0=\eta_0(u)$ sufficiently small. 

Let us therefore turn to establishing \eqref{the real real deal}. We begin by choosing $C(u)$ sufficiently large that
		\begin{equation}\nonumber
		\inf_{t\in[0,T_{max})}\int_{\vert x-x(t)\vert\leq\frac{C(u)}{N(t)}}\vert u(t,x)\vert^{\frac{dp}{2}}\,dx\gtrsim_u 1
		\end{equation}
(cf. Remark~\ref{arzela ascoli}). Then, with $c=c(\eta_0)$ as above, we see by H\"older's inequality, Sobolev embedding, and \eqref{so high so small} that
	\begin{align*}
	\bigg\vert\int_{\vert x-x(t)\vert\leq\frac{C(u)}{N(t)}} \vert u(t,x)&\vert^{\frac{dp}{2}}-\vert u_{\leq cN(t)}(t,x)\vert^{\frac{dp}{2}}\,dx\bigg\vert 
	\\ & \lesssim \xonorm{u_{>cN(t)}}{\infty}{\frac{dp}{2}}\xonorm{u}{\infty}{\frac{dp}{2}}^{\frac{dp}{2}-1}
	\\ &\lesssim_u \eta_0
	\end{align*}
for all $t\in[0,T_{max}).$ Thus for $\eta_0=\eta_0(u)$ sufficiently small, we have
					\begin{equation}
					\inf_{t\in[0,T_{max})}\int_{\vert x-x(t)\vert\leq\frac{C(u)}{N(t)}}\vert u_{\leq cN(t)}(t,x)\vert^{\frac{dp}{2}}\,dx
					\gtrsim_u 1.			\label{excessive1}
					\end{equation}
Finally, by H\"older and Bernstein, we estimate
	\begin{align}
	\int_{\vert x-x(t)\vert\leq\frac{C(u)}{N(t)}}\vert &u_{\leq cN(t)}(t,x)\vert^{\frac{dp}{2}}\,dx\nonumber
	\\&\lesssim\norm{u_{\leq cN(t)}(t)}_{L_x^\infty(\R^d)}^{\frac{dp}{2}-2} \int_{\vert x-x(t)\vert\leq\frac{C(u)}{N(t)}}\vert u_{\leq cN(t)}(t,x)\vert^2\,dx\nonumber
	\\&\lesssim_u N(t)^{2s_c}\norm{u(t)}_{L_x^{\frac{dp}{2}}(\R^d)}^{\frac{dp}{2}-2}\int_{\vert x-x(t)\vert\leq\frac{C(u)}{N(t)}}\vert u_{\leq cN(t)}(t,x)\vert^2\,dx.\nonumber
	\\ &\lesssim_u N(t)^{2s_c}\int_{\vert x-x(t)\vert\leq\frac{C(u)}{N(t)}}\vert u_{\leq cN(t)}(t,x)\vert^2\,dx
		\label{excessive2}
	\end{align}
for all $t\in [0,T_{max})$. Combining \eqref{excessive1} and \eqref{excessive2} now yields \eqref{the real real deal}, which completes the proof of Lemma \ref{almost periodic stuff}.
\end{proof}
\appendix
\section{Some basic estimates}
We collect here some basic estimates that are useful in Section \ref{reduction section}. We begin with
\begin{lemma}\label{stupid lemma} Let $1<p\leq 2$. Then 
		\begin{align}
		\big\vert\,\vert a+c\vert^p-\vert a\vert^p-\vert b+c\vert^p+\vert b\vert^p\,\big\vert\lesssim \vert a-b\vert\,\vert c\vert^{p-1}
		\label{stupid stupid stupid}
		\end{align}
	for all $a,b,c\in\mathbb{C}$. 
	\end{lemma}
\begin{proof} Defining $G(z):=\vert z+c\vert^p-\vert z\vert^p$, the fundamental theorem of calculus gives
			\begin{align*}
			 \text{LHS}\eqref{stupid stupid stupid}=\bigg\vert(a-b)\int_0^1 G_z(b+\theta(a-b))\,d\theta 
			 				+\overline{(a-b)}\int_0^1G_{\bar{z}}(b+\theta(a-b))\,d\theta\bigg\vert.
			\end{align*}
Thus, to establish \eqref{stupid stupid stupid}, it will suffice to establish
			\begin{align}
			\vert G_{z}(z)\vert+\vert G_{\bar{z}}(z)\vert \lesssim\vert c\vert^{p-1}\quad\text{uniformly for }z\in\C. \nonumber
			\end{align}
That is, we need to show
			\begin{equation}		\nonumber
			\big\vert\, \vert z+c\vert^{p-2}(z+c)-\vert z\vert^{p-2}z\,\big\vert\lesssim \vert c\vert^{p-1}
			\end{equation}
uniformly in $z$. If $c=0$, this inequality is obvious. Otherwise, setting $z=c\,\zeta$ reduces the problem to showing
			\begin{equation}		\label{c thing}
			\big\vert\, \vert z+1\vert^{p-2}(z+1)-\vert z\vert^{p-2}z\,\big\vert\lesssim 1
			\end{equation}
uniformly in $z$. For $\vert z\vert\lesssim 1$, we immediately get \eqref{c thing} from the triangle inequality. For $\vert z\vert\gg1$, we can use the fundamental theorem of calculus and the fact that $p\leq2$ to see
			$$\big\vert\,\vert z+1\vert^{p-2}(z+1)-\vert z\vert^{p-2}z\,\big\vert\lesssim \vert z\vert^{p-2}\lesssim 1.$$
Thus, we see that \eqref{c thing} holds, which completes the proof of Lemma \ref{stupid lemma}. 
\end{proof}
Next, we record a few inequalities in the spirit of \cite[$\S$2.3]{taylor}.
\begin{lemma}\label{aux lemma}
Let $\M$ denote the Hardy--Littlewood maximal function, and let $\psih$ denote the convolution kernel of the Littlewood--Paley projection $P_1$. For a fixed function $f,$ $y\in\R^d,$ and $N\in 2^{\mathbb{Z}}$, we have
\begin{align}
\smallint_{\R^d} N^d\vert\psih(Ny)\vert\,\vert f(x-y)\vert\,dy&\lesssim \M(f)(x), \label{aux 1}
\\  \vert \delta_y f_N(x)\vert&\lesssim N\vert y\vert\left\{\M(f_N)(x)+\M(f_N)(x-y)\right\},  \label{aux 2}
\\ \smallint_{\R^d} N^d\vert y\vert\,\vert\psih(Ny)\vert\,\,dy&\lesssim\tfrac{1}{N}.  \label{aux 3}
\end{align} 
\end{lemma}
\begin{proof} We begin with \eqref{aux 1}. Note first that $$\eta:=N^d\vert\psih(Ny)\vert$$ is a spherically symmetric, decreasing function of radius; thus, we can write $$\eta(y)=\int_0^\infty\chi_{B(0,r)}(y)(-\eta'(r))\,dr,$$
where $\eta':=\tfrac{\partial\eta}{\partial r}$. 
We can then use the definition of the Hardy--Littlewood maximal function and integrate by parts to estimate
		\begin{align*}
		\text{LHS}\eqref{aux 1} &\lesssim\int_0^\infty\left( \int_{\vert y\vert\leq r}\vert f(x-y)\vert\,dy\right)(-\eta'(r))\,dr
		\\ &\lesssim \left(\int_0^\infty \eta(r)r^{d-1}\,dr\right)\M(f)(x)
		\\ &\lesssim_{\psi} \M(f)(x).
		\end{align*}

For \eqref{aux 2}, we begin by defining $\psi_0(\xi)=\psi(2\xi)+\psi(\xi)+\psi(\xi/2)$, the `fattened' Littlewood--Paley multiplier. Then we can write 
		\begin{equation} \label{aux aux}
		\vert\delta_y f_N(x)\vert=\big\vert\smallint [N^d\psith(N(z-y))-N^d\psith(Nz)]f_N(x-z)\,dz\big\vert.
		\end{equation}
If $N\vert y\vert\geq 1$, we can use the triangle inequality and argue as above to see that 
				$$\vert\delta_y f_N(x)\vert\leq \M(f_N)(x-y)+\M(f_N)(x),$$
giving \eqref{aux 2} in this case. If instead $N\vert y\vert\leq 1$, we can use the fact that $\psith$ is Schwartz to estimate 
			$$\vert \psith(N(z-y))-\psith(Nz)\vert\lesssim N\vert y\vert(1+N\vert z\vert)^{-100d}.$$ 
Then continuing from \eqref{aux aux} and once again arguing as for \eqref{aux 1}, we find
			$$\vert\delta_y f_N(x)\vert\lesssim N\vert y\vert M(f_N)(x),$$
which gives \eqref{aux 2} in this case.

Finally, we note that since $\psih$ is Schwartz, we have
		$$\int_{\R^d} N^d\vert Ny\vert\,\vert\psih(Ny)\vert\,dy\lesssim 1,$$
which immediately gives \eqref{aux 3}.
\end{proof}


\end{document}